\documentclass[a4,12pt]{amsart}
\usepackage{amssymb}
\usepackage{amstext}
\usepackage{amsmath}
\usepackage{amscd}
\usepackage{amsthm}
\usepackage{amsfonts}
\usepackage{enumerate}
\usepackage{graphicx}
\usepackage{latexsym}
\usepackage[all]{xy}
\usepackage[usenames]{color}
\usepackage{comment}
\usepackage[total={6.5in,8.75in},
top=1.2in, right=1.0in, left=1.0in, bottom=1.0in,includefoot]{geometry} %
\newtheorem{maintheorem}{Main Theorem}
\newtheorem*{corollary*}{Corollary}
\newtheorem*{question*}{Question}
\newtheorem{theorem}{Theorem}[section]
\newtheorem{corollary}[theorem]{Corollary}
\newtheorem{lemma}[theorem]{Lemma}
\newtheorem{proposition}[theorem]{Proposition}
\newtheorem*{proposition*}{Proposition}
\newtheorem{question}[theorem]{Question}

\newtheorem{claim}{Claim}
\newtheorem*{claim*}{Claim}
\newtheorem{assumption}[theorem]{Assumption}
\theoremstyle{definition}
\newtheorem{definition}[theorem]{Definition}
\newtheorem{remark}[theorem]{Remark}
\newtheorem{condition}{Condition}
\newtheorem{example}[theorem]{Example}

\newtheorem{definitiontheorem}[theorem]{Definition-Theorem}

\theoremstyle{remark}

\numberwithin{equation}{theorem}

\newenvironment{pfclaim}{

\begin{proof}
}
{
\end{proof}

}

\renewcommand{\mod}{\operatorname{mod}}
\newcommand{\proj}{\operatorname{proj}}

\newcommand{\End}{\operatorname{End}}
\newcommand{\Hom}{\operatorname{Hom}}
\newcommand{\add}{\operatorname{\mathsf{add}}}

\newcommand{\Kb}{\mathsf{K}^{\rm b}}
\newcommand{\m}{\mathsf{m}}
\newcommand{\T}{\mathcal{T}}
\newcommand{\U}{\mathcal{U}}
\newcommand{\Tree}{\mathbb{T}}

\renewcommand{\P}{\mathbb{P}}
\newcommand{\B}{\mathbb{B}}

\newcommand{\co}{\mathsf{co}}

\renewcommand{\H}{\mathcal{H}}

\newcommand{\tilt}{\operatorname{\mathsf{tilt}}}

\newcommand{\tsilt}{\operatorname{\mathsf{2silt}}}
\newcommand{\tpsilt}{\operatorname{\mathsf{2psilt}}}
\newcommand{\stilt}{\operatorname{\mathsf{s\text{-}tilt}}}
\newcommand{\ttilt}{\operatorname{\mathsf{\tau -tilt}}}
\newcommand{\sttilt}{\operatorname{\mathsf{s\tau -tilt}}}

\newcommand{\trigid}{\operatorname{\mathsf{\tau\text{-}rigid}}}
\newcommand{\trigidp}{\operatorname{\mathsf{\tau\text{-}rigidp}}}
\newcommand{\supp}{\operatorname{Supp}}
\newcommand{\stmin}{\operatorname{\mathsf{min}}}
\newcommand{\stmax}{\operatorname{\mathsf{max}}}
\newcommand{\source}{\operatorname{\mathsf{source}}}
\newcommand{\sink}{\operatorname{\mathsf{sink}}}
\newcommand{\Z}{{\mathbb{Z}}}
\newcommand{\N}{{\mathbb{N}}}

\newcommand{\calP}{{\mathcal{P}}}
\newcommand{\Tr}{\operatorname{Tr}}
\newcommand{\Aut}{\operatorname{Aut}}

\newcommand{\rad}{\operatorname{rad}}

\renewcommand{\top}{\operatorname{top}}

\newcommand{\soc}{\operatorname{soc}}
\renewcommand{\Im}{\operatorname{Im}}

\newcommand{\fac}{\operatorname{\mathsf{Fac}}}
\newcommand{\Sym}{\mathfrak S}
\newcommand{\ind}{\operatorname{ind}}

\newcommand{\dip}{\operatorname{dp}}
\newcommand{\dis}{\operatorname{ds}}

\newcommand{\surj}{\twoheadrightarrow}
\newcommand{\inj}{\hookrightarrow}
\begin{document}
\title[From support $\tau$-tilting posets to algebras] 
%Algebras whose support $\tau$-tilting posets are isomorphic to poset of symmetric groups with weak order]
{From support $\tau$-tilting posets to algebras}
%Algebras whose support $\tau$-tilting posets are isomorphic to symmetric groups}

%\date{\today}
%\author{Takuma Aihara}
%\address{Graduate School of Mathematics, Nagoya University, Furocho, %Chikusaku, Nagoya 464-8602, Japan}
%\email{aihara.takuma@math.nagoya-u.ac.jp}
\author{ Ryoichi Kase}
\address{Faculty of Informatics, Okayama University of Science,
1-1 Ridaicho, Kita-ku, Okayama-shi 700-0005, Japan}
\email{r-kase@mis.ous.ac.jp}
\urladdr{}
\thanks{2010 {\em Mathematics Subject Classification.} Primary 16G20; Secondary  06A06, 16D80.}
\keywords{representation of quivers, support $\tau$-tilting module, support $\tau$-tilting poset, silting complex.}
\thanks{The author was supported by JSPS Grant-in-Aid for Young Scientists (B) 17K14169.}

\begin{abstract}
The aim of this paper is to study a poset isomorphism between two support $\tau$-tilting posets.
We take several algebraic information from combinatorial properties of support $\tau$-tilting posets.
As an application, we treat a certain class of basic algebras which contains preprojective algebras of type $A$, Nakayama algebras,
and generalized Brauer tree algebras. We provide a necessary condition for that an algebra $\Lambda$ share the same support $\tau$-tilting poset
with a given algebra $\Gamma$ in this class. Furthermore, we see that this necessary condition is also a sufficient condition if  
$\Gamma$ is either a preprojective algebra of type $A$, a Nakayama algebra,
or a generalized Brauer tree algebra.
%Furthermore, we show that each finite connected $2$-regular poset is  realized as a support $\tau$-tilting poset. 
\end{abstract}

%% \tableofcontents %% Just for papers exceeding 50 pages.

\maketitle
\section{Introduction}\label{sec:intro}
Adachi-Iyama-Reiten introduced the notion of support $\tau$-tilting modules as a 
   generalization of tilting modules \cite{AIR}. 
   They give a mutation of support $\tau$-tilting modules and complemented
   that of tilting modules. i.e., the support $\tau$-tilting mutation has following nice properties:
   \begin{itemize}  
\item Support $\tau$-tilting mutation is always possible.
\item There is a partial order on the set of (isomorphism classes of) basic support $\tau$-tilting modules
such that its Hasse quiver realizes the support $\tau$-tilting mutation. (An analogue of Happel-Unger's result \cite{HU1} for tilting modules.)
\end{itemize}
 Moreover, they showed deep connections between $\tau$-tilting theory, silting theory, torsion theory and cluster tilting theory.
Further developments of these connections was given in \cite{Asa,KY}.    
Theory of ($\tau$-)tilting mutation also gives us interesting connections between representation theory of finite dimensional algebras
and combinatorics, for example \cite{IRRT, MRZ, M}.   
 
  \subsection*{Notation} Throughout this paper, let $\Lambda=KQ/I$ be a basic finite dimensional algebra over an algebraically
closed field $K$, where $Q$ is a finite quiver and $I$ an admissible ideal of $KQ$.

We denote by $Q_0$ the set of 
vertices  of $Q$ and $Q_1$ the set of arrows of $Q$. We set $Q^{\circ}$ the quiver obtained from $Q$ by 
deleting all loops.

\begin{enumerate}[1.]
\item For arrows $\alpha:a_0\to a_1$ and $\beta: b_0\to b_1$ of $Q$,
we mean by $\alpha\beta$ the path $a_0\xrightarrow{\alpha}a_1\xrightarrow{\beta}b_1$ if $a_1=b_0$, otherwise 0 in $KQ$.

\item We denote by $\mod \Lambda\ (\proj\Lambda)$ the category of finitely generated (projective) right $\Lambda$-modules.

\item By a module, we always mean a finitely generated right module.
\item The Auslander-Reiten translation is denoted by $\tau$. 
(Refer to \cite{ASS, ARS} for definition and properties.)

\item Let $\mathbb{P}=(\mathbb{P},\leq)$ be a poset. 
We denote by $\H(\mathbb{P})$ the Hasse quiver
of $\mathbb{P}$ and  set  $[a,b]:=\{x\in \mathbb{P}
\mid a\leq x\leq b\}$ for $a,b\in \mathbb{P}$. 
We denote by $\dip(a)$ the set of direct predecessors of $a$ in $\H(\mathbb{P})$
and by $\dis(a)$ the set of direct successors of $a$ in $\H(\mathbb{P})$.
We say that $\mathbb{P}$ is $n$-{\bf regular} provided $\#\dip(a)+\#\dis(a)=n$ holds
 for each element $a\in \mathbb{P}$. 
Let  $\mathbb{P'}$ be a subset of $\mathbb{P}$ and $\leq'$ the partial order on $\mathbb{P'}$ given by $\leq$. Then we call
$\mathbb{P'}=(\mathbb{P'},\leq')$ a {\bf full subposet}. Throughout this paper every subposets are full. We call
a full subposet $\mathbb{P'}$ 
a {\bf strongly full subposet} if the inclusion $\mathbb{P'}\subset \mathbb{P}$ induces
 a quiver inclusion from $\H(\mathbb{P'})$ to $\H(\mathbb{P})$. 
By definition if $\mathbb{P'}$ is a strongly full subposet of $\mathbb{P}$, then $\H(\mathbb{P'})$
is a full subquiver of $\H(\mathbb{P})$. 
\end{enumerate}
\subsection*{Aim of this paper}
In \cite{HU2}, Happel and Unger showed the following fascinating result.
\begin{theorem}[{\cite[Theorem\;6.4]{HU2}}]
	We can reconstruct a quiver $Q$ up to multiple arrows from the tilting poset of $KQ$.  
\end{theorem}
This theorem states that the tilting poset of a hereditary algebra $\Lambda$ contains  lots of information for $\Lambda$. 
Therefore, it is interesting to extent Happel-Unger's reconstruction theorem to arbitrary finite dimensional algebras, i.e., we consider the following question. 
\begin{question*}
To what extent can we reconstruct an algebra
 from their support $\tau$-tilting poset? 
\end{question*}

For a $\tau$-tilting finite algebra $\Lambda$, it was shown in \cite{IRRT}
 that there are bijections between isomorphism classes
of indecomposable $\tau$-rigid modules of $\Lambda$,
 join-irreducible elements in $\sttilt \Lambda$ and meet-irreducible
elements in $\sttilt \Lambda$. %As a slightly generalization of these bijections, 
We summarize these bijections and realize a basic $\tau$-rigid pair
 of $\Lambda$ as a full subquiver of $\sttilt \Lambda$ in two ways. By using
these realizations, we show the following result. 
\begin{maintheorem}
\label{maintheorem1} Let $\rho$ be a poset isomorphism $\sttilt \Lambda\stackrel{\sim}{\to} \sttilt \Gamma$.
 \begin{itemize}
 \item $\rho$ preserves supports of basic support $\tau$-tilting modules. 
 In particular, $\rho$ sends  basic $\tau$-tilting modules 
 of $\Lambda$ to basic $\tau$-tilting modules of $\Gamma$.
\item If $\sttilt \Lambda$ is a lattice, then $\rho$ induces a natural bijection
between isomorphism classes of basic $\tau$-rigid pair of $\Lambda$ and that of $\Gamma$.
\end{itemize}
\end{maintheorem}
We note that above result is a generalization of \cite[Theorem\;1.1]{K1}. In fact, if $\Lambda$
is hereditary, then  (support) $\tau$-tilting modules are (support) tilting modules.

It is well-known that each basic finite dimensional algebra is given by (a unique) quiver and 
relations (admissible ideal). 
\begin{maintheorem}
\label{mt3}
The support $\tau$-tilting poset of $\Lambda$ determines
 the quiver of $\Lambda$ up to multiple arrows and loops.
Furthermore, if $\Lambda=KQ/I$ is a $\tau$-tilting finite algebra, then 
$Q$ has no multiple arrows and the group of poset automorphisms of
 support $\tau$-tilting poset of $\Lambda$
is realized as a subgroup of the group of quiver automorphisms of
$Q\setminus\{\mathrm{loops}\}$.
\end{maintheorem}
 By using this result, we can recover Happel-Unger's reconstruction theorem.
 
Let $\Lambda$ and $\Gamma$ be two basic finite dimensional algebras.
If the posets of support $\tau$-tilting modules of $\Lambda$ and that of $\Gamma$ are isomorphic, then 
we denote $\Lambda\stackrel{\ttilt}{\sim} \Gamma$ and set 
\[\T(\Gamma):=\{\Lambda\mid \Lambda\stackrel{\ttilt}{\sim}\Gamma\}.\]
In \cite{EJR}, Eisele, Janssens and Raedschelders give us 
a sufficient condition for that two finite dimensional algebras share the same support $\tau$-tilting poset.
 By this result, we can see that there are infinitely many (non-isomorphic) basic finite dimensional 
algebras in $\T(\Gamma)$ for any $\Gamma$.
Therefore, it seems difficult to characterize algebras which are in $\T(\Gamma)$ for a given algebra $\Gamma$.
Successful examples are tree quiver algebras and the preprojective algebras of type $A$. 
\begin{theorem}[{\cite{AK,K2}}]
\label{refe}
Assume that $\Gamma=KQ'/I$ is either a tree quiver algebra or a preprojective algebra of type $A$.
Then $\Lambda\in \T(\Gamma)$ if and only if $\Lambda$ satisfies the following conditions.
\begin{enumerate}[{\rm (a)}]
\item There is a quiver isomorphism $\sigma:Q\setminus\{\mathrm{loops}\}\to Q'$ satisfying
$\supp e_{\sigma(i)} \Gamma=\sigma(\supp e_i \Lambda )$ for any $i\in Q_0$.
\item Each arrow $\alpha:i\to j$ $(i\neq j)$ satisfies $\alpha \Lambda e_j=e_i\Lambda e_j=e_i\Lambda \alpha$.
\end{enumerate}
\end{theorem}
%Now the following is an interesting question.
%\begin{question}
%\label{mainques}
%Why the conditions $(a)$ and $(b)$ apper for different classes of algebras?
%\end{question} 
To generalize above result, we consider a poset isomorphism between two support 
$\tau$-tilting posets and introduce a class $\Theta$ of basic algebras
containing tree quiver algebras, preprojective algebras of type $A$,
 Nakayama algebras and generalized Brauer tree algebras etc.
\begin{maintheorem}
\label{maintheorem2}
For a given algebra $\Gamma\in \Theta$, we get a necessary condition for that an algebra $\Lambda$ is in $\mathcal{T}(\Gamma)$.
Furthermore, this necessary condition is also a sufficient condition if $\T(\Gamma)$ contains either a tree quiver
algebra, a preprojective algebra of type $A$, a Nakayama algebra or a generalized
 Brauer tree algebra.
\end{maintheorem}
As an application, we can recover the following statements.
\begin{itemize} 
\item\cite[Theorem\;3.11]{A}
Let $\Lambda$ be a Nakayama algebra.
Assume that $\ell\ell(P_i)\geq n$ holds for each $i\in Q_0$. Then we have a poset isomorphism
\[\sttilt \Lambda\simeq \sttilt KC/R^n,\]
where $C$ is a cyclic quiver with $C_0:=\{1,\dots,n\}$ and $R=R_n:=\rad KC$.
\item\cite[Proposition\;4.7]{AAC} Let $\Lambda$ be a generalized Brauer tree algebra. Then $\sttilt \Lambda$ does not
depend on the multiplicity of the corresponding generalized Brauer tree.
\end{itemize}

%It is interesting to give a combinatorial characterlization of support $\tau$-tilting posets.
%$\tau$-tilting finiteness of $\Lambda$ implies that $\sttilt \Lambda$ is finite, connected and $|\Lambda|$-regular. The converse is not true. In fact, for each $3\leq n$, there exists a finite connected $n$-regular poset $\P$ which is 
%not isomorphic to each support $\tau$-tilting poset. However, if $n=2$, we obtain the following result.
%\begin{maintheorem}
%\label{maintheorem3}
%For each finite connected $2$-regular poset $\P$, there exists an algebra $\Lambda$ such that
%\[\sttilt \Lambda\simeq \P.\]
%\end{maintheorem}

\section{Fundamentals of support $\tau$-tilting posets}\label{sec:prelim}
In this section, we recall the definitions and their basic properties of 
support $\tau$-tilting posets.
For a module $M$, we denote by $|M|$ the number of non-isomorphic indecomposable direct summands of $M$
and by $\supp(M):=\{i\in Q_0\mid Me_i\neq 0\}$ the support of $M$, where $e_i$ is a primitive idempotent corresponding 
to a vertex $i\in Q_0$. We put $e_M:=\sum_{i\in \supp(M)} e_i$.

A module $M\in\mod \Lambda$ is said to be {\bf $\tau$-rigid} if it satisfies
$\Hom_{\Lambda}(M, \tau M)=0$.  
If $\tau$-rigid module $T$ satisfies 
$|T|=\#\supp(T)$ (resp. $|T|=n$), then we call $T$ a {\bf support $\tau$-tilting module}
 (resp. {\bf $\tau$-tilting module}).
We denote by $\sttilt \Lambda$ (resp. $\ttilt \Lambda$, $\trigid \Lambda$)
 the set of (isomorphism classes of) basic support $\tau$-tilting modules 
 (resp. $\tau$-tilting modules, $\tau$-rigid modules)
of $\Lambda$. 

We call a pair $(M,P)\in \mod \Lambda\times \proj \Lambda$ a {\bf $\tau$-rigid pair} (resp. {\bf $\tau$-tilting pair})
if $M$ is $\tau$-rigid (resp. support $\tau$-tilting) and $\add P\subset \add (1-e_M) \Lambda$ (resp. $\add P=\add (1-e_M)\Lambda$).

Let $(M,P)$ be a $\tau$-rigid pair.
We say that $(M,P)$ is {\bf basic} if so are $M$ and $P$.
A direct summand $(N, R)$ of $(M,P)$ is a pair of a module $N$ and a projective module $R$
 which are direct summands of $M$ and $P$, respectively.
 From now on, we put 
 \[M\oplus P^-:=(M,P) \text{ and } |M\oplus P^-|:=|M|+|P|.\]
\begin{remark}
If $M$ is $\tau$-rigid, then we have $|M|\leq \# \supp(M)$ (see \cite[Proposition 1.3]{AIR}).
In particular, a $\tau$-rigid pair $M\oplus P^-$ is $\tau$-tilting if and only if $|M\oplus P^-|=|\Lambda|$.
\end{remark}
We denote by $\trigidp \Lambda$ the set of (isomorphism classes of) basic $\tau$-rigid pairs of $\Lambda$.

%A support $\tau$-tilting pair $(M,P)$ is defined to be a $\tau$-rigid pair with $|M|+|P|=|\Lambda|$. A pair $(M,P)$ is said to be $\tau$-\emph{rigid} if $M$ is $\tau$-rigid and $\Hom_\Lambda(P,M)=0$.
%%%%%%%%%%%%%%%%%%%%%%%%%%%%%%%%%%%%%%%%%%%%%%%%%%%%%%%%%%%%%%%%%%%%%%%%%%%%%%%
%%%%%%%%%%%%%%%%%%%%%%%%%%%%%%%%%%%%%%%%%%%%%%%%%%%%%%%%%%%%%%%%%%%%%%%%%%%%%%%
\subsection{Basic properties}
In this subsection, we collect important properties of support $\tau$-tilting modules.
The following proposition gives us a connection between $\tau$-rigid modules of $\Lambda$
and that of a factor algebra of $\Lambda$.
\begin{proposition}[{\cite[Lemma 2.1]{AIR}}]
\label{basicfact} 
 Let $J$ be a two-sided ideal of $\Lambda$. Let $M$ and $N$ be $(\Lambda/J)$-modules.
If $\Hom_{\Lambda}(M,\tau N)=0$, then $\Hom_{\Lambda/J}(M,\tau_{\Lambda/J} N)=0$.
Moreover, if $J=(e)$ is an two-sided ideal generated by an idempotent $e$, then the converse holds.
\end{proposition}

Denote by $\fac M$ the category of factor modules of finite direct sums of copies of $M$.
Then the notion of support $\tau$-tilting posets is given by the following result.
\begin{definitiontheorem}[{\cite[Lemma 2.25]{AIR}}]
For support $\tau$-tilting modules $M$ and $M'$, we write $M\geq M' $ if $\fac M\supseteq \fac M'$.
 Then the following are equivalent.
\begin{enumerate}
\item $M\geq M'$.
\item  $\Hom_{\Lambda}(M',\tau M)=0$ and $\supp(M)\supseteq \supp(M')$.
\end{enumerate}
Moreover, $\geq$ gives a partial order on $\sttilt\Lambda$.
\end{definitiontheorem} 

Next we consider a relationship between
the support $\tau$-tilting poset of $\Lambda$ and that of $\Lambda^{\mathrm{op}}$.
\begin{proposition}[{\cite[Theorem\;2.14,\;Proposition\;2.27]{AIR}}]
\label{revers}
Let $M\oplus P^-=M_{\mathrm{np}}\oplus M_{\mathrm{pr}}\oplus P^-$ be a $\tau$-tilting pair with
$M_{\mathrm{pr}}$ being a maximal projective direct summand of $M$. 
We put $(M\oplus P^-)^{\dagger}:=\Tr M_{\mathrm{np}}\oplus P^*\oplus (M_{\mathrm{pr}}^*)^-$,
where $(-)^*=\Hom_{\Lambda}(-,\Lambda):\proj \Lambda\to \proj \Lambda^{\mathrm{op}}$.
Then $(M\oplus P^-)^{\dagger}$ is a $\tau$-tilting pair. Moreover, $(-)^{\dagger}$
gives a poset anti-isomorphism from $\sttilt \Lambda$ to $\sttilt \Lambda^{\mathrm{op}}$.
\end{proposition}

A $\tau$-rigid pair $X$ is said to be 
{\bf almost complete $\tau$-tilting} provided it satisfies $|X|=|\Lambda|-1$.
Then the mutation of support $\tau$-tilting modules is formulated by the following theorem.
\begin{theorem}\label{basicprop}
\begin{enumerate}[{\rm (1)}]
\item \cite[Theorem\;2.18]{AIR} 
Let $X$ be a basic almost complete $\tau$-tilting pair. Then there are exactly two basic support $\tau$-tilting modules
$T$ and $T'$ such that $X$ is a direct summand of $T\oplus (1-e_T) \Lambda^-$ and $T'\oplus (1-e_{T'})\Lambda^-$. 
\item \cite[Corollary\;2.34]{AIR} Let $T$ and $T'$ be basic support $\tau$-tilting modules.
Then $T$ and $T'$ are connected by an arrow of $\H(\sttilt\Lambda)$ if and only if $T\oplus (1-e_T)\Lambda^-$ and $T'\oplus (1-e_{T'})\Lambda^-$
 have a common basic almost complete $\tau$-tilting pair as a direct summand.
In particular, $\sttilt\Lambda$ is $|\Lambda|$-regular. 
\item\cite[Theorem\;2.35]{AIR} Let $T,T'\in\sttilt \Lambda$. If $T<T'$, 
then there is a direct predecessor $U$ of $T$ (resp. a direct successor $U'$ of $T'$) 
such that $U\leq T'$ (resp. $T\leq U'$). 
\item \cite[Corollary\;2.38]{AIR} If $\H(\sttilt\Lambda)$ has a finite connected component $\mathcal{C}$, then
$\mathcal{C}=\H(\sttilt\Lambda)$. 
\end{enumerate}   
\end{theorem}

For a basic $\tau$-rigid pair $N\oplus R^-$, we define
\[\sttilt_{N\oplus R^-}\Lambda:=\{T\in \sttilt\Lambda\mid N\in \add T,\;\Hom_{\Lambda}(R, T)=0\},\]
equivalently, which consists of all support $\tau$-tilting modules $T$ such that $T\oplus (1-e_T)\Lambda^-$ has $N\oplus R^-$ as a direct summand. 
For simplicity, we omit 0 if $N=0$ or $R=0$.
\begin{definitiontheorem}[{\cite[Theorem\;2.10]{AIR}}]
Let $X$ be a $\tau$-rigid pair. Then there is the maximum element of $\sttilt_{X}\Lambda$.
We call this maximum element the {\bf Bongartz completion} of $X$.
\end{definitiontheorem}
Given an idempotent $e=e_{i_1}+\cdots+e_{i_\ell}$ of $\Lambda$ so that $R=e\Lambda$,
we see that $M$ belongs to $\sttilt_{R^-}\Lambda$ if and only if it is a basic support $\tau$-tilting module
with $\supp(M)\subset Q_0\setminus\{i_1,\dots,i_\ell\}$ (or equivalently, $M$ is a $\Lambda/(e)$-module).
Hence, Proposition\;\ref{basicfact} leads to an equality $\sttilt_{R^-}\Lambda=\sttilt\Lambda/(e)$.
More generally, we have the following reduction theorem.
\begin{theorem}[{\cite{J}}]
\label{reduction theorem}
Let $X=N\oplus R^-$ be a basic  $\tau$-rigid pair and let $T$ be the Bongartz completion of $X$.
 If we set $\Gamma=\Gamma_X:=\End_{\Lambda}(T)/( e )$, then we have
  $|\Gamma|=|\Lambda|-|X|$ and 
  \[\sttilt_X \Lambda\simeq \sttilt\Gamma,\]
   where $e$ is the idempotent corresponding to the projective $\End_{\Lambda}(T)$-module $\Hom_{\Lambda}(T,N)$.
  \end{theorem}
 Theorem\;\ref{reduction theorem} implies that for an idempotent $e\in \Lambda$, we have a poset isomorphism
 \[\sttilt_{e\Lambda} \Lambda\simeq \sttilt \Lambda/(e).\] 
 In fact, the Bongarts completion of $e\Lambda$ is $\Lambda$ and $\Gamma_{e\Lambda}\cong \Lambda/( e )$.
%%%%%%%%%%%%%%%%%%%%%%%%%%%%%%%%%%%%%%%%%%%%%%%%%%%%%%%%%%%%%%%%%%%%%%%%%%%%%%%%%%%%%%%%%%%%%%
 \subsection{$\tau$-tilting finite algebras}
An algebra $\Lambda$ is said to be  {\bf $\tau$-tilting finite} if
one of the following equivalent conditions holds:
\begin{itemize}
\item $\# \sttilt \Lambda<\infty$.
\item $\# \ttilt \Lambda<\infty$.
\item $\#\trigid \Lambda<\infty$.
\end{itemize}

In \cite{DIJ}, $\tau$-tilting finite algebras are characterized via the torsion theory.   
A full subcategory $\mathcal{T}$ of $\mod \Lambda$ which is closed under factor modules and extensions
is called a {\bf torsion class} in $\mod \Lambda$. $\mathcal{T}$ is said to be {\bf functorally finite}
if for any $M\in \mod \Lambda$, there are $f\in \Hom_{\Lambda}(X,M)$ and $g\in \Hom_{\Lambda}(M,Y)$ with 
$X,Y\in \mathcal{T}$ such that $\Hom_{\Lambda}(N,f):\Hom_{\Lambda}(N,X)\to \Hom_{\Lambda}(N,M)$
and $\Hom_{\Lambda}(g,N):\Hom_{\Lambda}(Y,N)\to \Hom_{\Lambda}(M,N)$ are surjective for all $N\in \mathcal{T}$.

\begin{proposition}[{\cite[Proposition\;4.6]{AS}}]
\label{faceqfuncfinite}
An additive subcategory $\mathcal{T}$ of $\mod \Lambda$ is functorially finite if and only if there exists $M\in \mod \Lambda$ such that
$\mathcal{T}=\fac M$
\end{proposition}

\begin{theorem}[{\cite[Theorem\;3.8]{DIJ}}]
\label{tautiltfinitealgebra}
$\Lambda$ is a $\tau$-tilting finite algebra if and only if every torsion classes in $\mod \Lambda$
are functorially finite.
\end{theorem}

The following lemma is a direct consequence of Proposition\;\ref{faceqfuncfinite} and Theorem\;\ref{tautiltfinitealgebra}. 
\begin{lemma}\label{tautiltingfiniteness}
Let $\Lambda=KQ/I$ be a $\tau$-tilting finite algebra and $\Gamma$ a factor algebra of $\Lambda$.
Then $\Gamma$ is also $\tau$-tilting finite. In particular, there are no multiple arrows in $Q\setminus\{\mathrm{loops}\}$. 
\end{lemma}
\begin{proof}
Let $\mathcal{T}$ be a torsion class in $\mod \Gamma$ and
$\widehat{\mathcal{T}}:=\{X\in \mod \Lambda\mid X\otimes_{\Lambda} \Gamma\in \mathcal{T}\}$.
It is easy to check that $\widehat{\mathcal{T}}$ is a torsion class in $\mod \Lambda$.
Since $\Lambda$ is $\tau$-tilting finite, $\widehat{\mathcal{T}}$ is functorially finite by Theorem\;\ref{tautiltfinitealgebra}.
Then Proposition\;\ref{faceqfuncfinite} says that there exists $M\in \widehat{\mathcal{T}}$ such that $\fac M=\widehat{\mathcal{T}}$.
This implies $\fac(M\otimes_{\Lambda} \Gamma )=\mathcal{T}$. In fact, for any $X\in \mathcal{T}\subset \widehat{\mathcal{T}}$,
we have an exact sequence
\[\fac M\ni N\to X\to 0.\]
Thus we have an exact sequence
\[\fac (M\otimes_{\Lambda} \Gamma)\ni N\otimes_{\Lambda} \Gamma \to X\otimes_{\Lambda} \Gamma(=X)\to 0.\]
Hence the assertion follows from Proposition\;\ref{faceqfuncfinite}.
\end{proof}
%%%%%%%%%%%%%%%%%%%%%%%%%%%%%%%%%%%%%%%%%%%%%%%%%%%%%%%%%%%%%%%%%%%%%%%%%%%%%%%%%%%%%%%%%%%%%%%
%%%%%%%%%%%%%%%%%%%%%%%%%%%%%%%%%%%%%%%%%%%%%%%%%%%%%%%%%%%%%%%%%%%%%%%%%%%%%%%%%%%%%%%%%%%%%%%
\subsection{Lattice structure}
Let $\P$ be a poset and $x,y\in \P$. If $\{z\in \P\mid z\geq x,y\}$ (resp. $\{z\in \P\mid z\leq x,y\}$)
 admits a minimum element (resp. a maximum element), then we denote it by $x\vee y$ (resp. $x\wedge y$) and call the
  {\bf join} (resp. the {\bf meet}) of $x,y$. 
$\P$ is said to be a {\bf lattice} if for any $x,y\in\P$, there are both the join and the meet of $x,y$. 

The following result is useful to study finite support $\tau$-tilting posets and we use it everywhere in this paper.      
\begin{theorem}[{\cite[Theorem\;1.2]{IRTT}}]
Support $\tau$-tilting posets of $\tau$-tilting finite algebras have a lattice structure.
\end{theorem}   

\subsection{A connection between two-term silting complexes}
%In this subsection, we recall the notion of (two-term) silting complexes of perfect derived category $\Kb(\proj \Lambda)$.

We denote by $\Kb(\proj\Lambda)$ the bounded homotopy category of $\proj\Lambda$.
A complex $T=[\cdots\rightarrow T^i\rightarrow T^{i+1}\rightarrow\cdots]$ in $\Kb(\proj\Lambda)$ 
is said to be {\bf two-term} provided $T^i=0$ unless $i=0,-1$.
We recall the definition of silting complexes.

\begin{definition} Let $T$ be a complex in $\Kb(\proj \Lambda)$.
\begin{enumerate}
\item We say that $T$ is {\bf presilting} if $\Hom_{\Kb(\proj \Lambda)}(T,T[i])=0$ for any positive integer $i$.
\item A {\bf silting complex} is defined to be presilting and generate $\Kb(\proj \Lambda)$ by taking direct summands, mapping cones and shifts.   
\end{enumerate}
We denote by $\tsilt\Lambda$ (resp. $\tpsilt \Lambda)$ the set of isomorphism classes of basic two-term silting (resp. basic two-term presilting) complexes in $\Kb(\proj \Lambda)$.
\end{definition}

The set $\tsilt\Lambda$ also has poset structure as follows.

\begin{definitiontheorem}[{\cite[Theorem 2.11]{AI}}]
For two-term silting complexes $T$ and $T'$ of $\Kb(\proj\Lambda)$, we write $T\geq T'$
 if $\Hom_{\Kb(\proj\Lambda)}(T, T'[1])=0$.
Then the relation $\geq$ gives a partial order on $\tsilt\Lambda$.
\end{definitiontheorem}

The following result connects silting theory with $\tau$-tilting theory.

\begin{theorem}[{\cite[Corollary\;3.9]{AIR}}]\label{bijection}

We consider an assignment 
\[\begin{array}{ccccr}
&  &\;\;(-1\mathrm{th})& & (0\mathrm{th})\; \\
\mathbf{S}:(M,P)&\mapsto &[\;P_1\oplus P&\stackrel{(p_M,0)}{\longrightarrow}&P_0\;\;]\\
\end{array} \vspace{5pt}
\]
 where $p_M:P_1\to P_0$ is a minimal projective presentation of $M$.
\begin{enumerate}[{\rm (1)}]
\item\cite[Lemma\;3.4]{AIR} For modules $M,N$, the following are equivalent:
\begin{enumerate}[{\rm (a)}]
\item $\Hom_{\Lambda}(M,\tau N)=0$.
\item $\Hom_{\Kb(\proj \Lambda)}(\mathbf{S}(N),\mathbf{S}(M)[1])=0$.
\end{enumerate}
\item\cite[Lemma\;3.5]{AIR} For any projective module $P$ and any module $M$, the following are equivalent:
\begin{enumerate}[{\rm (a)}]
\item $\Hom_{\Lambda}(P,M)=0$.
 \item $\Hom_{\Kb(\proj \Lambda)}(\mathbf{S}(0,P),\mathbf{S}(M)[1])=0.$
\end{enumerate} 
\end{enumerate}
Moreover, the assignment $\mathbf{S}$ gives rise to a poset isomorphism 
$\sttilt\Lambda\xrightarrow{\sim}\tsilt\Lambda$.
\end{theorem}

\begin{lemma}[{\cite[Lemma\;2.25]{AI}}]\label{factformpp}
Let $M$ be a $\tau$-rigid module and $P_1\stackrel{d}{\to} P_0\to M\to 0$ a minimal
projective presentation of $M$. Then $\add P_1\cap \add P_0=\{0\}.$
In particular, for a two-term silting complex $[P_1\stackrel{d}{\to} P_0]$,
 we may assume that $\add P_1\cap \add P_0=\{0\}$.
\end{lemma}

We will close this section by recalling the definition and an important property of $g$-vectors of complexes of $\Kb(\proj \Lambda)$.

Let $K_0(\proj\Lambda)$ be the Grothendieck group of $\proj\Lambda$ and $[P]$ denote the element in $K_0(\proj\Lambda)$ corresponding to a projective module $P$.
As is well-known, the set $\{[e_i\Lambda]\ |\ i\in Q_0 \}$ forms a basis of $K_0(\proj\Lambda)$.

\begin{definition}
Let $X=[P'\to P]$ be a two-term complex of $\Kb(\proj \Lambda)$ and write $[P]-[P']=\sum_{i\in Q_0}g_i^X [e_i \Lambda]$ in $K_0(\proj \Lambda)$ for some $g_i^X\in\Z$.
Then we call the vector $g^X:=(g_i^X)_{i\in Q_0}\in \Z^{Q_0}$ the {\bf $g$-vector} of $X$.
\end{definition}

\begin{theorem}
\label{gvector}
\cite[Theorem 5.5]{AIR}
The map $T\mapsto g^T$ gives an injection from the set of isomorphism classes of two-term presilting complexes
to $K_0(\proj \Lambda)$.
\end{theorem}

\section{Remarks on poset isomorphism between two support $\tau$-tilting posets}

 In this section, we give some general results on poset isomorphism between two support $\tau$-tilting posets.
We assume that $|\Lambda|=n$ and $Q_0=\{1,2,\dots,n\}$.

 We first consider the direct predecessors of $0$ and the direct successors of $\Lambda$.
  We let
 \[X_i=X^{\Lambda}_i:=e_i\Lambda/e_i\Lambda(1-e_i)\Lambda\simeq \Lambda/(1-e_i).\] 
Then $X_i$ is in $\sttilt \Lambda$ with $\supp(X_i)=\{i\}$. Hence we have 
\[\dip(0)=\{X_i\mid i\in Q_0\}.\]
Since $\Lambda=P_1\oplus P_2\oplus \cdots\oplus P_n\in \sttilt \Lambda$, 
there exists a unique direct successor of $\Lambda$ in $\H(\sttilt \Lambda)$ which does not contain $P_i$
as a direct summand, for each $i\in Q_0$. We denote it by $Z_i\in \sttilt \Lambda$.
Thus we have
\[\dis(\Lambda)=\{Z_i\mid i\in Q_0\}.\]
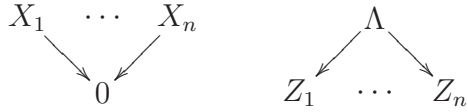
\begin{figure}[h]
\[\begin{xy}
(0,-5)*[o]+{0}="A", (-10,5)*[o]+{X_1}="B", (0,5)*[o]+{\cdots}="C", (10,5)*[o]+{X_n}="D",
\ar "B";"A"
\ar "D";"A" 
\end{xy}\hspace{25pt}
\begin{xy}
(0,5)*[o]+{\Lambda}="A", (-10,-5)*[o]+{Z_1}="B", (0,-5)*[o]+{\cdots}="C", (10,-5)*[o]+{Z_n}="D",
\ar "A";"B"
\ar "A";"D" 
\end{xy}
\]
\caption{Neighbors of $0$ and  $\Lambda$}
\label{fig:one}
\end{figure}

 \subsection{$\tau$-rigid pairs in the support $\tau$-tilting poset}
 \label{subsec:tau-rigid in poset}
Let  $\U^+_{\ell}$ (resp. $\U^-_{\ell}$) be the set of all connected fullsubquivers of $\H(\sttilt\Lambda)$ having
$\ell+1$ vertices with $\ell$ sources (resp. sinks). We set
\[\U^+=\U^+_{\Lambda}:=\bigsqcup_{\ell}\U^+_{\ell}\text{ and } \U^-=\U^-_{\Lambda}:=\bigsqcup_{\ell}\U^-_{\ell}.\]
Let $u\in \U^+_{\ell}$, $T=T_0$ the unique sink of $u$ and $T_1,\dots,T_{\ell}$ are 
sources of $u$.
 We denote by $\widetilde{T}_i:=T_i\oplus U_i^-$
the corresponding $\tau$-tilting pair of $T_i$. Then there exists
a unique basic $\tau$-rigid pair $X_u$ such that $\add X_u=\bigcap \add \widetilde{T}_i$.
 It is easy to check that $|X_u|=|\Lambda|-\ell$. Then we denote this assignment ($u\mapsto X_u$) by
 $\kappa^+=\kappa^+_{\Lambda}:\U^+\to \trigidp \Lambda$. Similarly, we define $\kappa^-=\kappa^-_{\Lambda}:\U^-\to \trigidp \Lambda$.
 
 Conversely, let $X\in \trigidp \Lambda$ with $|X|=|\Lambda|-\ell$. By Jasso's reduction theorem (Theorem\;\ref{reduction theorem}),
 there are the minimum element $\stmin (X)$ and the maximum element
 $\stmax (X)$ of $\sttilt_X \Lambda$. We note that $\stmin (X)$ (resp. $\stmax (X)$)
 has $\ell$ direct predecessors (resp. successors) in $\sttilt_X \Lambda$.  
 Let $T_1,\dots,T_{\ell}$ (resp. $T'_1,\dots,T'_{\ell}$) be direct predecessors of $\stmin (X)$ (resp. direct successors of $\stmax (X)$)
  in $\sttilt_X \Lambda$.
  Then we define $\upsilon^+:\trigidp \Lambda\to \U^+$ and $\upsilon^-:\trigidp \Lambda\to \U^-$ as follows:
 \[\begin{array}{lll}
 \upsilon^+(X)=\upsilon^+_{\Lambda}(X)&:=&\text{the full subquiver of $\H(\sttilt \Lambda)$ consists of $\stmin(X)$ and $T_1,\dots,T_{\ell}$},\\
 \upsilon^-(X)=\upsilon^-_{\Lambda}(X)&:=&\text{the full subquiver of $\H(\sttilt \Lambda)$ consists of $\stmax(X)$ and $T'_1,\dots,T'_{\ell}$}.\\
 \end{array} 
 \]
  Similarly, we define $\upsilon^-=\upsilon^-_{\Lambda}:\trigidp \Lambda\to \U^-$.
 \[\upsilon^+(X)=\begin{xy}
 (0,-5)*[o]+{\stmin(X)}="A", (-10,5)*[o]+{T_1}="B", (0,5)*[o]+{\cdots}="C", (10,5)*[o]+{T_{\ell}}="D",
 \ar "B";"A"
 \ar "D";"A" 
 \end{xy}\hspace{25pt}
 \upsilon^-(X)=\begin{xy}
 (0,5)*[o]+{\stmax(X)}="A", (-10,-5)*[o]+{T'_1}="B", (0,-5)*[o]+{\cdots}="C", (10,-5)*[o]+{T'_{\ell}}="D",
 \ar "A";"B"
 \ar "A";"D" 
 \end{xy}
 \]

 By constructions, one sees that
 \[\upsilon^{\pm}\circ\kappa^{\pm}=\mathrm{id}_{\U^{\pm}}\text{ and } 
 \kappa^{\pm}\circ\upsilon^{\pm}=\mathrm{id}_{\trigidp \Lambda}\ (\text{double-sign corresponds}).\] 
 
 
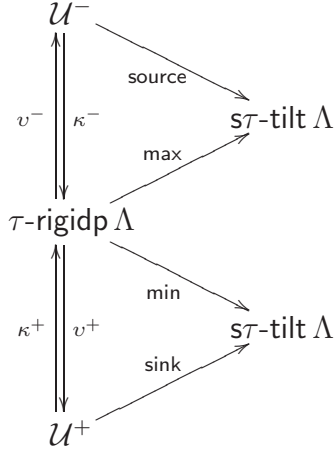
\[\begin{xy}
(0,0) *[o]+{\trigidp \Lambda}="A", (0,28) *[o]+{\U^-}="B", (0,-28) *[o]+{\U^+}="C", 
(28,14) *[o]+{\sttilt \Lambda}="D", (28,-14) *[o]+{\sttilt \Lambda}="E",
\ar @<6pt> "A";"B"^{\upsilon^-}
\ar @<-3pt> "B";"A"^{\kappa^-}
\ar @<-3pt> "A";"C"^{\upsilon^+}
\ar @<6pt> "C";"A"^{\kappa^+}
\ar "B";"D"_{\source}
\ar "C";"E"^{\sink}
\ar "A";"D"^{\stmax}
\ar "A";"E"_{\stmin}
\end{xy}\]
\begin{remark}
If $X$ is indecomposable, then $\stmin (X)$ has a unique direct successor and $\stmax (X)$ has a unique direct predecessor.
Hence $\stmin (X)$ is a join-irreducible element and $\stmax (X)$ is a meet-irreducible element.
 For more details, please refer to
\cite{IRRT}.
\end{remark}
The following lemma is useful in this section.
  \begin{lemma}
\label{updown}
 \begin{enumerate}[{\rm (1)}]
\item If $T\leq Z_i$ for any $i\in Q_0$, then $T=0$.
\item If $T\geq X_i$ for any $i\in Q_0$, then $T=\Lambda$.
\end{enumerate}
\end{lemma}
\begin{proof}
We show the assertion (1).  
We claim that 
\[Z_i\in  \fac \oplus_{k\neq i} P_k.\]
If $Z_i=\oplus_{k\neq i} P_k$, then we have nothing to show. Thus we may assume that
there exists a non projective indecomposable direct summand $M_i$ of $Z_i$. We consider a minimal projective presentation
\[Q^{(i)}\to P^{(i)}\to M_i\to 0\]
of $M_i$. Since $M_i\oplus P_k$ is $\tau$-rigid for any $k\neq i$,
Lemma\;\ref{factformpp} implies that $\add P^{(i)}\cap \add Q^{(i)}=\{0\}$ and 
$P_k\not\in\add Q^{(i)}$ for any $k\neq i$. Note that $M_i$ is not projective. Thus
we obtain $P_i\in \add Q^{(i)}$ and  $Z_i\in \fac \oplus_{k\neq i} P_k$. 

We assume that
$T\leq Z_i$ for any $i\in Q_0$. Then we have
 \[\fac T\subset \cap_{i\in Q_0} \fac Z_i\subset\cap_{i\in Q_0}\fac \oplus_{k\neq i} P_k=\{0 \}.\]  
 Hence $T=0$. The assertion (2) follows from (1) and Proposition\;\ref{revers}. 
\end{proof}

\begin{proposition}
\label{determiningsupport}
Let $T\in \sttilt \Lambda$ and $V\subset Q_0$. We put $e_V=\sum_{i\not\in V} e_i$.
\begin{enumerate}[{\rm (1)}] 
\item  $T=\Lambda/(e_V)=\stmax (e_V \Lambda^-)$
if and only if the following conditions hold.
\begin{enumerate}[{\rm (i)}]
\item $T\geq X_i$ for any $i\in V$.
\item The number of direct successors of $T$ is equal to that of $V$.
\item If $Y\leq T'$ holds for any $T'\in \dis(T)$, then $Y=0$.
\end{enumerate}
\item $T=\stmin (e_V \Lambda)$ if and only if
the following hold.
\begin{enumerate}[{\rm (i)}]
\item $T\leq Z_i$ for any $i\in V$.
\item The number of direct predecessors of $T$ is equal to that of $V$.
\item If $Y\geq T'$ holds for any $T'\in \dip(T)$, then $Y=\Lambda$.
\end{enumerate}
\end{enumerate}
\end{proposition}
\begin{proof} We show the assertion (1).
Assume that $T=T_0$ satisfies the conditions (i), (ii) and (iii). % and $\widetilde{T}_0:=T_0\oplus U_0^-$ is a $\tau$-tilting pair.
We denote by $\ell$ the number of vertices in $V$ and $T_1,\dots,T_{\ell}$
 the direct successors of $T$.
%Let $U_k$ be a basic projective module such that $\widetilde{T}_k:=T_k\oplus U_k^-$ is a $\tau$-tilting pair. 
Then we denote by $u_V$ the full subquiver of $\H(\sttilt \Lambda)$ consists of $T_0,T_1,\dots, T_{\ell}$
and put $\kappa^-(u_V):=M\oplus P^-$.
\[\U^-\ni \upsilon_V=\begin{xy}
 (0,5)*[o]+{T_0}="A", (-10,-5)*[o]+{T_1}="B", (0,-5)*[o]+{\cdots}="C", (10,-5)*[o]+{T_{\ell}}="D",
 \ar "A";"B"
 \ar "A";"D" 
 \end{xy} \]
%We take a basic $\tau$-rigid pair $(M,P)$ such that $\add (M,P)=\cap_{k=0}^{\ell} \add (T_k,U_k)$.
%Note that $|M\oplus P^-|=n-\ell$. In particular, $T$ is the maximum element of $\sttilt_{M\oplus P^-} \Lambda$.

 Now let $Y:=\stmin(M\oplus P^-).$
Since $T_k\in \sttilt_{M\oplus P^-} \Lambda$, we have that $Y\leq T_k$ for any $k$. By (iii), we obtain that
$Y=0$. In particular,  we have $M=0$ and $P=e_{V'}\Lambda$ for some $V'\subset Q_0$.
Then $T\in \sttilt_{P^-} \Lambda$ and (i) imply that $V\subset \supp (T)\subset V'$. 
Hence $V=V'=\supp(T)$ follows from the following equations.
\[\#V=\ell=n-(n-\ell)=n-|P|=n-(n-\#V')=\#V'.\] 
Since $T=\max(P^-)$, we have $T=\Lambda/(e_V)$.

Next we assume that $T=\Lambda/(e_V)=\stmax (e_V \Lambda^-)$. Since $\sttilt_{e_V\Lambda^-} \Lambda=\{T'\in \sttilt \Lambda\mid T'\leq T\}$, % is the maximum element of $\sttilt_{e_V \Lambda^-} \Lambda$,
 %We note that any $Y'\leq T$ is in $\sttilt_{e_V \Lambda^-} \Lambda$.
(i), (ii) and (iii) follow from Theorem\;\ref{reduction theorem} and Lemma\;\ref{updown}.

We remark that poset anti-isomorphism $(-)^{\dagger}:\sttilt \Lambda \to \sttilt \Lambda^{{\rm op}}$ in Proposition\;\ref{revers} sends
 $\sttilt_{e_V \Lambda}\Lambda$ to  $\sttilt_{(e_V\Lambda^{{\rm op}})^-} \Lambda^{\mathrm{op}}$.
Also we have $(Z_i)^{\dagger}=X_i^{\Lambda^{\mathrm{op}}}$ and $\Lambda^{\dagger}=0$. Hence the assertion (2) follows from (1).  
   \end{proof}
 Now we state main result of this subsection.   
\begin{corollary}
\label{preservingsupport}
Let $\Lambda=KQ/I$ and $\Gamma=KQ'/I'$. Assume that there is a poset isomorphism
$\rho: \sttilt \Lambda\stackrel{\sim}{\to } \sttilt \Gamma$ and define $\sigma:Q_0\to Q'_0$ by
$\rho (X_i^{\Lambda})=X_{\sigma(i)}^{\Gamma}$. 
\begin{enumerate}[{\rm (1)}]
\item Let $V\subset Q_0$, $V'=\sigma(V)$, $e=\sum_{i \in V} e_i\in \Lambda$ and $e'=\sum_{i'\in V'} e_{i'}\in \Gamma$.
Then $\rho$ induces poset isomorphisms
\[\sttilt_{e\Lambda^-} \Lambda\simeq \sttilt_{e'\Gamma^-} \Gamma\text{ and } \sttilt_{e\Lambda} \Lambda\simeq \sttilt_{e'\Gamma} \Gamma.\]
\item We have
\[\supp(\rho(T))=\sigma(\supp(T)).\]
In particular, $\rho$ induces a poset isomorphism
\[\rho|_{\ttilt \Lambda}:\ttilt \Lambda\stackrel{\sim}{\to} \ttilt \Gamma.\]

\item If $\sttilt \Lambda$ is a lattice, then we have
\[\begin{array}{lllll}
\source\circ\upsilon^-\circ\kappa^+(u^+)&=&\stmax\circ  \kappa^+(u^+)&=&\bigvee u^+\\
\sink\circ\upsilon^+\circ\kappa^-(u^-)&=&\stmin\circ\kappa^-(u^-)&=&\bigwedge u^-\\
\end{array}
\]
 for any $u^+\in \U^+$ and $u^-\in \U^-$.
\item Define bijections $\widetilde{\rho}^{\pm}:\trigidp \Lambda\to \trigidp \Gamma$ by
\[\widetilde{\rho}^{\pm}:=\kappa_{\Gamma}^{\pm}\circ \rho\circ \upsilon_{\Lambda}^{\pm}\ (\text{double-sign corresponds}).\]
If $\sttilt \Lambda$ is a lattice, then we have $\widetilde{\rho}^+=\widetilde{\rho}^-(=:\widetilde{\rho})$.
Moreover, for each basic $\tau$-rigid pair $X$ of $\Lambda$, $\rho$ induces a poset isomorphism
\[\rho\mid_{\sttilt_X \Lambda}:\sttilt_{X} \Lambda\stackrel{\sim}{\to}\sttilt_{\widetilde{\rho}(X)}  \Gamma. \]
\end{enumerate}
\end{corollary}
\begin{proof}
%Since $\{T\in\sttilt \Lambda\mid \supp(T)\subset V\}=\sttilt_{e_V\Lambda^-} \Lambda=[0,\Lambda/(e_V)]$,
The assertions (1) follows from Proposition\;\ref{determiningsupport} and the assertion (2) is a direct consequence of the assertion (1).
  
We prove (3). By definition, $\stmax (\kappa^+(u^+))\geq T$ for any $T\in u^+$.
Hence, we have 
\[\stmax\circ  \kappa^+(u^+)\geq \bigvee u^+.\]
This implies that $\bigvee u^+$ is in 
\[[\sink u^+,\stmax(\kappa^+(u^+))]
=[\stmin(\kappa^+(u^+)),\stmax(\kappa^+(u^+))]=\sttilt_{\kappa^+(u^+)} \Lambda.\]
Therefore, $\stmax\circ \kappa^+(u^+)=\bigvee u^+$ follows from Lemma\;\ref{updown}.
A similar argument implies 
\[\stmin\circ\kappa^-(u^-)=\bigwedge u^-.\]

We show (4). From (3), we have equalities
\[\begin{array}{lll}
\sttilt_{\widetilde{\rho}^+(X)} \Gamma &=&[\stmin(\widetilde{\rho}^+(X)),\stmax(\widetilde{\rho}^+(X))]\\\\
                                   &=&[\stmin(\kappa^+\circ \rho\circ \upsilon^+(X)),\stmax(\kappa^+\circ\rho\circ\upsilon^+(X))]\\\\
                                   &=&[\sink(\rho(\upsilon^+(X))),\bigvee \rho(\upsilon^+(X))]\\\\
                                   &=&[\rho(\sink( \upsilon^+(X))),\rho(\bigvee (\upsilon^+(X)))]\\\\
                                   &=&\rho([\sink( \upsilon^+(X)),\stmax(\kappa^+\circ \upsilon^+(X))])\\\\
                                   &=&\rho([\stmin(X),\stmax(X)])\\\\
                                   &=&\rho(\sttilt_X \Lambda).\\
\end{array}
\]
Similarly, one can check that
\[\sttilt_{\widetilde{\rho}^-(X)} \Gamma=\rho(\sttilt_X \Lambda).\]
This finishes a proof.
\end{proof}
\subsection{From support $\tau$-tilting posets to quivers}
The aim of this subsection is to reconstruct the Gabriel quiver of $\Lambda$ (up to multiple arrows and loops) from their support 
$\tau$-tilting poset.

We define a new quiver $Q^*$ from $Q$ as follows$:$
\begin{enumerate}[{\rm (i)}]
\item $Q^*_0:=Q_0$.
\item We draw one arrow from $i$ to $j$ if there is an arrow from $i$ to $j$ on $Q^{\circ}$.
\end{enumerate}
\begin{example}
Let $Q$ be the following quiver.
\[\begin{xy}
(0,0)*[o]+{1}="A", (10,0)*[o]+{2}="B", (20,0)*[o]+{3}="C",
\ar @(ul,dl) "A";"A"
\ar @<4pt> "A";"B"
\ar "A";"B"
\ar @<4pt> "B";"A"
\ar "B";"C"
\end{xy}\]
Then $Q^*$ is given by the following quiver.
\[\begin{xy}
(0,0)*[o]+{1}="A", (10,0)*[o]+{2}="B", (20,0)*[o]+{3}="C",
%\ar @(ul,dl) "A";"A"
%\ar @<2pt> "A";"B"
\ar @<2pt>"A";"B"
\ar @<2pt> "B";"A"
\ar "B";"C"
\end{xy}
\]
\end{example}
\begin{proposition}
\label{determiningarrows}
Let $i\neq j\in Q_0$.
\begin{enumerate}[{\rm (1)}]
\item There is no arrow between $i$ and $j$ if and only if
 $\Lambda/(1-e_i-e_j)\in \dip(X_i)\cap \dip(X_j)$.
\item There is an arrow from $i$ to $j$ and  no arrow from $j$ to $i$
if and only if $\Lambda/(1-e_i-e_j)\in \dip(X_j)\setminus\dip(X_i)$.
\item There is an arrow from $i$ to $j$ and an arrow from $j$ to $i$
 if and only if $\Lambda/(1-e_i-e_j)\not\in \dip(X_i)\cup \dip(X_j).$
\item Let $Q'$ be a finite quiver, $I'$ an admissible ideal of $KQ'$ and
$\Gamma=KQ'/I'$. Assume that there is a poset  
isomorphism $\rho:\sttilt \Lambda\stackrel{\sim}{\to} \sttilt \Gamma$ and
put $\sigma:Q_0\to Q'_0$ as in Corollary\;\ref{preservingsupport}. 
Then $\sigma$ induces a quiver isomorphism
\[Q^*\stackrel{\sim}{\to} (Q')^*.\]
\end{enumerate}
\end{proposition}
\begin{proof} Let $\Lambda':=\Lambda/(1-e_i-e_j)$.
Note that $\Lambda'\in \dip(X_i)\cap \dip(X_j)$ if and only if 
\[
 \sttilt \Lambda'=\sttilt_{\oplus_{k\neq i,j}P_k^-} \Lambda=
 \begin{xy}
 (0,-10)*[o]+{0}="A", (-10,0)*[o]+{X_i}="B", (10,0)*[o]+{X_j}="C", (0,10)*[o]+{\Lambda'}="D"
 \ar "B";"A"
 \ar "C";"A" 
 \ar "D";"B"
 \ar "D";"C" 
\end{xy}
\]
This is equivalent to that $\Lambda'\simeq X_i\oplus X_j$ and thus
$X_i$ and $X_j$ are projective as $\Lambda'$-modules.
Since the quiver of $\Lambda/(1-e_i-e_j)$ is the full subquiver of $Q$ with 
two vertices $i$ and $j$, we obtain the assertion (1).

We show the assertion (2). First we assume that there is an arrow from $i$ to $j$ and no arrow from $j$ to $i$ on $Q^{\circ}$.
Then $e_j\Lambda' e_i =0$ and $e_i \Lambda' e_j\neq 0$. In particular, we have
\[X_j=X_j^{\Lambda'}=e_j \Lambda'/e_j\Lambda' e_i \Lambda'=e_j\Lambda' \in \add\Lambda'\]
 and $X_i\not\in \add \Lambda'$. 
This implies that  $\Lambda'\in \dip(X_j)\setminus\dip(X_i)$.
Next we assume that $\Lambda'\in \dip(X_j)\setminus\dip(X_i)$.
In this case, $X_j\in \add \Lambda'$. Hence, there is no arrow from $j$ to $i$.
Existence of an arrow from $i$ to $j$ follows from (1).

Then the assertion (3) follows from (1) and (2), and the assertion (4) follows from Corollary\;\ref{preservingsupport} (1),
 (1), (2) and (3).
\end{proof}
%%%%%%%%%%%%%%%%%%%%%%%%%%%%%%%%%%%%%%%%%%%%%%%%%%%%%%%%%%%%%%%%%%%%%%%%%%%%%%%%%%%%%%%%%%
%%%%%%%%%%%%%%%%%%%%%%%%%%%%%%%%%%%%%%%%%%%%%%%%%%%%%%%%%%%%%%%%%%%%%%%%%%%%%%%%%%%%%%%%%%
Now we can recover Happel-Unger's result in \cite{HU2}. For a finite quiver $Q$, we define a decorated quiver $Q_{\mathrm{dec}}$ of $Q$
as follows: (i) The vertices of $Q_{\mathrm{dec}}$ is that of $Q$; (ii) If there is a unique arrow from $i$ to $j$ in
$Q$, then we draw a one arrow $i\to j$ in $Q_{\mathrm{dec}}$; (iii) If there are at least two arrows from $i$ to $j$ in $Q$, then we draw
 a decorated arrow $i\stackrel{*}{\to} j$ in $Q_{\mathrm{dec}}$. 
\begin{corollary}[{\cite[Theorem\;6.4]{HU2}}]
Let $Q$ and $Q'$ be two finite acyclic quivers. Then $\stilt KQ\simeq \stilt KQ'$ only if
$Q_{\mathrm{dec}}\simeq Q'_{\mathrm{dec}}$.
\end{corollary} 
\begin{proof} We put $\Lambda=KQ$ and $\Gamma=KQ'$.
Let $\rho:\stilt \Lambda\stackrel{\sim}{\to} \stilt \Gamma$ and $\sigma:Q_0\rightarrow Q'_0$ as in Corollary\;\ref{preservingsupport}.
By Corollary\;\ref{preservingsupport} (1), we have $\rho(\Lambda/(1-e_i-e_j))=\Gamma/(1-e_{i'}-e_{j'})$,
where $i'=\sigma(i)$ and $j'=\sigma(j)$. Hence we obtain
\[\stilt(\Lambda/(1-e_i-e_j))=[0,\Lambda/(1-e_i-e_j)]\simeq
[0,\Gamma/(1-e_{i'}-e_{j'})]=\stilt(\Gamma/(1-e_{i'}-e_{j'})).
\] 
Assume that there are at least two arrows from $i$ to $j$
in $Q$. By Proposition\;\ref{determiningarrows}, there is an arrow from $i'$ to $j'$ in $Q'$.
Since $\stilt \Gamma/(1-e_{i'}-e_{j'})\simeq \stilt \Lambda/(1-e_i-e_j)$ has infinitely many
elements, we have that there are at least two arrows from $i'$ to $j'$ in $Q'$.
Thus the assertion follows from Proposition\;\ref{determiningarrows}.
\end{proof}
\begin{remark}
{\cite[Theorem\;6.4]{HU2}} says more strongly result than above corollary, i.e., a poset isomorphism $\tilt KQ \simeq \tilt KQ'$
implies $Q_{\mathrm{dec}}\simeq Q'_{\mathrm{dec}}$. Then it is interesting  whether a poset isomorphism
 $\ttilt \Lambda\simeq \ttilt \Gamma$ gives us a poset isomorphism $\sttilt \Lambda\simeq \sttilt \Gamma$.
\end{remark}
\begin{corollary}
\label{posetauttoquiveraut}
Assume that $\Lambda=KQ/I$ is $\tau$-tilting finite. 
\begin{enumerate}[{\rm (1)}]
\item In the setting of Proposition\;\ref{determiningarrows} (4),
%Let $\Gamma=KQ'/I'$, $\rho:\sttilt \Lambda\stackrel{\sim}{\to} \sttilt \Gamma$ and $\sigma:Q_0\to Q'_0$ as in 
 $\sigma$ induces an quiver isomorphism
\[Q\setminus\{\mathrm{loops}\}\stackrel{\sim}{\to} Q'\setminus\{\mathrm{loops}\}.\]
\item  Let $\rho,\rho'$ be a poset isomorphism from $\sttilt \Lambda$ to $\sttilt \Gamma$.
If $\rho(X_i^{\Lambda})=\rho'(X_i^{\Lambda})$ holds for any $i\in Q_0$, then we have
$\rho=\rho'$. In particular, there is a group monomorphism
\[\Aut_{\mathsf{poset}} (\sttilt \Lambda)\inj \Aut_{\mathsf{quiver}}(Q^{\circ}).\]
\end{enumerate}
\end{corollary}
\begin{proof}
The assertion (1) follows from Lemma\;\ref{tautiltingfiniteness} and Proposition\;\ref{determiningarrows} (4).

We prove the assertion (2) by using an induction on $|\Lambda|$.
It is obvious that the assertion holds for $|\Lambda|=1,2$.
Thus we assume that the assertion holds for the case that $|\Lambda|<n$ ($n>2$) and consider the case 
$|\Lambda|=n.$ 
\begin{claim}
\label{claimposetaut}
Let $(T_0< T_1< \cdots< T_{\ell})$
be a sequence of vertices in $\H(\sttilt \Lambda)$ satisfying the following conditions:
\begin{enumerate}[{\rm (a)}]
\item $T_0=0$.
\item $\# \dip (T_r)=n-1$ holds for any $r>0$.
\item $T_1\in [0, \bigvee_{j\neq i} X_j]$ for some $i\in Q_0$ and $T_r\in [T_{r-1},\bigvee_{Y\in \dip(T_{r-1})} Y]$ $(r>1)$.
\end{enumerate}
We set \[\P_r:=\begin{cases}
[0,\bigvee_{j\neq i} X_j] & r=0 \\ 
[T_r,\bigvee_{Y\in \dip(T_r)} Y] & r>0.\\
\end{cases}\] 
If $T\in \bigcup_{r=0}^{\ell}[T_r,\bigvee_{Y\in \dip(T_r)} Y]$, then we have $\rho (T)=\rho'(T)$.
\end{claim}
\begin{pfclaim}
Let $u_0$ be the element of $\U^+_{n-1}$ given by $0$ and $\dip(0)\setminus\{X_i\}$ and
$u_r$ the element of $\U^+_{n-1}$ given by $T_r$ and $\dip(T_r)$.
By Corollary\;\ref{preservingsupport} (3), we have 
\[\P_r=\sttilt_{\kappa^+(u_r)} \Lambda.\]
%If $r>0$, then $\supp (\oplus_{Y\in \dip(T_r)} Y)=Q_0$. Thus (b) implies that there exists an indecomposable
%$\tau$-rigid module $M_r$ such that $M_r\in \add T_r\cap (\cap_{Y_{\in \dip(T_r)}} \add Y)$.
%In particular, we have 
%\[\P_r=\sttilt_{M_r} \Lambda.\]
By Theorem\;\ref{reduction theorem}, there is an finite dimensional algebra $\Lambda_r$ with $|\Lambda_r|=n-1$ such that
\[\P_r\simeq \sttilt \Lambda_r.\]
We have $\kappa^+(u_0)=P_i^-$ and
\[\P_0=\sttilt_{P_i^-} \Lambda=\sttilt \Lambda/(e_i).\]

Since $\rho(X_j)=\rho'(X_j)$ holds for any $j\in Q_0$, we have
that $\rho(\P_0)=\rho'(\P_0)$.
Then by using hypothesis of induction, we obtain that
\[\rho(T)=\rho'(T)\]   
 for any $T\in \P_0$.
Now we consider $\dip(T_1)$. Since $\P_0$ is $(n-1)$-regular and $T_1\neq T_0=\stmin(\kappa^+(u_0))$, there is a unique
direct predecessor $Y_1$ of $T_1$ which is not contained in $\P_0$.
We let $\{Y_2,\dots, Y_{n-1}\}=\dip (T_0)\setminus\{Y_1\}$. Then $\rho(Y_k)=\rho'(Y_k)$
holds for any $k\geq 2$. Therefore $\rho(Y_1)=\rho'(Y_1)$ also holds. 
This gives that $\rho(\P_1)=\rho'(\P_1)$. Hence the hypothesis of induction implies that
\[\rho(T)=\rho'(T)\]
holds for any $T\in \P_1$. A similar argument gives the assertion.  
\end{pfclaim}
Let $\mathcal{P}$ be a subset of $\sttilt \Lambda$ consists of those element $T$
such that $T\in \P_\ell$ for some $(T_0<\cdots<T_{\ell})$ satisfying (a), (b) and (c).
Suppose that $\mathcal{P}\neq \sttilt \Lambda$.
Since $\Lambda$ is $\tau$-tilting finite, we can take a minimal element $T$ of $\sttilt \Lambda\setminus \mathcal{P}$.
 We note that $0\in \mathcal{P}$. Hence $T\neq 0$ and there is a direct successor $T'$ of $T$.
 If $T'=0$, then it is obvious that $T\in \mathcal{P}$.
 Thus, we may assume that $T'\neq 0$.
 In this case, there is an indecomposable $\tau$-rigid module $M$
 such that $T,T'\in \sttilt_M \Lambda$. Let $T''$ be the minimum element of $\sttilt_M \Lambda$.
 By minimality of $T$, we get  $0\neq T''\in \mathcal{P}$. Thus there is a sequence $(T_0<\cdots<T_{\ell})$
 satisfying (a), (b), (c) and $T''\in \P_{\ell}$. Indecomposability of $M$ implies that
 $\sttilt_M \Lambda$ is $(n-1)$-regular. Thus $T''$ has $n-1$ direct predecessors. Therefore, $(T_0<T_1<\cdots <T_{\ell}<T'')$
 satisfies (a), (b) and (c). We also have  $T\in \sttilt_M \Lambda=[T'',\bigvee_{Y\in \dip(T'')} Y]$.
 This contradicts to $T\not\in \mathcal{P}$. Hence we obtain $\mathcal{P}=\sttilt \Lambda$.
 Then the assertion follows from Claim\;\ref{claimposetaut}. 
\end{proof}
%%%%%%%%%%%%%%%%%%%%%%%%%%%%%%%%%%%%%%%%%%%%%%%%%%%%%%%%%%%%%%%%%%%%%%%%%%%%%%%%%%%%%%%%
%%%%%%%%%%%%%%%%%%%%%%%%%%%%%%%%%%%%%%%%%%%%%%%%%%%%%%%%%%%%%%%%%%%%%%%%%%%%%%%%%%%%%%%%%
\subsection{Other remarks}
In this subsection, we show some results used in the next section.
\begin{lemma}
\label{projsupport}
Let $\Lambda=KQ/I$ and $\Gamma=KQ'/I'$ be two basic algebras.
Assume that there is a poset isomorphism $\rho:\sttilt \Lambda\stackrel{\sim}{\to} \sttilt \Gamma$.
We define a quiver isomorphism $\sigma:Q^*\stackrel{\sim}{\to} (Q')^*$ as in Proposition\;\ref{determiningarrows}.
For any subset $V$ of $Q_0$ and $i\not\in V$, we have the following equality.
\[\supp(e_{\sigma(i)} \Gamma/(\sum_{v\in V} e_{\sigma(v)}))=\sigma(\supp(e_i \Lambda/(\sum_{v\in V} e_v))).\]
\end{lemma}
\begin{proof}
 We put $e=\sum_{v\in V} e_v$ and $e'=\sum_{v \in V} e_{\sigma(v)}$.
By Corollary\;\ref{preservingsupport}, $\rho$ induces an isomorphism
\[\rho_V:\sttilt \Lambda/(e)\simeq \sttilt \Gamma/(e').\]
Since $\rho_V$ sends $X^{\Lambda/(e)}_j=X^{\Lambda}_j$ to $X^{\Gamma/(e')}_{\sigma(j)}=X^{\Gamma}_{\sigma(j)}$ for any $j\not\in V$,
it is sufficient to show the case $V=\emptyset$.
\begin{claim*}
Let $Z\in\dis(\Lambda)$. Then $Z=Z_i$ if and only if $Z\geq X_k$ for any $k\neq i$.
\end{claim*}
\begin{pfclaim}
The assertion follows from the fact that $\fac Z_j=\fac \bigoplus_{k\neq j}P_k$. (See the proof of Lemma\;\ref{updown}.)
\end{pfclaim} 
Let $T_i=T_i^{\Lambda}:=\stmin(e_i\Lambda)$ (i.e., the minimum element of $\sttilt_{e_i\Lambda} \Lambda$).
By the above claim, we have that $\rho(Z_i^{\Lambda})=Z_{\sigma(i)}^{\Gamma}$.
Therefore, Proposition\;\ref{determiningsupport} implies an equality
\[\rho (T_i^{\Lambda})=T_{\sigma(i)}^{\Gamma}.\]

We show that $\supp(e_i\Lambda)=\supp(T_i^{\Lambda})$ and $\supp(e_{\sigma(i)}\Gamma)=\supp(T_{\sigma(i)}^{\Gamma})$.
We put
\[e:=1-\sum_{k\in \supp(e_i \Lambda)} e_k.\] Then $X:=e_i\Lambda\oplus e\Lambda^-\in \trigidp \Lambda$.
By definition, we have $\stmin(X)\geq T_i^{\Lambda}$. This shows that $\supp(T_i^{\Lambda})$ is contained in $\supp(e_i \Lambda)$.
On the other hand, $\supp(e_i \Lambda)\subset \supp (T_i^{\Lambda})$ follows from $e_i \Lambda\in \add T_i^{\Lambda}$.
Thus we have $\supp(e_i\Lambda)=\supp(T_i^{\Lambda})$ and $\supp(e_{\sigma(i)}\Gamma)=\supp(T_{\sigma(i)}^{\Gamma})$. 
Then $\supp(e_{\sigma(i)} \Gamma)=\sigma(\supp(e_i \Lambda))$ follows from Corollary\;\ref{preservingsupport}.
\end{proof}
\begin{lemma}
 \label{keylemma}
Assume that there is an arrow $\alpha$ from $i$ to $j$ in $Q^{\circ}$.
Then $\alpha\Lambda e_j=e_i\Lambda e_j=e_i\Lambda \alpha$ if and only if $P_i\oplus e_i\Lambda/e_i\Lambda e_j \Lambda$ is $\tau$-rigid.
\end{lemma}
\begin{proof} We put $M=e_i\Lambda/e_i\Lambda e_j \Lambda$.

We first assume that $P_i\oplus M$ is $\tau$-rigid.
Let $P_{j}^{\oplus r}\stackrel{f}{\to} P_i\to M \to 0$ be a minimal projective presentation of $M$ and
$P_{M}:=[P_{j}^{\oplus r}\stackrel{f}{\to} P_i]$ the corresponding two-term presilting complex in $\Kb(\proj \Lambda)$.
Since $P_i\oplus M$ is $\tau$-rigid, we have an equality
\[\Hom_{\Kb(\proj \Lambda)}(P_{M},P_{i}[1])=0.\]
 We put $P_j^{\oplus r}=\underset{1\leq t \leq r}{\bigoplus}P_j^{(t)}$, $f=(f^{(t)}:P_{j}^{(t)}\to P_i)$ and consider $\varphi\in\Hom_{\Kb(\proj \Lambda)}(P_{M}, P_i[1] ) $
 given by $\varphi^{(t)}:P_j^{(t)}\to P_i$, where
 $\varphi^{(t)}=\left\{
 \begin{array}{cl}
 \alpha & t=1 \\
 0& t\neq 1\\
 \end{array}\right.
 $.
  Then there exists $h\in \End_{\Lambda}(P_i)$ such that
 \[ h\circ f^{(t)}=\varphi^{(t)}\]
 for any $t$. Since $\alpha\in \rad \Lambda\setminus \rad^2 \Lambda$,  $h$ has to be an isomorphism and $r=1$.
 Let $x=f(e_j)$ and $y=h(e_i)\in e_i \Lambda e_i\setminus e_i \rad \Lambda e_i$. Then $x\Lambda=\Im f= e_i\Lambda e_j \Lambda$
 and $yx=\alpha$.
 Since $x\Lambda=e_i\Lambda e_j \Lambda$, there exists $y'\in e_j\Lambda e_j\setminus \rad(e_j \Lambda e_j)$
 such that $xy'=\alpha$. Hence we obtain
 \[\alpha\Lambda=xy' \Lambda =xe_j\Lambda=x\Lambda=e_i \Lambda e_j \Lambda.\]
   $\Hom_{\Kb(\proj \Lambda)}(P_M,P_{i}[1])=0$ implies that
  for any morphism  $g:P_j\rightarrow P_i$, there exists $h'\in \End_{\Lambda}(P_i)$ such that $g=h'\circ f$. This says that
  $e_i \Lambda e_j=e_i \Lambda x$.
 Therefore, we see that
 \[e_i\Lambda \alpha=e_i \Lambda yx=e_i \Lambda e_i x=e_i \Lambda  x= e_i \Lambda e_j.\] 

Next we assume that $\alpha\Lambda e_j=e_i\Lambda e_j=e_i\Lambda \alpha$. Then it is easy to check that
\[P_{j} \stackrel{\alpha\cdot-}{\to} P_i\to M\to 0\]
is a minimal projective presentation of $M$ and $\Hom_{\Kb}(P_M,P_{i}[1])=0$.
Since $M \simeq e_i(\Lambda/(e_j))$ is $\tau$-rigid, we obtain that
 $P_i\oplus M$ is also $\tau$-rigid.
\end{proof}
%%%%%%%%%%%%%%%%%%%%%%%%%%%%%%%%%%%%%%%%%%%%%%%%%%%%%%%%%%%%%%%%%%%%%%%%%%%%%%%%%%%%
%%%%%%%%%%%%%%%%%%%%%%%%%%%%%%%%%%%%%%%%%%%%%%%%%%%%%%%%%%%%%%%%%%%%%%%%%%%%%%%%%%%%
\begin{proposition}
\label{keyproposition}
Let $\Gamma=KQ'/I'$.  
Assume that $\sttilt \Lambda$ is a lattice and there exists a poset isomorphism
$\sttilt \Lambda\stackrel{\sim}{\to}\sttilt \Gamma$. We define a quiver isomorphism
$\sigma:Q^*\to (Q')^*$ as in 
Proposition\;\ref{determiningarrows}.
\begin{enumerate}[{\rm (1)}]
\item Let $i\neq j\in Q_0$. Then the restriction of $\rho$ on $\sttilt_{e_i(\Lambda/( e_j))} \Lambda$ gives a 
poset isomorphism
\[\sttilt_{e_i(\Lambda/( e_j))} \Lambda
\stackrel{\sim}{\to}\sttilt_{e_{i'}(\Gamma/( e_{j'}))} \Gamma,\]
where $i'=\sigma(i)$ and $j'=\sigma(j)$.
\item Assume that there is an arrow from $i$ to $j$ on $Q^{\circ}$. Then the following conditions are equivalent.
\begin{itemize}
\item $\alpha\Lambda e_j=e_i\Lambda e_j=e_i\Lambda \alpha $ holds for some $\alpha:i\to j $.
\item $ \alpha' \Gamma e_{\sigma(j)}=e_{\sigma(i)}\Gamma e_{\sigma(j)}=e_{\sigma(i)}\Gamma \alpha'$ 
holds for some $\alpha':\sigma(i)\to \sigma(j)$.
 \end{itemize}  
%\item $\supp(e_{\sigma(i)} \Gamma)=\sigma(\supp(e_i \Lambda)).$
\end{enumerate}
\end{proposition}
\begin{proof}
We may assume that $Q_0=Q'_0$ and $\rho(X_i^{\Lambda})=X_i^{\Gamma}.$
By Corollary\;\ref{preservingsupport}, we have that $\rho(\Lambda/(e_j))=\Gamma/(e_j)$.
We let $i\neq j\in Q_0$ and 
 $T=\stmin(e_i \Lambda/(e_j))$. Since $e_i\Lambda\oplus e_j \Lambda^-\in \trigidp \Lambda$, $T$
	is the minimum element $\stmin(e_i \Lambda/(e_j)\oplus e_j \Lambda^-)$ of 
	$\sttilt_{e_i\Lambda/(e_j)\oplus e_j \Lambda^-} \Lambda=\sttilt_{e_i\Lambda/(e_j)} \Lambda/(e_j)$.
Then Corollary\;\ref{preservingsupport} (1) implies 
 $\rho(T)=\stmin(e_i\Gamma/(e_j)\oplus e_j \Gamma^-)=\stmin(e_i\Gamma/(e_j))$.	

Since $0\neq T=\stmin(e_i \Lambda/(e_j))$,
there are exactly $n-1$ direct predecessors of $T$ and each of them is in $\sttilt_{e_i \Lambda/(e_j)} \Lambda$.
In particular, we obtain
\[\add e_i \Lambda/(e_j)=\add T\cap (\underset{T'\in \dip(T)}\cap \add T') .\]
Similarly, we have 
\[\add e_i \Gamma/(e_j)=\add \rho(T)\cap (\underset{T'\in \dip(T)}\cap \add \rho(T') ).\]
Then we have an equality
\[\widetilde{\rho}(e_i \Lambda/(e_j))=e_i \Gamma/(e_j).\]
Then the assertion (1) follows from Corollary\;\ref{preservingsupport} (4).

By Corollary\;\ref{preservingsupport} (1) and the assertion (1), we obtain the following statement.
\[(\ast)\hspace{5pt} \sttilt_{e_i(\Lambda/(e_j))}\Lambda\cap \sttilt_{e_i\Lambda} \Lambda\neq \emptyset
\text{ if and only if } \sttilt_{e_i(\Gamma/(e_j))}\Gamma\cap \sttilt_{e_i\Gamma} \Gamma\neq \emptyset.\]
%\item $\rho(T'')$ is a minimum element of $\sttilt_{e_i\Gamma} \Gamma$.

Since there is an arrow from $\sigma(i)=i$ to $\sigma(j)=j$, 
the assertion (2) follows from Lemma\;\ref{keylemma} and ($\ast$). 
%Note that
%$\supp(M)=\supp(\mathrm{min}(\sttilt_{M}\Lambda))$ holds for any basic $\tau$-rigid module.
%Thus (3) follows from  Corollary\;\ref{preservingsupport} and (ii).
\end{proof}
%%%%%%%%%%%%%%%%%%%%%%%%%%%%%%%%%%%%%%%%%%%%%%%%%%%%%%%%%%%%%%%%%%%%%%%%%%%%%%%%%%%%%%%%%%%%%%%%
%%%%%%%%%%%%%%%%%%%%%%%%%%%%%%%%%%%%%%%%%%%%%%%%%%%%%%%%%%%%%%%%%%%%%%%%%%%%%%%%%%%%%%%%%%%%%%%%
\begin{remark}
\label{nomultiarrowrem}
If $\alpha\Lambda e_j=e_i\Lambda e_j=e_i\Lambda \alpha $ holds for an arrow $\alpha:i\to j $ on $Q^{\circ}$, then
$\alpha$ is a unique arrow from $i$ to $j$. 
\end{remark}
\subsection{An example}
In this subsection, we consider the following finite, connected, 3-regular lattice:
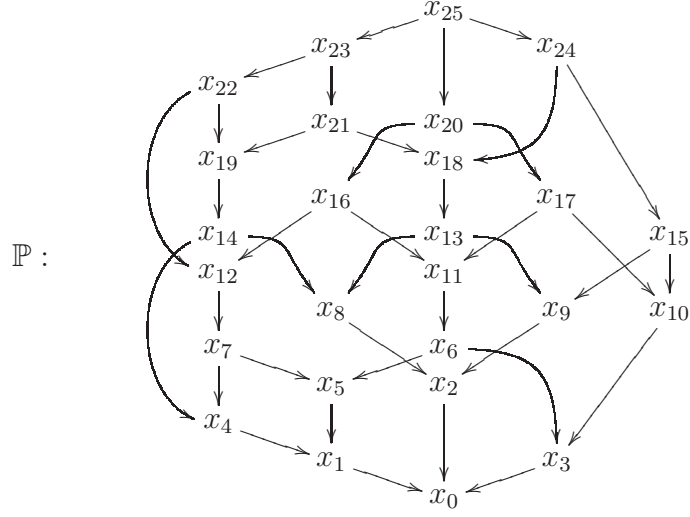
\[
\begin{xy}
(-55,-15)*{\P:},
 (0,-47.5)*[o]+{x_0}="A",
 (-15,-42.5)*[o]+{x_1}="B" ,(0,-32.5)*[o]+{x_2}="C", (15,-42.5)*[o]+{x_3}="D",
 (-30,-37.5)*[o]+{x_4}="E", (-15,-32.5)*[o]+{x_5}="F",
 (0,-27.5)*[o]+{x_6}="G",	(-30,-27.5)*[o]+{x_7}="H",
 (-15,-22.5)*[o]+{x_8}="I",	(15,-22.5)*[o]+{x_9}="J",
 (0,-17.5)*[o]+{x_{11}}="K",(-30,-17.5)*[o]+{x_{12}}="L",(30,-22.5)*[o]+{x_{10}}="M", 
 (0,-12.5)*[o]+{x_{13}}="O",(-30,-12.5)*[o]+{x_{14}}="P",(30,-12.5)*[o]+{x_{15}}="Q", 
 (-15,-7.5)*[o]+{x_{16}}="R",	(15,-7.5)*[o]+{x_{17}}="S",
 (0,-2.5)*[o]+{x_{18}}="T",	(-30,-2.5)*[o]+{x_{19}}="U",
 (0,2.5)*[o]+{x_{20}}="V",	(-15,2.5)*[o]+{x_{21}}="W",
 (-30,7.5)*[o]+{x_{22}}="X",
  (-15,12.5)*[o]+{x_{23}}="Y",(15,12.5)*[o]+{x_{24}}="Z",
  (0,17.5)*[o]+{x_{25}}="1",
 
\ar "B";"A"
\ar "C";"A"
\ar "D";"A"	
\ar "E";"B"
\ar "F";"B"
\ar "G";"F"
\ar @(r,u)"G";"D"
\ar "H";"E"
\ar "H";"F"
\ar "I";"C"
\ar "J";"C"
\ar "K";"G"
\ar "L";"H"
\ar "M";"D"
\ar @(l,l) "P";"E"
\ar @(r,lu) "P";"I"
\ar @(l,ru) "O";"I"
\ar @(r,lu)"O";"J"
\ar "Q";"J"
\ar "Q";"M"
\ar "R";"L"
\ar "R";"K"
\ar "S";"M"
\ar "S";"K"
\ar "T";"O"
\ar "U";"P"
\ar @(l,ru) "V";"R"
\ar @(r,lu)"V";"S"
%\ar "V";"S"
\ar "W";"T"
\ar "W";"U"
\ar @(l,l) "X";"L"
\ar "X";"U"
\ar "Y";"X"
\ar "Y";"W"
\ar "Z";"Q"
\ar @(d,r)"Z";"T"
\ar "1";"V"
\ar "1";"Y"
\ar "1";"Z"
\end{xy}
\vspace{10pt}\]
We assume that $\P$ is isomorphic to the support $\tau$-tilting poset of $\Lambda=KQ/I$ and reconstruct
$\Lambda$ from $\P$ by using results in this section.

Since $x_0$ is the minimum element of $\P$, we may assume that $Q_0=\{1,2,3\}$ with 
\[x_0=0=\Lambda^-,\ x_1=X_1=X_1\oplus P_2^- \oplus P_3^-,\ x_2=X_2=X_2\oplus P_1^- \oplus P_3^-,\ 
 x_3=X_3=X_3\oplus P_1^- \oplus P_2^-.\]
It follows from $x_1\vee x_2=x_{14}$ and Corollary\;\ref{preservingsupport}\;(3) 
 that $\sttilt_{P_3^-} \Lambda=\{x_0,x_1,x_2,x_4,x_8,x_{14}\}$.
Similarly, we obtain $\sttilt_{P_2^-} \Lambda=\{x_0,x_1,x_3,x_5,x_4,x_6\}$ and 
$\sttilt_{P_1^-} \Lambda=\{x_0,x_2,x_3,x_9,x_{10},x_{15}\}$.
Then  Lemma\;\ref{tautiltingfiniteness} and Proposition\;\ref{determiningarrows} give us 
\[\begin{xy}
(0,0)*[o]+{1}="A",(13,0)*[o]+{2}="B",(26,0)*[o]+{3}="C",
(-7,0)*[0]={Q^{\circ}=}="E"
\ar @<2pt> "A";"B"^{\alpha}
\ar @<2pt> "B";"A"^{\alpha^*}
\ar @<2pt> "B";"C"^{\beta}
\ar @<2pt> "C";"B"^{\beta^*}
%\ar @(ld,rd) @<1pt>"C";"A"^{\gamma}
\ar @(ru,lu) @<1pt> "A";"C"^{\gamma^*}
\end{xy}
\]
Since $x_{14}=\stmax(P_3^-)=e_1(\Lambda/(e_3))\oplus e_2(\Lambda/(e_3))$, 
$x_6=\stmax(P_2^-)=e_1(\Lambda/(e_2))\oplus e_3(\Lambda/(e_2))$
 and $x_{15}=\stmax(P_1^-)=e_2(\Lambda/(e_1))\oplus e_3(\Lambda/(e_1))$, we have
\[x_4=X_1\oplus e_1(\Lambda/(e_3)),\ x_8=X_2\oplus e_2(\Lambda/(e_3)),\  x_{14}=\Lambda/(e_3), \] 
 \[x_5=X_1\oplus e_1(\Lambda/(e_2)),\ x_6=\Lambda/(e_2), \]
\[x_9=X_2\oplus e_2(\Lambda/(e_1)),\ x_{10}=X_3\oplus e_3(\Lambda/(e_1)),\  x_{15}=\Lambda/(e_1). \]
In particular, we obtain $x_8=x_2\oplus e_2 (\Lambda/(e_3))$ and $x_9=x_2\oplus e_2(\Lambda/(e_1))$.
Therefore, \[x_{13}=X_2\oplus e_2 (\Lambda/(e_1))\oplus e_2 (\Lambda/(e_1)).\] 
Note that $x_4,x_5\in \dis(x_7)$, we obtain
\[x_7=X_1\oplus e_1(\Lambda/(e_3))\oplus e_1(\Lambda/(e_2)).\]
Then it follows from $x_{16}=x_6 \vee x_7$ and Corollary\;\ref{preservingsupport}\;(3) that
 \[\sttilt_{e_1(\Lambda/(e_2))} \Lambda=\{x_5,x_6,x_7,x_{11},x_{12},x_{16}\}.\] 
Now assume that $x_{20}=Z_i,\ x_{23}=Z_j,\ x_{24}=Z_k$. 
 Corollary\;\ref{preservingsupport}\;(3) implies   
\[\sttilt_{P_i^-} \Lambda\simeq \sttilt_{P_i} \Lambda=\{x_{18},x_{21},x_{23},x_{24},x_{25}\}\]
\[\sttilt_{P_j} \Lambda=\{x_{10},x_{15},x_{17},x_{20},x_{24},x_{25}\}\]
\[\sttilt_{P_k} \Lambda=\{x_{12},x_{16},x_{20},x_{22},x_{23},x_{25}\}.\]
Since there exists an arrow from $3$ to $2$, $e_1(\Lambda/(e_2))\oplus e_3 \Lambda$ should not be $\tau$-rigid. 
In particular, we obtain $i=2$, $j=3$ and $k=1$.
Corollary\;\ref{preservingsupport}\;(3) induces that $[x_4,x_7\vee x_{14}]=\sttilt_{e_1(\Lambda/(e_3))} \Lambda$.
This implies that $x_{12}\in \sttilt_{ e_1(\Lambda/(e_3))} \Lambda\cap \sttilt_{P_1} \Lambda$. Hence, the following equalities
hold by Lemma\;\ref{keylemma}:
\[\gamma^* \Lambda e_3=e_1 \Lambda e_3 =e_1 \Lambda \gamma^*.\]
Similar arguments give us 
\[x_{10}\in \sttilt_{ e_3(\Lambda/(e_1))} \Lambda\cap \sttilt_{P_3} \Lambda\cap \sttilt_{e_3(\Lambda/(e_2))} \Lambda,\]
 \[x_{18}\in \sttilt_{ e_2(\Lambda/(e_1))} \Lambda\cap \sttilt_{P_2} \Lambda\cap \sttilt_{ e_2(\Lambda/(e_3))} \Lambda,\] 
\[x_{12}\in \sttilt_{ e_1(\Lambda/(e_2))} \Lambda\cap \sttilt_{P_1} \Lambda,\]
and it follows from Lemma\;\ref{keylemma} that equalities
\[x \Lambda e_{t(x)}=e_{s(x)} \Lambda e_{t(x)}=e_{s(x)}\Lambda x\]
holds for each $x\in Q_1^{\circ}$. Moreover, we can uniquely determine $x_{\ell}$ ($1\leq \ell\leq 25$). 

Conversely, if
\[ \begin{xy}
	(0,0)*[o]+{1}="A",(13,0)*[o]+{2}="B",(26,0)*[o]+{3}="C",
	(-7,0)*[0]={Q^{\circ}=}="E"
	\ar @<2pt> "A";"B"^{\alpha}
	\ar @<2pt> "B";"A"^{\alpha^*}
	\ar @<2pt> "B";"C"^{\beta}
	\ar @<2pt> "C";"B"^{\beta^*}
	%\ar @(ld,rd) @<1pt>"C";"A"^{\gamma}
	\ar @(ru,lu) @<1pt> "A";"C"^{\gamma^*}
\end{xy}
\] and the equalities
\[x \Lambda e_{t(x)}=e_{s(x)} \Lambda e_{t(x)}=e_{s(x)}\Lambda x\]
hold for each $x\in Q_1^{\circ}$,
then we see that the support $\tau$-tilting poset of $\Lambda=KQ/I$ is isomorphic to $\P$ (see Section\;\ref{subsect4.4}).

We end this section with giving a remark.
$\tau$-tilting finiteness of $\Lambda$ implies that $\sttilt \Lambda$ is finite, 
connected and $|\Lambda|$-regular. 
The converse is not true. In fact, for each $3\leq |\Lambda|=n$, we can construct a finite connected $n$-regular poset $\P$ which is 
not isomorphic to each support $\tau$-tilting poset.  However every finite, 
connected and $2$-regular lattice is realizes as a support $\tau$-tilting poset (see Section\;\ref{subsect4.1}).
   
\section{A question from previous section}
In this section, we introduce a class $\Theta$ of basic algebras satisfying Condition\;\ref{cd} (sect\;\ref{subsec:condition1}) and
 Condition\;\ref{cd2} (sect\;\ref{subsec:condition2}). Then we give a question from results in previous section (sect\;\ref{subsec:question}).
\subsection{First condition}
\label{subsec:condition1} 
For a bound quiver $(Q,I)$ (i.e., $Q$ is a finite quiver and $I$ is an admissible ideal of $KQ$), we set 
\[\begin{array}{llll}
W_j^i(Q,I)&:=&\{w:\text{path from $i$ to $j$ in $Q$ which does not contain a cycle as a subpath}\}\\\\
G_j^i(Q,I)&:=&\{w\in W_j^i\mid w \Lambda e_j=e_i\Lambda w=e_i\Lambda e_j\neq 0\}.\\\\
G(Q,I)&:=&\underset{(i,j)\in Q_0\times Q_0}{\bigsqcup} G_j^i(Q,I)\\
\end{array}
\]
If $\Lambda=KQ/I$, then we denote $W_j^i(\Lambda):=W_j^i(Q,I)$, $G_j^i(\Lambda):=G_j^i(Q,I)$ and $G(\Lambda):=G(Q,I)$.
We consider the following condition for a bound quiver $(Q,I)$.
\begin{condition}
\label{cd}  
 If $e_i (KQ/I) e_j\neq 0$, then $G_j^i(Q,I)\neq \emptyset$. %Moreover, $G$ is closed under taking a subpath.  
 %$x\in e_i\Lambda e_j$  such that 
%$x\Lambda e_j=e_i \Lambda e_j=e_i\Lambda x$. 
\end{condition}
\begin{lemma}
\label{radtorad}
%Assume that $\alpha \Lambda e_{t(\alpha)}=e_{s(\alpha)} \Lambda e_{t(\alpha)}=e_{s(\alpha)} \Lambda \alpha$
 %holds for any arrow $\alpha$ in $Q^{\circ}$.
 Let $0 \neq \lambda\in e_i \Lambda e_j$. %such that $w\Lambda e_j =e_i \Lambda e_j=e_i \Lambda w$. 
If $l\in e_i \Lambda e_i$ and $l'\in e_j\Lambda e_j$ satisfy $l\lambda=\lambda l'$,
then $l\in e_i\rad \Lambda e_i$ if and only if $l'\in e_j \rad \Lambda e_j.$
\end{lemma}
\begin{proof}
Suppose that $l\in e_i\rad \Lambda e_i$ and $l'\not\in e_j\rad \Lambda e_j$.
Then $l'$ is invertible in $e_j\Lambda e_j$. Thus there exists $\epsilon\in e_j \Lambda e_j$
such that
\[l\lambda \epsilon=\lambda.\]
Hence we have $l^m \lambda \epsilon^m=\lambda$ for all $m\in \N$. Since $l\in \rad \Lambda$, we obtain
$\lambda=0$ which leads to a contradiction. Therefore $l\in e_i\rad \Lambda e_i$ implies $l'\in e_j\rad \Lambda e_j$.
Similarly, we see that $l'\in e_j\rad \Lambda e_j$ implies $l\in e_i\rad \Lambda e_i$.
%It is sufficient to check the case that $w=\alpha$ an arrow from $i$ to $j$.
%In this case the assertion follows from the fact that $I$ is an addmissible ideal.  
\end{proof}
\begin{lemma}
	\label{lemmaforeqoncd1} Let $e$ and $f$ be two primitive idempotents of $\Lambda$ satisfying $e\Lambda f\neq 0$. If
	$G^e_f:=\{w\in e\Lambda f\mid w\Lambda f=e\Lambda f=e\Lambda w\}\neq \emptyset$, then we have
\[G^e_f=e\Lambda f\setminus \rad^{\ell+1} \Lambda,\]
where $\ell$ denotes the maximum integer in $\{m\in \Z_{\geq 0}\mid e\Lambda f\subset \rad^m \Lambda \}$.
\end{lemma}
\begin{proof}
Assume that $e\Lambda f\neq 0$ and $G^e_f\neq \emptyset$.

We let $w\in G^e_f$ and $w'\in e\Lambda f\setminus \rad^{\ell+1} \Lambda$.
Since $w\Lambda f=e\Lambda f=e \Lambda w$, there are $\lambda\in f\Lambda f$ and $\lambda'\in e\Lambda e$ such that $w'=w\lambda =\lambda' w$. 
Thus $w\in e\Lambda f\setminus \rad^{\ell+1} \Lambda$ and $\lambda$ (resp. $\lambda'$) is invertible in 
$f \Lambda f$ (resp. $e\Lambda e$). This shows that $w'$ is in $G^e_f$.
\end{proof}
\begin{lemma}
\label{takingsubpath}
Assume that $(Q,I)$ satisfies Condition\;\ref{cd}.
 Then $G(Q,I)$ is closed under taking a subpath.
\end{lemma}
\begin{proof} We let $G=G(Q,I)$ and $G^i_j(Q,I)=G^i_j$.
	
It is sufficient to show that for  $G\ni w=\alpha_1\alpha_2\cdots\alpha_{\ell}$ with $\alpha_1,\dots,\alpha_{\ell}\in Q_1$ and $\ell\geq 2$,
both $\alpha_1\cdots\alpha_{\ell-1}$ and $\alpha_2\cdots\alpha_{\ell}$ are in $G$.
Let $w'=\alpha_1\cdots\alpha_{\ell-1}$ and $(s(w'),t(w'))=(i,j)$ (i.e., $w'$ is a path from $i$ to $j$).
Since $w\neq 0$ in $\Lambda=KQ/I$, $w'$ is also non-zero in $\Lambda$. By hypothesis, there exists $g\in G^i_j$.
Thus there are $l\in e_i\Lambda e_i$ and $l'\in e_j\Lambda e_j$ such that $w'=lg=gl'$.
This implies $w=w'\alpha_{\ell}=lg\alpha_{\ell}$. Since $w\in G$, there exists $l''\in e_i \Lambda e_i$
such that $g\alpha_{\ell}=l''w$. Hence we obtain
\[g\alpha_{\ell}=l''w=l''lg\alpha_{\ell}.\]
This shows that $l''l\in e_i\Lambda e_i\setminus e_i \rad \Lambda e_i$ (otherwise $g\alpha_{\ell}=0$ which leads us to a contradiction). 
In particular, $l$ is invertible in $e_i \Lambda e_i$. 
By using Lemma\;\ref{radtorad}, we also have that $l'$  is invertible in $e_j \Lambda e_j$.
Therefore we obtain $w'\in G$. Similarly, we can check that $\alpha_2\cdots\alpha_{\ell}$ 
is in $G$.
\end{proof}
The following lemma gives equivalent conditions for Condition\;\ref{cd} and then it
is naturally viewed as a condition for arbitrary finite dimensional basic algebras.
\begin{lemma}
\label{eqcd} Let $\Lambda\simeq KQ/I$ and $\Lambda_e:=e \Lambda e$ for any idempotent $e$ of $\Lambda$.

\textup{{\rm (1)}} Following statements are equivalent.
\begin{enumerate}[{\rm (i)}]
\item $(Q,I)$ satisfies Condition\;\ref{cd}.
\item For any pair of projective modules $(P,P')$ of $\Lambda$ with $P\not\simeq P'$, 
there exists $f\in \Hom_{\Lambda}(P,P')$ which generates $\Hom_{\Lambda}(P,P')$
 both as a right $\End_{\Lambda}(P)$-module and as a left $\End_{\Lambda}(P')$-module. 
 \item For any pair of primitive idempotents $(e,f)$ with $ef=fe=0$, there is $w\in e\Lambda f$ such that 
$w\Lambda f=e \Lambda f =e\Lambda w$.
\item If $e, f$ be two primitive idempotents with $ef=0=fe$, then $\Lambda_{e+f}$ satisfies the condition {\rm (iii)} above.
\item If $e, f$ be two primitive idempotents with $ef=0=fe$, then $\sttilt \Lambda_{e+f}$ has one of the following forms.
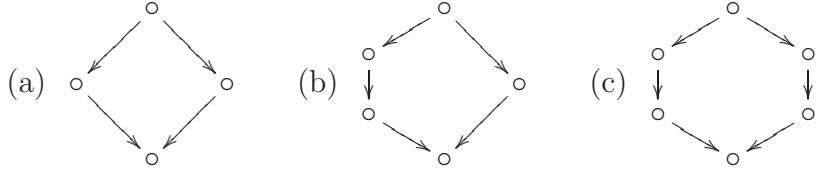
\[\mathrm{( a)}\hspace{5pt}\begin{xy}
 (0,10)*[o]+{\circ}="A", (-10,0)*[o]+{\circ}="B", (10,0)*[o]+{\circ}="C", (0,-10)*[o]+{\circ}="D",
\ar "A";"B"
\ar "A";"C"
\ar "B";"D"
\ar "C";"D"
\end{xy}\hspace{20pt}
\mathrm{( b)}\hspace{5pt}\begin{xy}
(0,10)*[o]+{\circ}="A", (-10,4)*[o]+{\circ}="B", (-10,-4)*[o]+{\circ}="E", (10,0)*[o]+{\circ}="C", (0,-10)*[o]+{\circ}="D",
\ar "A";"B"
\ar "B";"E"
\ar "A";"C"
\ar "E";"D"
\ar "C";"D"
\end{xy}
\hspace{20pt}
\mathrm{( c)}\hspace{5pt}\begin{xy}
(0,10)*[o]+{\circ}="A", (-10,4)*[o]+{\circ}="B", (-10,-4)*[o]+{\circ}="E", (10,4)*[o]+{\circ}="C", (10,-4)*[o]+{\circ}="F", (0,-10)*[o]+{\circ}="D",
\ar "A";"B"
\ar "B";"E"
\ar "A";"C"
\ar "C";"F"
\ar "E";"D"
\ar "F";"D"
\end{xy}
\]
\end{enumerate}
In particular, Condition\;\ref{cd} is closed under isomorphism.

\textup{{\rm (2)}} Under the condition {\rm (iii)}, we have $Q^{\circ}_1\subset G(Q,I)$.
In particular, $Q^{\circ}$ has no multiple arrow.
 \end{lemma} 
 \begin{proof}
 First, we show (2). If $\Lambda$ satisfies (iii),
 then $KQ/I$ also satisfies (iii). Thus we may assume that $\Lambda=KQ/I$.
 Let $\alpha$ be an arrow from $i$ to $j$ with $i\neq j$. Since $e_i \Lambda e_j\neq 0$, there is $w\in e_i \Lambda e_j$
 such that $w \Lambda e_j=e_i \Lambda e_j =e_i \Lambda w$.
Hence there are $l\in e_i\Lambda e_i$ and $l'\in e_j \Lambda e_j$
such that
\[\alpha=lw=wl'.\]
If either $l$ or $l'$ is in $\rad \Lambda$, then $\alpha\in \rad^2 \Lambda$.
This is a contradiction. Thus $l$ (resp. $l'$) is invertible in $e_i \Lambda e_i$
(resp. $e_j \Lambda e_j$). In particular, we have the assertion (2). 
  
 We prove (1). We show that conditions (i), (ii) and (iii) are equivalent.
Since $\Lambda$ is basic, $ef=0=fe$ implies $e\Lambda\not\simeq f\Lambda$. Then the implications
\[\mathrm{(i)}\Rightarrow \mathrm{(ii)}\Rightarrow \mathrm{(iii)}\]
is clear.  
 %We first assume that $\Lambda=KQ/I$ satisfies Condition\;\ref{cd}.
% We assume (i). Let $(e,f)$ be a pair of primitive idempotents of $\Lambda$ such that $ef=0=fe$ and $e\Lambda f\neq 0$.
%Then there are $i,j\in Q_0$ such that $e\Lambda\simeq e_i \Lambda$ and $f\Lambda\simeq e_j \Lambda$.
%Hence, there are $\lambda_1\in e\Lambda e_i$, $\lambda_1^-\in e_i\Lambda e$
%and  $\lambda_2\in f\Lambda e_j$ and $\lambda_2^-\in e_j\Lambda f$ such that
%\[\lambda_1\lambda_1^-=e,\ \lambda_1^-\lambda_1=e_i,\ \lambda_2\lambda_2^-=f,\ \lambda_2^-\lambda_2=e_j.\]
 %In particular, we have
%\[e\Lambda f=\lambda_1 e_i\Lambda e_j \lambda_2^-.\]
%Since $(Q,I)$ satisfies Condition\;\ref{cd}, there exists $w\in e_i \Lambda e_j$ such that
%\[w\Lambda e_j=e_i \Lambda e_j=e_i \Lambda w.\]
%Thus we obtain
% \[\lambda_1 w \Lambda f=\lambda_1 w \Lambda \lambda_2^-=\lambda_1 e_i \Lambda e_j \lambda_2^-=\lambda_1 \Lambda w %\lambda_2^-=e\Lambda w\lambda_2^-.\]
%Therefore we have (ii).
 We suppose that the condition (iii) holds.
Since $\Lambda\cong KQ/I$ satisfies (iii), we may assume that $\Lambda=KQ/I$. We consider two vertices $i\neq j\in Q_0$ 
 such that $e_i \Lambda e_j\neq 0$
 and show that $G_j^i:=G^i_j(\Lambda)\neq \emptyset$.
 Let $w\in e_i\Lambda e_j$ such that $w\Lambda e_j=e_i \Lambda e_j=e_i\Lambda w$.
 We write $w=a_1x_1+\cdots+a_tx_t$, where
 $a_p\in K\setminus\{0\}$ and $x_p$ is a path from $i$ to $j$ of $Q$ which is not in $I$.
 Then we can take $l_p\in e_i \Lambda e_i$ and $l'_p\in e_j \Lambda e_j$
 satisfying 
 \[l_p w=a_px_p=w l'_p.\] Let $l:=\sum l_p$ and $l':=\sum l'_p$.
 Then we have
 \[lw=w=wl'.\]
 Since $w\neq 0$, we see that $l\in e_i\Lambda e_i\setminus e_i\rad \Lambda e_i$. 
 Thus there is $p$ such that $l_p$ is invertible in $e_i \Lambda e_i$.
 Without loss of generality, we may assume that $l_1$ is invertible in $e_i \Lambda e_i$.
 Since  $l_1 w=a_1x_1=wl'_1$, Lemma\;\ref{radtorad} implies that $l'_1$ 
 is invertible in $e_j \Lambda e_j$. In particular, a path $x_1$ satisfies
 \[(\dagger)\ x_1 \Lambda e_j=e_i \Lambda e_j= e_i\Lambda x_1.\]
 Now suppose that the path $x_1$ contains a cycle.
 By the assertion (2), 
 \[\alpha \Lambda e_{t(\alpha)}=e_{s(\alpha)} \Lambda e_{t(\alpha)}=e_{s(\alpha)} \Lambda \alpha\]
 holds for any arrow $\alpha$ in $Q^{\circ}$. Hence Lemma\;\ref{radtorad} implies that
  there exists $\epsilon\in e_i\rad\Lambda e_i$ and
 $w'\in W^i_j(\Lambda)$ such that $x_1=\epsilon w'$. By $(\dagger)$, we have $\epsilon'\in e_i\Lambda e_i$
 such that
 \[\epsilon'\epsilon w'=w'.\]
 This gives $w'=0$ which leads to a contradiction. 
 Thus the path $x_1$ is in $G^i_j$. Therefore condition (i), (ii) and (iii) are equivalent.
 
 Let $e'$ and $f'$ be primitive idempotents of $\Lambda_{e+f}$ such that $e'f'=0=f'e'$.
 Since $e' \Lambda_{e+f} f'=e' \Lambda f'$ holds and $e'$ and $f'$ also are primitive idempotents of $\Lambda$,
  condition (iii) and (iv) are equivalent.
 
Finally, we show that (iv) and (v) are equivalent. 
%For two distinct vertices $i$ and $j$, we 
Let $(e,f)$ be a pair of primitive idempotents with $ef=0=fe$. 
We take a quiver $Q(e,f)$ and an admissible ideal $I(e,f)$ of $KQ(e,f)$ such that
  $\Lambda_{e+f}\cong KQ(e,f)/I(e,f)$. 
 Since $\Lambda_{e+f}$
  satisfies condition (iii) if and only if
   $(Q(e,f),I(e,f))$ satisfies the condition (i).
 Therefore \cite[Proposition\;3.2]{AK} implies that the condition (iv) and (v) are equivalent.
 (We only note that $\sttilt \Lambda_{e,f}$ has the form (a) if and only if 
 $\Lambda_{e,f}=e\Lambda e\times f \Lambda f$ 
 or equivalently $e\Lambda f=0=f \Lambda e$.)  
   \end{proof}
From now on, we say that a basic algebra $\Lambda$ satisfies Condition\;\ref{cd} if 
Condition\;\ref{cd} holds for some (thus every) $(Q,I)$ satisfying $\Lambda\cong KQ/I$.
\subsection{Second condition}
\label{subsec:condition2}
For a quiver $Q$, we set
\[\begin{array}{cll} 
\mathsf{sub}(Q)&:=&\text{ the set of all connected full subquivers of }Q.\\  
\calP(Q)&:=&\{\mu=\{Q^{1},\cdots, Q^{\ell}\}
\mid \ell\in \Z_{\geq 1}, Q^a\in \mathsf{sub}(Q), Q_0=\sqcup Q^a_0\}.\\
\end{array}\]
Let $\mu=\{Q^{1},\cdots, Q^{\ell}\}\in \calP(Q)$. We define a quiver $\mathbf{Q}^{\mu}$ as follows:
\begin{itemize}
\item $\mathbf{Q}^{\mu}_0:=\mu$.
\item For each pair $(a\neq b)$ of $\{1,\dots,\ell\}$, we put 
$t_{(a,b)}:=\#\{\alpha\in Q_1\mid s(\alpha)\in Q^a_0, t(\alpha)\in Q^b_0 \}$ and
draw $t_{(a,b)}$ arrows from $Q^a$ to $Q^b$. 
 \end{itemize}
For $\Lambda\cong KQ/I$ and $\mu=\{Q^1,\dots, Q^{\ell}\}$,
 let $e_a^{\mu}$ be the idempotent of $\Lambda$ corresponding to $Q^a_0$ and
$\Lambda_{\mu}^{a,b}:=\Lambda/(1-e^{\mu}_a-e^{\mu}_b)$. 
We note that
if $\mathbf{Q}^{\mu}$ is a tree quiver and there exists an arrow $Q^a\to Q^b$ in $\mathbf{Q}^{\mu}$ , then 
\[\Lambda_{\mu}^{a,b}\cong \Lambda_{e^{\mu}_a+e^{\mu}_b}.\]

We are ready to state Condition\;2
\begin{condition}
\label{cd2} 
There exists $\mu=\{Q^1,\dots,Q^{\ell}\}\in \calP(Q)$ such that $\mathbf{Q}^{\mu}$ is a tree 
quiver and  $\sttilt \Lambda_{\mu}^{a,b}$ is a lattice for each $a\neq b\in \{1,\dots,\ell\}$. 
\end{condition}
\begin{remark} We give some remarks for Condition\;\ref{cd2}.
	\begin{enumerate}[{\rm (1)}]
		\item We recall the construction of the Gabriel quiver of a basic algebra $\Lambda$.
		Let $\mathbf{e}=\{e_1,e_2,\cdots,e_n\}$ be a complete set of primitive orthogonal idempotents of $\Lambda$.
		Then the Gabriel quiver $Q=Q_{\mathbf{e}}$ of $\Lambda$ is defined as follows:
		\begin{itemize}
			\item $Q_0=\{1,2,\dots,n\}$.
			\item Draw $t(i,j)$-th arrows from $i$ to $j$, where $t(i,j):=\dim_K e_i (\rad \Lambda/\rad^2 \Lambda) e_j$.
		\end{itemize}
		It is well-known that $Q$ does not depend on the choice of a complete set of primitive orthogonal idempotents of $\Lambda$.
		More precisely, if $\mathbf{f}=\{f_1,\dots,f_n\}$ is another complete set of primitive orthogonal idempotents of $\Lambda$ such that
		$e_i\Lambda\simeq f_i \Lambda$, then $Q_{\mathbf{e}}=Q_{\mathbf{f}}$ holds (see \cite[I\hspace{-.1em}I.3]{ASS} for example). 
		Furthermore, we have $\Lambda (\sum_{k\in V} e_k) \Lambda=\Lambda(\sum_{k\in V} f_k)\Lambda$ for any $V\subset \{1,2,\dots,n\}$.
		Hence Condition\;\ref{cd2} does not depend on the choice of a bound quiver $(Q,I)$ of $\Lambda$. 
		\item If $\Lambda=KQ$ is $\tau$-tilting finite, then $\Lambda$ satisfies Condition\;\ref{cd2} via $\mu=\{Q\}$.
		\item If $\Lambda=KQ$ is a tree quiver algebra, then $\Lambda$ satisfies Condition\;\ref{cd2} via $\mu=\{Q^{i}=\overset{i}{\circ} \mid i\in Q_0\}$. 
	\end{enumerate}
\end{remark}
\subsection{Examples of algebras in $\Theta$}
As we mentioned in the beginning of this section,  we define a class $\Theta$ of basic  algebras as follows:
\[\Theta:=\{\Lambda\mid \text{$\Lambda$ satisfies Condition\;\ref{cd} and Condition\;\ref{cd2}}\}.\]
Since the definition of $\Theta$ is a little complicated, we give some examples.
\begin{example} \textup{(1)} The following algebras are in $\Theta$.
\begin{enumerate}[{\rm (i)}]
\item Tree quiver algebras.
\item Preprojective algebras of type $A$.
\item Nakayama algebras.
\item Generalized Brauer tree algebras. 
\item $\tau$-tilting finite algebras with radical square zero.
\end{enumerate}

\textup{(2)} Let $\Lambda$ be the following bound quiver algebra $KQ/I$:
\[\begin{xy}
(0,0) *[o]+{1^{(1)}}="A", (14,0) *[o]+{2^{(1)}}="B", (28,7) *[o]+{1^{(2)}}="C", 
(42,7) *[o]+{2^{(2)}}="D", (28,-7) *[o]+{1^{(3)}}="E", (42,-7) *[o]+{2^{(3)}}="F",
(56,-7) *[o]+{1^{(4)}}="G", (70,-7) *[o]+{2^{(4)}}="H",
(0,-14) *[o]+{1^{(5)}}="I", (14,-14) *[o]+{2^{(5)}}="J",
(-14,-5) *[o]+{Q:}="K", (110,-5)*[o]+{I=\langle x_ay_a+y_ax_a\mid a\in\{1,2,3,4,5\} \rangle}
\ar @<3pt> "A";"B"^{x_1}
\ar @<3pt> "B";"A"^{y_1}
\ar @<3pt> "C";"D"^{x_2}
\ar @<3pt> "D";"C"^{y_2}
\ar @<3pt> "E";"F"^{x_3}
\ar @<3pt> "F";"E"^{y_3}
\ar @<3pt> "G";"H"^{x_4}
\ar @<3pt> "H";"G"^{y_4}
\ar @<3pt> "I";"J"^{x_5}
\ar @<3pt> "J";"I"^{y_5}
\ar  "B";"C"^{\alpha}
\ar  "B";"E"^{\beta}
\ar  "F";"G"^{\gamma}
\ar  "J";"E"^{\delta}
\end{xy}
\]
It is obvious that $\Lambda$ satisfies Condition\;\ref{cd}.
Let $\mu=\{Q^1,Q^2,Q^3,Q^4,Q^5\}\in \calP(Q)$ with $Q^a_0=\{1^{(a)},2^{(a)}\}$ for each $a\in \{1,2,3,4,5\}$.
Then $\mathbf{Q}^{\mu}$ is the following tree quiver.
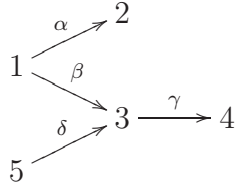
\[\begin{xy}
(0,0) *[o]+{1}="A", (14,7) *[o]+{2}="B", (14,-7) *[o]+{3}="C", 
(28,-7) *[o]+{4}="D", (0,-14) *[o]+{5}="E", 
\ar  "A";"B"^{\alpha}
\ar  "A";"C"^{\beta}
\ar  "C";"D"^{\gamma}
\ar  "E";"C"^{\delta}
\end{xy}\]
Since $\Lambda_{\mu}^{a,b}$ is a factor algebra
of the preprojective algebra of type $A_4$ for each  pair $(a\neq b)$
of $ \{1,2,3,4,5\}$,
we have that $\Lambda_{\mu}^{a,b}$ is $\tau$-tilting finite. In particular, 
$\Lambda$ satisfies Condition\;\ref{cd2}.  
\end{example}
\begin{remark} If $\Lambda$ is either a tree quiver algebra or a preprojective algebra,
then Condition\;\ref{cd} is equivalent to the condition (b) in Theorem\;\ref{refe}.
\end{remark}
\subsection{A question for $\Theta$}
\label{subsec:question}
From now on, for an algebra $\Lambda=KQ/I$ and $i\neq j\in Q_0$, we set 
\[\Lambda_{i,j}:=\Lambda_{e_i+e_j}=(e_i+e_j)\Lambda(e_i+e_j).\] 
\begin{lemma}
\label{preservingnomultiarrow}
Let $\Lambda=KQ/I$ and $\Gamma=KQ'/I'$.
 Assume that $\sttilt \Lambda\simeq \sttilt \Gamma$ and $\Gamma$ satisfies Condition\;\ref{cd}. Then
 $Q$ has no multiple arrow. 
\end{lemma}
\begin{proof}
Let $\rho:\sttilt \Lambda\simeq \sttilt \Gamma$ be a poset isomorphism and $\sigma:Q_0\to Q'_0$ a 
bijection considered in Corollary\;\ref{preservingsupport}.
We may assume that $Q_0=Q'_0$ and $\sigma$ is the identity.
Let $i\neq j$ be in $Q_0=Q'_0$. Suppose that there are two arrows from $i$ to $j$
in $Q$. 

By Corollary\;\ref{preservingsupport}, we have an isomorphism
\[\sttilt \Lambda/(1-e_i-e_j)\simeq \sttilt \Gamma/(1-e_i-e_j).\]
Since $\Gamma$ satisfies Condition\;\ref{cd}, $\Gamma/(1-e_i-e_j)$ also
satisfies Condition\;\ref{cd}. %It follows from \cite[Proposition\;3.2]{AK} that
 Hence $\sttilt \Lambda/(1-e_i-e_j)\simeq \sttilt \Gamma/(1-e_i-e_j)$ has one of the forms in Lemma\;\ref{eqcd} (v).
This contradicts the fact that there are two arrows from $i$ to $j$ in $Q$. 
\end{proof}
We now sate a main result of this section.
\begin{corollary}
\label{maincor}
Let $\Gamma\cong KQ'/I'\in \Theta$. 
 %Assume that either $\sttilt \Gamma$ is a lattice, or $Q'$ is a tree quiver, holds. 
Then $\Lambda\cong KQ/I\in \T(\Gamma)$ only if there is a quiver isomorphism $\sigma:Q^{\circ}\to (Q')^{\circ}$
 satisfying the following conditions. 
  \begin{enumerate}[{\rm (a)}]
 \item $\supp(e_{\sigma(i)} \Gamma)=\sigma(\supp e_i\Lambda)$ for any $i\in Q_0$.
 \item $G(\Gamma)=\sigma (G(\Lambda))$.   
  \end{enumerate}
  
  Moreover, $\Lambda$ is also in $\Theta$.
  \end{corollary}
\begin{proof} We may assume $\Lambda=KQ/I$, $\Gamma=KQ'/I'$ and there is a poset isomorphism
 $\rho:\sttilt \Lambda\stackrel{\sim}{\to} \sttilt \Gamma$. 
Let $\sigma:Q_0\to Q'_0$ be as in 
Corollary\;\ref{preservingsupport}. Then Proposition\;\ref{determiningarrows} and Lemma\;\ref{preservingnomultiarrow}
imply that $\sigma$ is extended as a quiver isomorphism
\[\sigma: Q^{\circ}\stackrel{\sim}{\to} (Q')^{\circ}.\]
We may assume that $Q^{\circ}=(Q')^{\circ}$ and $\sigma$
is the identity and put $G=G(\Lambda)$, $G'=G(\Gamma)$.
Then the condition (a) follows from Lemma\;\ref{projsupport}. 

We first consider the case that $\sttilt \Gamma$ is a lattice.
%Assume that $Q'$ is a tree quiver.
%Then by combining Proposition\;\ref{determiningsupport} and
 %Proposition\;\ref{keyproposition}, we can easily check the assertion.
\begin{claim}
\label{claimformaincor} 
If $\sttilt\Gamma$ is a lattice, then we have $G=G'$.
\end{claim}
%We may assume that $Q_0=Q'_0$ and $\sigma(i)=i$ for any $i\in Q_0$ and
\begin{pfclaim} %Assume that there is an arrow from $i$ to $j$ in $Q^*=(Q')^*=(Q')^{\circ}$. 
%Then Proposition\;\ref{keyproposition} implies that $G^i_j(\Lambda)\neq 0$. In particular,
%we have that $Q^{\circ}$ has no multiple arrow. Hence $Q^*=Q^{\circ}$. 
Let $\alpha$ be an arrow 
from $i$ to $j$ in $Q^{\circ}=(Q')^{\circ}$.
Then $G_j^i=(G_j^i)'=\{\alpha\}$
follows from Proposition\;\ref{keyproposition} (we remark that $I$ and $I'$ are admissible).

Suppose that $G'\not\subset G$. %By considering a length, Lemma\;\ref{takingsubpath} implies that there exists
Take a path $w\in G'\setminus G$ whose length is minimum in $G'\setminus G$. Then the length of $w$ is at least $2$.
Let $w:x=x_0\to x_1\to\cdots\to x_{\ell}=y\stackrel{\alpha}{\to}z $ and $w':x=x_0\to x_1\to \cdots\to x_{\ell}=y$
(i.e. $w=w'\alpha$). Lemma\;\ref{takingsubpath} gives us that $w'\in G\cap G'$. 
Note that the following equality follows from Lemma\;\ref{projsupport}:
\[\supp(e_x(\Gamma/(e_y)))=\supp(e_x (\Lambda/(e_y))).\]
Since $w\in G'$, we have $e_x\Gamma e_y \Gamma e_z\subset e_x\Gamma e_z=w \Lambda e_z \subset e_x\Gamma e_y \Gamma e_z$
In particular, we obtain $e_x \Gamma e_z=e_x \Gamma e_y \Gamma e_z$ and  $e_x (\Gamma/(e_y)) e_z=0$.
This shows $z\not\in \supp(e_x(\Gamma/(e_y)))=\supp(e_x (\Lambda/(e_y)))$.
 In particular, we have  $e_x (\Lambda/(e_y)) e_z=0$ and 
\[e_x\Lambda e_z=e_x\Lambda e_y\Lambda e_z.\]
Since $e_x\Lambda e_z\neq 0$ (by Lemma\;\ref{projsupport}), $w=w'\alpha\in G$ follows from
 $\{w',\alpha\}\subset G$. In fact we have
 \[e_x \Lambda w'=e_x\Lambda e_y\Lambda e_z=w'\Lambda e_z.\]
This is a contradiction. Hence, we have $G'\subset G$. Since $e_i \Lambda e_j\neq 0$ if and only if $e_i \Gamma e_j\neq 0$,
 $\Lambda$ satisfies 
Condition\;\ref{cd}. Since $\sttilt \Lambda$ is a lattice, $\Lambda$ is in $\Theta$. 
 Therefore, we also have $G\subset G'$ by using the  above argument.   
\end{pfclaim}

We consider arbitrary $\Gamma\in \Theta$. Let $\mu'=\{(Q')^1,\dots, (Q')^{\ell}\}\in \calP(Q')$
such that $\mathbf{Q}^{\mu'}$ is a tree quiver and $\sttilt \Gamma_{\mu'}^{a,b}$ is a lattice
for any $a\neq b\in \{1,\dots,\ell\}$. We put $\mu=\{Q^1,\dots, Q^{\ell}\}\in \calP(Q)$
such that $Q^a_0=(Q')^a_0$ for any $a\in \{1,\dots,\ell\}$.
Then by Corollary\;\ref{preservingsupport} (1), $\Lambda$ satisfies Condition\;\ref{cd2} via $\mu$. 

Thus it is sufficient to show that $G=G'$.
\begin{claim}
\label{claimformaincor2}	
Assume that there is an arrow $Q^a\to Q^b$ in $Q^{\mu}$ and $w\in (G')^i_j$ with $i,j\in Q^a_0\cup Q^b_0$.
Then $w\in G^i_j$.	
\end{claim}
\begin{pfclaim}
By corollary\;\ref{preservingsupport} (1), $\rho$ induces a poset isomorphism  
\[\sttilt \Lambda_{\mu}^{a,b}\simeq \sttilt \Gamma_{\mu'}^{a,b}.\]
Since $\Gamma_{\mu'}^{a,b}$ is a lattice (thus $\Gamma_{\mu'}^{a,b}\in \Theta$) and 
\[e_i \Gamma e_j =e_i \Gamma_{\mu'}^{a,b} e_j,\ 
e_i \Lambda e_j =e_i \Lambda_{\mu}^{a,b} e_j,\]
it follows from Claim\;\ref{claimformaincor} that $w\in G^i_j$.
\end{pfclaim}

Let $i$ and $j$ be two vertices of $Q'_0=Q_0$ with $e_i \Lambda e_j\neq 0(\Leftrightarrow e_i \Gamma e_j\neq 0)$ and $w\in (G')_j^i$.
We claim that $w\in G_j^i$.
Since $\mathbf{Q}^{\mu}$ is a tree quiver, there exists a unique path
\[Q^{a}=Q^{a_0}\to Q^{a_1}\to \cdots \to Q^{a_{t}}=Q^b\]
in $\mathbf{Q}^{\mu}$ such that $i\in Q^{a}_0$ and $j\in Q^b_0$.
If $a=b$, then $w\in G^i_j$ follows from Claim\;\ref{claimformaincor2}.
Hence we may assume  $a\neq b$. Let $\alpha_s$ be the arrow in $Q^{\circ}=(Q')^{\circ}$  corresponding to
$Q^{a_s}\to Q^{a_{s+1}}$. We denote by $j_s$ starting point of $\alpha_s$ and
 by $i_{s+1}$ the target point of $\alpha_s$. We note that $j_s\in Q^{a_s}$ and $i_{s+1}\in Q^{a_{s+1}}$.
 We also note that \[\begin{array}{lll}
 e_i \Lambda e_j&=&e_{i}\Lambda e_{j_0} \alpha_0 e_{i_1}
 \Lambda e_{j_1}\cdots \alpha_{t-1} e_{i_t} \Lambda e_j\\\\
 e_i \Gamma e_j&=&e_i\Gamma e_{j_0} \alpha_0 e_{i_1}
 \Gamma e_{j_1}\cdots \alpha_{t-1} e_{i_t} \Gamma e_j.\\
 \end{array}
 \] 
Then there is a unique description $w_0\alpha_0 w_1\cdots \alpha_{t-1} w_t$ of $w$, where $w_s$ is a path 
 in $(Q^{a_s})^{\circ}=((Q')^{a_s})^{\circ}$. By Lemma\;\ref{takingsubpath}, we have that 
$w_s$, $w_s\alpha_s$ and $\alpha_s w_{s+1}$
are in $G'$. Then by Claim\;\ref{claimformaincor2}, $w_s$, $w_s\alpha_s$ and $\alpha_s w_{s+1}$  are  also in $G$. 
Therefore, we obtain %there are $l_0\in e_i\Lambda e_i$ and $l_s\in e_{i_s}\Lambda e_{i_s}$ such that 
%\[x=l_0w_0\alpha_0 l_1 w_1\cdots\alpha_{t-1}l_{t-1}w_t.\]
\[\begin{array}{lll}
e_i \Lambda e_j&=&e_{i}\Lambda w_0 \alpha_0 \Lambda w_1\alpha_1 \cdots \Lambda w_{t-1}\alpha_{t-1} \Lambda w_t\\
&=&w_0\alpha_0 w_1\cdots \alpha_{t-1} w_t \Lambda e_j=w \Lambda e_j.\\
\end{array}\]
Similarly, we obtain
\[e_i \Lambda e_j=e_i \Lambda w.\]
Thus we have $G'\subset G$. Hence $\Lambda$  is also in $\Theta$.
In particular, we obtain $G\subset G'$ by using the same argument.  
 \end{proof} 
Let $(Q,I)$ and $(Q'I)$ be bound quivers.
Then we denote by $(Q,I)\sim(Q'I)$ if  there is a quiver isomorphism $\sigma:Q^{\circ}\to (Q')^{\circ}$ 
satisfying (a), (b) of Corollary\;\ref{maincor}.
\begin{lemma}
\label{lemmafort'}
Let $\Lambda=kQ/I$ and $\Gamma=kQ'/I'$. If $\Lambda\cong \Gamma$ and $\Lambda$ satisfies Condition\;\ref{cd},
then $(Q,I)\sim(Q',I')$.
\end{lemma}
\begin{proof}
Let $\varphi:\Gamma\stackrel{\sim}{\to}\Lambda$.
We may assume that $Q=Q'$ and $\varphi(e_i\Gamma)\simeq e_i \Lambda$ for each $i\in Q_0$.
Let $\varphi(e_i)=e'_i=x_ie_i+\sum_{j\neq i} y^{(i)}_j e_j +l_i$ with $x,y^{(i)}_j\in K$ and $l_i\in \rad \Lambda$.
$e_i \Lambda\simeq e'\Lambda$ implies that there is $ \lambda\in e'\Lambda e_i$ such that
$e'\Lambda=e'\lambda e_i \Lambda$. Then there is $\lambda'$ such that
 $e'=e'\lambda e_i \lambda'$. This shows  $y^{(i)}_j=0$, $x_i=1$ and $l_i\in \Lambda e_i \Lambda\cap \rad \Lambda$.
 
 We let $w=(i=i_0\stackrel{\alpha_1}{\to} i_1\stackrel{\alpha_2}{\to}\cdots \stackrel{\alpha_m}{\to} i_m=j)\in G^i_j(\Lambda)$ with
 $w\in \rad^{\ell} \Lambda\setminus \rad^{\ell+1} \Lambda$.
  We show that $w$ is also in $G^i_j(\Gamma)$. By Lemma\;\ref{lemmaforeqoncd1}, it is sufficient to prove
$\varphi(w)\in G^{e'_i}_{e'_j}=e'_i\Lambda e'_j\setminus \rad^{\ell+1} \Lambda$. 

Let $\lambda_0,\lambda'_0,\lambda_1,\lambda'_1,\dots,\lambda_m,\lambda'_m\in \Lambda$ and consider
\[w':=\lambda_0 e_{i_0} \lambda'_0\varphi(\alpha_1)
\lambda_1 e_{i_1} \lambda'_1\varphi(\alpha_2)\cdots \lambda_{m-1} e_{i_{m-1}} \lambda'_{m-1}\varphi(\alpha_{m-1})\lambda_m e_{i_m} \lambda'_m. \] 
Since $\Lambda$ satisfies Condition\;\ref{cd}, $Q_1^{\circ}\in G$ and we can describe $e_{i_{t-1}} \lambda'_{t-1}\varphi(\alpha_t)
\lambda_t e_{i_t}= \alpha_{i_t} a_t$ for some $a_t \in e_{i_t}\Lambda e_{i_t}$.  
Now assume 
$\{\lambda_0,\lambda'_0,\lambda_1,\lambda'_1,\dots,\lambda_m,\lambda'_m\in \Lambda\}\cap \rad \Lambda\neq \emptyset$.
 Then it follows from Lemma\;\ref{radtorad} that $w'\in \rad^{\ell+1}\Lambda$. Hence we obtain
 \[\begin{array}{lll}
 \varphi(w) &=&(e_{i_0}+l_{i_0})\varphi(\alpha_1)(e_{i_1}+l_{i_1})\varphi(\alpha_2)
 \cdots (e_{i_{m-1}}+l_{i_{m-1}})\varphi(\alpha_m)(e_{i_m}+l_{i_m})\\
 &=&e_{i_0}\varphi(\alpha_1)e_{i_1}\varphi(\alpha_2) \cdots e_{i_{m-1}}\varphi(\alpha_m)e_{i_m}+r\\
                  &=& a w+r,\\ 
 \end{array}\]
 for some $r\in \rad^{\ell+1} \Lambda$ and $a\in e_i \Lambda e_i\setminus e_i \rad \Lambda e_i$.
In particular, $\varphi(w)\in \rad^{\ell}\Lambda\setminus \rad^{\ell+1} \Lambda$. 
\end{proof}
We now define an equivalent relation $\sim$ on class of basic algebras satisfying Condition\;\ref{cd}: Let $\Lambda\cong KQ/I,\ \Gamma\cong KQ'/I'$. 
\[\Lambda\sim \Gamma\Leftrightarrow (Q,I)\sim (Q',I').\]
We set 
\[\T'(\Gamma):=\{\Lambda\mid \Lambda\sim \Gamma\}.\]
Then we  have the following question.
  \begin{question}
  \label{cj} %\begin{enumerate}[{\rm (1)}]
%\item 
%Does $-\otimes_{\Lambda} \overline{\Lambda}$ induce a poset isomorphism
%from $\sttilt \Lambda$ to $\sttilt \overline{\Lambda}$?
%\item
Does $\T(\Gamma)=\T'(\Gamma)$ hold for any $\Gamma\in \Theta$?  \end{question}

\begin{theorem}
\label{maintheorem2}
Question\;\ref{cj} holds true if one of the following statements holds.
\begin{enumerate}[{\rm (i)}]
\item $\Gamma$ is a tree quiver algebra. {\cite{AK}}
\item $\Gamma$ is a preprojective algebra of type $A$. {\cite{K2}}
\item $\Gamma$ is a Nakayama algebra. $[${\rm Section\;\ref{subsect4.2}}$]$
\item $\Gamma$ is a Brauer tree algebra. $[${\rm Section\;\ref{subsect4.3}}$]$
%\item $\Gamma$ is an algebra with $\rad^2 \Gamma=0$. 
\item $|\Gamma|\leq 3$. $[${\rm Section\;\ref{subsect4.1},\  Section\;\ref{subsect4.4}}$]$
\end{enumerate}
\end{theorem}
%It is also interesting to consider following question.
%\begin{question} 
%Under the setting in Lemma\;\ref{projsupport}, is $\sttilt \Lambda_{i,j}$
%isomorphic to $\sttilt \Gamma_{\sigma(i),\sigma(j)}$ for any $i\neq j$?
%\end{question}
%If $\Gamma\in \Theta$, then above question holds true.  
%\begin{lemma}
%\label{lemmaforq2}
%Let $\Lambda$, $\Gamma$, $\rho$ and $\sigma$ be as in  Lemma\;\ref{projsupport}.
%If $\Gamma\in \Theta$, then $\sttilt \Lambda_{i,j}$
%isomorphic to $\sttilt \Gamma_{\sigma(i),\sigma(j)}$ for any $i\neq j$.
%\end{lemma}  
%\begin{proof}
%By Lemma\;\ref{eqcd} and Corollary\;\ref{maincor}, $\Lambda_{i,j}$
%and $\Gamma_{\sigma(i),\sigma(j)}$ satisfy Condition\;\ref{cd} for any $i\neq j$.
%Note that
%\[e_p \Lambda_{i,j} e_q =e_p \Lambda e_q\ (p,q\in \{i,j\})\]
%holds for any $i\neq j$. Thus the assertion follows from Proposition\;\ref{projsupport} and 
%\cite[Proposition\;3.2]{AK}.
%\end{proof}
%%%%%%%%%%%%%%%%%%%%%%%%%%%%%%%%%%%%%%%%%%%%%%%%%%%%%%%%%%%%%%%%%%%%%%%%%%%%%%%%%%%%%%%%%%%%%%%
%%%%%%%%%%%%%%%%%%%%%%%%%%%%%%%%%%%%%%%%%%%%%%%%%%%%%%%%%%%%%%%%%%%%%%%%%%%%%%%%%%%%%%%%%%%%%%%%
\section{Reduction to mimimal factor algebras in $\T'(\Lambda)$ and its applications.}
Assume that $\Lambda=KQ/I$ satisfies Condition\;\ref{cd}. Let $J:=\Lambda (\sum_{i\in Q_0} e_i \rad \Lambda e_i )\Lambda$
be a two-sided ideal of $\Lambda$ and $\overline{\Lambda}:=\Lambda/J$. 
For an element $\lambda\in \Lambda$, we set $\bar{\lambda}:=\lambda+J\in \overline{\Lambda}$.
We note that $\overline{\Lambda}$ also satisfies Condition\;\ref{cd}.
\begin{lemma}
\label{gpreserve}\begin{enumerate}[{\rm (1)}]
\item Let $g\in e_i \Lambda e_j\neq 0$. Then 
\[\bar{g}\neq 0\Leftrightarrow g \Lambda e_j=e_i \Lambda e_j= e_i \Lambda g.\]
\item Let $\epsilon\in e_i J e_j$ and $l\in \rad^r(e_j \Lambda e_j)$.
Then for any $g\in e_i \Lambda e_j\setminus e_i J e_j$, there exists $l'\in \rad^{r+1} (e_j\Lambda e_j)$
such that 
\[\epsilon l=g l'.\]
\end{enumerate}
\end{lemma}
\begin{proof}
We show (1). First we assume that $\bar{g}\neq 0$. Let $w\in G^i_j$. Then there are $l\in e_i \Lambda e_i$ and $l'\in e_j \Lambda e_j$
such that
\[g=lw=wl'.\]
Since $g\not\in J$, we have $l\not\in e_i \rad \Lambda e_i$ and $l'\not\in e_j \rad \Lambda e_j$.
This shows that $l$ (resp. $l'$) is invertible in $e_i \Lambda e_i$ (resp. $e_j \Lambda e_j$).
In particular, we have
\[g \Lambda e_j=e_i \Lambda e_j= e_i \Lambda g.\]
Next we assume $g \Lambda e_j=e_i \Lambda e_j= e_i \Lambda g.$
Suppose that $\bar{g}=0$. Then by definition of $J$, there are $m\in \Z_{\geq 1}$ and $(i_t,\lambda_t,\lambda'_t,l_t)$ $(t=1,\dots,m)$
with $i_t\in Q_0$, $\lambda_t\in e_i\Lambda e_{i_t}$, $\lambda'_t\in e_{i_t}\Lambda e_j$ and 
$l_t\in e_{i_t}\rad\Lambda e_{i_t}$ such that    
 \[g=\sum_t \lambda_{t} l_t \lambda'_t.\]
We may assume that $\lambda_tl_t \lambda'_t\neq 0$ and take $g'_t\in G^{i_t}_j$. Then we have 
$\lambda'_t=u'_t g'_t$ for some $u'_t\in e_{i_t}\Lambda e_{i_t}$. By using Lemma\;\ref{radtorad}, there is
$l'_t\in e_j\rad{\Lambda} e_j$ such that 
\[l_t\lambda'_t=l_tu'_tg'_t=g'_t l'_t.\]
Since $\lambda_tg'_t\in e_i \Lambda e_j$, there exists $u_t\in e_j\Lambda e_j$ such that
$\lambda_tg'_t=g u_t$. Thus we have
\[\lambda_tl_t \lambda'_t=\lambda_tg'_tl'_t=gu_tl'_t.\]
Therefore $g=gl'$ holds for some $l'\in e_j\rad \Lambda e_j$. In particular, we have $g=0$ which leads to a contradiction.

Next we prove (2). By (1), there are $u\in e_j \Lambda e_j $ and $v\in e_i \Lambda e_i$ such that
$\epsilon=gu=vg$. If $u\not\in e_j \rad \Lambda e_j$, then we have 
 $v\not\in e_i \rad \Lambda e_i$ by Lemma\;\ref{radtorad}. In particular, we obtain
 \[\epsilon \Lambda e_j=e_i \Lambda e_j=e_i \Lambda \epsilon.\]
Hence (1) implies that $\epsilon\not\in J$ which leads to a contradiction. 
Therefore $u\in e_j \rad \Lambda e_j$ and $l'=ul\in  \rad^{r+1}(e_j\Lambda e_j)$ satisfies
\[\epsilon l=gl'.\]
This finishes a proof.
\end{proof}
We have the following commutative diagram:
\[
\begin{xy}
(0,0)*[o]+{KQ}="A",(30,0)*[o]+{KQ/L}="B",%(60,0)*[o]+{L=}="C",
(0,-20)*[o]+{\Lambda}="D",(30,-20)*[o]+{(KQ/L)/\widetilde{I}}="E",
(0,-40)*[o]+{\overline{\Lambda}}="F",(30,-40)*[o]+{(KQ/L)/\widetilde{I'}}="G",
%(60,0)*[o]+{L}
\ar @{->>} "A";"B"
%\ar @{=} "B";"C"
\ar @{->>} "A";"D"
\ar @{->>} "B";"E"
\ar @{->>} "D";"F"
\ar @{->>} "D";"E"
\ar @{->>} "E";"G"
\ar @{=} "F";"G"
\end{xy}
\] where $L$ is an ideal of $KQ$ generated by all loops in $Q$, $\widetilde{I}=(I+L)/L$ and 
$\widetilde{I'}=(I+L+C)/L$ with $C= KQ( \sum_{i\in Q_0} e_i\rad^2 KQ e_i)KQ$.
Since $I$ is an admissible ideal of $KQ$ and $C\subset \rad^2 KQ$, there is $m\in \Z_{\geq 2}$ such that
\[\rad^m KQ/L\subset \widetilde{I'}\subset \rad^2 kQ/L.\]
We note that $KQ/L=KQ^{\circ}$. Then Lemma\;\ref{gpreserve} (1) gives us that $\overline{\Lambda}$ is in $\T'(\Lambda)$.

Let $e_i \overline{\Lambda} e_j\ni \bar{g}\neq 0$.
Since $e_i\overline{\Lambda}e_i=Ke_i$ holds for each $i\in Q_0$, we have  
\[e_i\overline{\Lambda} e_j=K\bar{g}.\]
In particular, $\dim_K e_i \Lambda e_j\leq 1$ for any $i,j\in Q_0$ and
each proper factor algebra of $\overline{\Lambda}$
is not in $\T'(\Lambda)$. Moreover, the following lemma holds.

\begin{lemma}
\label{minimumfactor} Assume $\#G_j^i(Q,I)\leq 1$ holds for each $(i,j)\in Q_0\times Q_0$.
Then we have 
\[\overline{\Lambda}\cong \overline{\Gamma}\]
 for any $\Gamma\in \T'(\Lambda)$.
\end{lemma}
\begin{proof}
Let $\Lambda=KQ/I$, $\Gamma=KQ'/I'$ and $\sigma:Q^{\circ}\stackrel{\sim}{\to}(Q')^{\circ}$
satisfying (a), (b) of Corollary\;\ref{maincor}. We may assume that $Q^{\circ}=(Q')^{\circ}$, $\sigma={\rm id}$ and $G:=G(\Lambda)=G(\Gamma)$. 
Let \[B:=\{(i,j)\in Q_0\times Q_0\mid e_i \Lambda e_j\neq 0\}=\{(i,j)\in Q'_0\times Q'_0\mid e_i \Gamma e_j\neq 0\}.\]

We denote by $J'$ the ideal of $KQ^{\circ}$ generated by all paths of $Q^{\circ}$ not in $G$. Then we have an algebra homomorphism
\[\varsigma:KQ^{\circ}/J'\to \overline{\Lambda}\]
given by $\{\text{paths in $Q^{\circ}$}\}\ni w\mapsto \overline{w}$.
Then it is easy to check that $\varsigma$ is surjective and $\dim_K KQ^{\circ}/J'\leq \#B=\dim_K \overline{\Lambda}$.
In particular,  $\varsigma$ is an isomorphism. The same argument gives us  $KQ^{\circ}/J'\cong \overline{\Gamma}$.
\end{proof} 
Let $f:U' \to U$ be a morphism where $U,U'$ are projective 
modules of $\Lambda$ with $\add U\cap \add U'=\{0\}$. We also let 
$\mathbf{u}:=(U_1,U_2,\dots,U_{\ell})$ and $\mathbf{u'}:=(U'_{-1},U'_{-2},\dots,U'_{-\ell})$ such that
\[U=U_1\oplus U_2\oplus\cdots \oplus U_{\ell} ,\ U'=U'_{-1}\oplus U'_{-2}\oplus U'_{-\ell'}.\]
 Then we define a quiver $Q^{f}=Q^{(f,\mathbf{u},\mathbf{u}')}$ as follows:
\begin{itemize}
\item We set 
 $Q^f_0=\{-\ell',\dots,-1,1,\dots,\ell\}.$
\item Draw an arrow from $-t'$ to $t$ if the composition
$U'_{-t'}\inj U'\stackrel{f}{\to} U\surj U_t$ is not $0$.
\end{itemize}

 \begin{assumption}
\label{assumption}
Let $\Lambda$ be a $\tau$-tilting finite algebra satisfying Condition\;\ref{cd}. 
Assume that for any $T\in \tpsilt \Lambda$ there is
 $f:U'\to U$ with $\add U'\cap \add U=\{0\}$ such that
$T\simeq [U'\stackrel{f}{\to} U]$ and $Q^f$ is a tree.
\end{assumption}
\begin{theorem}
\label{reductiontominimal} Let $\Lambda$ be a basic algebra satisfying Condition\;\ref{cd}.
\begin{enumerate}[{\rm (1)}]
\item If $\Lambda$ satisfies  Assumption\;\ref{assumption}, then $\overline{\Lambda}$
also satisfies  Assumption\;\ref{assumption}.
\item Assume that $\overline{\Lambda}$ satisfies Assumption\;\ref{assumption}.
Then the tensor functor $-\otimes_{\Lambda}\overline{\Lambda}$ induces a poset isomorphism
\[\sttilt \Lambda\stackrel{\sim}{\to} \sttilt \overline{\Lambda}.\]
\end{enumerate}
\end{theorem}
By combining Lemma\;\ref{minimumfactor} and Theorem\;\ref{reductiontominimal}, we have the following corollary. 
\begin{corollary}
\label{application}
Let $\Lambda\cong KQ/I$ be a basic algebra satisfying Condition\;\ref{cd} and $\#G_j^i(Q,I)\leq 1$ for any $i,j\in Q_0$.
If either $\Lambda$ or $\overline{\Lambda}$ satisfies Assumption\;\ref{assumption}, then 
we have $\T(\Lambda)=\T'(\Lambda)$.  
\end{corollary}
\begin{proof}
By Theorem\;\ref{reductiontominimal} (1), we may assume that $\overline{\Lambda}$ satisfies Assumption\;\ref{assumption}.
Then it follows from Theorem\;\ref{reductiontominimal} (2) that
\[\sttilt \Lambda\simeq \sttilt \overline{\Lambda}.\]
In particular, $\Lambda$ is a $\tau$-tilting finite algebra and in $\Theta$.
Thus we obtain $\T(\Lambda)\subset\T'(\Lambda)$ from Corollary\;\ref{maincor}.

Conversely, we let $\Gamma\in \T'(\Lambda)$. Then $\Gamma$ satisfies Condition\;\ref{cd}.
It follows from  Lemma\;\ref{minimumfactor} that $\overline{\Gamma}$ satisfies Assumption\;\ref{assumption}.  Then Theorem\;\ref{reductiontominimal} (2) implies 
\[\sttilt \Gamma\simeq \sttilt \overline{\Gamma}\simeq \sttilt \overline{\Lambda}\simeq \sttilt \Lambda.\]
This shows $\T'(\Lambda)\subset \T(\Lambda)$. 
\end{proof}
%%%%%%%%%%%%%%%%%%%%%%%%%%%%%%%%%%%%%%%%%%%%%%%%%%%%%%%%%%%%%%%%%%%%%%%%%%%%%%%%%%%%%%%%%%%%%%%%%%%%%%
%%%%%%%%%%%%%%%%%%%%%%%%%%%%%%%%%%%%%%%%%%%%%%%%%%%%%%%%%%%%%%%%%%%%%%%%%%%%%%%%%%%%%%%%%%%%%%%%%%%%%%
\subsection{A proof of Theorem\;\ref{reductiontominimal}}
In this subsection, we give a proof of Theorem\;\ref{reductiontominimal}. Let $\Lambda=KQ/I\in \Theta$. 
We regard an morphism from $e_i \Lambda$ to $e_j\Lambda$ as an element of $\Lambda$ by natural
isomorphism $\Hom_{\Lambda}(e_i\Lambda,e_j\Lambda )\simeq e_j \Lambda e_i$.
For two projective modules $U$ and $V$, we define subspace $\widetilde{J}(U,V)$ of $\Hom_{\Lambda}(U,V)$ as follows:
$\varphi\in \widetilde{J}(U,V)$ if and only if for any split monomorphism $\iota:e_i \Lambda\inj U$ and split epimorphism $\pi:V\surj e_j \Lambda$,
the composition map $(e_i \Lambda\stackrel{\iota}{\inj} U\stackrel{\varphi}{\to} V\stackrel{\pi}{\surj} e_j \Lambda)\in e_j\Lambda e_i $ is in $J$.
If $U$ and $V$ are indecomposable with  $f:e_i\Lambda\stackrel{\sim}{\to} U $ and $g:V\stackrel{\sim}{\to} e_j\Lambda $,
then $\varphi\in \widetilde{J}(U,V)$ if and only if
\[ (e_i\Lambda\stackrel{f}{\simeq} U \stackrel{\varphi}{\to} V \stackrel{g}{\simeq} e_j \Lambda)\in J.\]
In this case, we simply denoted by $\varphi\in J$. %or ($\varphi\in e_i J e_j$).
Then for indecomposable decompositions
\[U=U_1\oplus\cdots \oplus U_{\ell}\text{ and }V=V_1\oplus\cdots \oplus V_m,\]
%with $f_p:e_{t_p}\Lambda\stackrel{\sim}{\to} U_p $ and $g_q:V_q\stackrel{\sim}{\to} e_{s_q}\Lambda $,
it is easy to verify that $\varphi\in \widetilde{J}(U,V)$ if and only if
 \[ ( U_p \inj U\stackrel{\varphi}{\to} V\surj V_q)\in J\text{ for any $p,q$}.\]
 
 %Let $U_t\simeq e_{i_t} \Lambda $ and $V_s\simeq e_{j_s}\lambda$.
  %Then $\mathrm{id}=\sum(U\surj U_t\simeq e_{i_t}\Lambda)\simeq U_t\inj U$ and
  %$\mathrm{id}=\sum(V\surj V_s\simeq e_{j_s}\Lambda)\simeq V_s\inj V$.
 %
 \begin{lemma}
 \label{kertensor}
 Let $U,V\in \proj \Lambda$ and $\varphi\in \Hom_{\Lambda}(U,V)$. Then $\varphi\in \widetilde{J}(U,V)$
 if and only if $\varphi\otimes_{\Lambda} \overline{\Lambda}=0$. 
 \end{lemma}
 \begin{proof}
 It is sufficient to show the assertion for the case that $U$ and $V$ are indecomposable.
 We take isomorphisms $f:e_i \Lambda\stackrel{\sim}{\to} U$, $g: V \stackrel{\sim}{\to} e_j \Lambda$
 and denote by $\phi$ the composition map 
 $(e_i \Lambda \stackrel{f}{\to} U\stackrel{\varphi}{\to} V\stackrel{g}{\to} e_j \Lambda)\in e_j \Lambda e_i$.

  First we assume $\varphi\in \widetilde{J}(U,V)$. Then we have that $\phi$
 is in $J$. This implies $\phi\otimes_{\Lambda}\overline{\Lambda}=0$ which leads to $\varphi\otimes_{\Lambda}\overline{\Lambda}=0$.
 
 Next we assume $\varphi\otimes_{\Lambda} \overline{\Lambda}=0$.
 It is clear that $\phi\otimes_{\Lambda}\overline{\Lambda}$ is also $0$.
 This shows $\phi\in J$.  
   \end{proof}

The following lemma is a key to proving Theorem\;\ref{reductiontominimal}.   
\begin{lemma}
\label{teqlemma}
Let $U=U_1 \oplus\cdots\oplus U_{\ell}$, $U'=U'_{-1} \oplus\cdots\oplus U'_{-\ell'}$,
$V=V_1 \oplus\cdots\oplus V_{m}$ and $V'=V'_{-1} \oplus\cdots\oplus V'_{-m'}$
 with $\add U\cap \add U'=\{0\}=\add V\cap \add V'$. Suppose that
 $T=[U'\stackrel{\zeta}{\to} U]$ and $S=[V'\stackrel{\eta}{\to} V]$  are indecomposable two-term objects in $\Kb(\proj \Lambda)$
such that $Q^{\zeta}$ and $Q^{\eta}$ are tree quivers. 
We denote by $\zeta_{q'}^q$ the composition map $(U'_{-q'}\inj U'\stackrel{\zeta}{\to} U \surj U_q)$ and
$\Omega=\Omega(\zeta):=\{(q,q')\mid \zeta_{q'}^q\in J\setminus\{0\}\}$.
\begin{enumerate}[{\rm (1)}]
\item Assume that $\Omega=\emptyset$. For each $(p,q')$ such that $\Hom_{\Lambda}(U'_{-q'},V_{p})\neq 0$, 
we take $g_{q'}^p\in \Hom_{\Lambda}(U'_{-q'},V_{p})\setminus J$ and
denote by $\varphi^{(p,q')}$ the composition map
\[U'\surj U'_{-q'}\stackrel{g_{q'}^p}{\to} V_p \inj V.\]
   Then the following conditions are equivalent.
   \begin{enumerate}[{\rm (i)}]
   \item $\Hom_{\Kb(\proj \Lambda)}(T,S[1])=0$.
   \item For any $p,q'$, there are $h\in \Hom_{\Lambda}(U,V)$ and $h'\in \Hom_{\Lambda}(U',V')$ such that
  \[\varphi^{(p,q')}-(h\circ \zeta+\eta\circ h')\in \widetilde{J}(U',V).\]
 \end{enumerate} 
\item If $T$ is presilting, then we have an isomorphism
\[T\simeq [U'\stackrel{\widetilde{\zeta }}{\to} U],\]
where $\widetilde{\zeta}=\zeta-\sum_{(q,q')\in \Omega}(U'\surj U'_{-q'}\stackrel{\zeta_{q'}^q}{\to} U_q \inj U))$.
\item Assume that $\Hom_{\Kb(\proj \Lambda)}(T,S[1])=0$ and $\Omega=\emptyset$. %Let $\varphi:U'\to V$.
If $\varphi\in \widetilde{J}(U',V)$, then we can choose $h\in \widetilde{J}(U,V)$ and $h'\in \widetilde{J}(U',V')$
such that
\[\varphi=h\circ\zeta+\eta\circ h'.\]
\end{enumerate}
\end{lemma}
\begin{proof}We may assume 
\[U_q=e_{t_q} \Lambda,\ U'_{-q'}=e_{t'_{q'}} \Lambda,\ V_p=e_{s_p} \Lambda\text{ and } V'_{-p'}=e_{s'_{p'}} \Lambda.\]
We distinguish $t_i$ and $t_j$ (resp. $t'_i$ and $t'_j$, $s_i$ and $s_j$, $s'_i$ and $s'_j$)
even if $t_i=t_j$ (resp. $t'_i=t'_j$, $s_i=s_j$, $s'_i=s'_j$) as a vertex of $Q$. Hence
we may assume that $Q^{\zeta}_0=\{t_1,\dots,t_{\ell},t'_1\dots,t'_{\ell'}\}$.
 %Without loss of generality, we may assume that
 %$f_1^1\in \Hom_{\Lambda}(e_{j_1} \Lambda,e_{i_1}\Lambda)=e_{i_1} \Lambda e_{j_1}$ is in $J$,
 %$q=(U'\surj U'_{-1}\stackrel{f_1^1}{\to} U_1 \inj U))$ and $\widetilde{f}=f-q$.
Then we rewrite
\[g_{t'_{q'}}^{s_p}:=g_{q'}^{p},\ \varphi^{(s_p,t'_{q'})}:=\varphi^{(p,q')}.\]

We show the assertion (1) (and (3)). It is immediate that (i) implies (ii). Therefore, we assume (ii) and prove that (i) holds. 
%%%%%%%%%%%%%%%%%%%%%%%%%%%%%%%%%%%%%%%%%%%%%%%%%%%%%%%%%%%%%%%%%%%%%%%%%%%%%%%%%%%%%%%%%%%%%%%%%
%%%%%%%%%%%%%%%%%%%%%%%%%%%%%%%%%%%%%%%%%%%%%%%%%%%%%%%%%%%%%%%%%%%%%%%%%%%%%%%%%%%%%%%%%%%%%%%%%%
%Then the assertion (1) follows from the following claim. 
\begin{claim}
\label{claimforreduction}
Let $t'\in \{t'_1,\dots,t'_{\ell'}\}$ and $s\in \{s_1,\dots,s_{m}\}$.
If $\delta: e_{t'}\Lambda\to e_{s}\Lambda$, then there are $h: U\to V$
 and $h': U'\to V'$ 
such that
\[(U'\surj e_{t'}\Lambda \stackrel{\delta}{\to} e_s \Lambda\inj V)
=h\zeta+\eta h'.\]
\end{claim}
\begin{pfclaim}
We denote by $\pi_{t}$ (resp. $\pi'_{t'}$) the canonical surjection
 $U\surj e_t \Lambda$
(resp. $U'\surj e_{t'} \Lambda$) and $\iota_t$ (resp. $\iota'_{t'}$) the canonical 
injection $e_t \Lambda \inj U$ (resp. $e_{t'}\Lambda\inj U'$). 
%We regard an element of $e\Lambda e'$ as an element of $\Hom_{\Lambda}(e'\Lambda, e\Lambda)$.

 Let $r_i:={\rm max}\{r' \mid  \rad^{r'}(e_i \Lambda e_i)\neq 0\}$ and $r:={\rm max} \{r_i\mid i\in Q_0\}$.
%By Lemma\;\ref{radtorad} and
By Lemma\;\ref{gpreserve}, each element of $\Hom_{\Lambda}(e_{t'} \Lambda,e_s \Lambda)=e_s \Lambda e_{t'}$
has a form $g_{t'}^{s}l'$ with $l'\in  e_{t'} \Lambda e_{t'}$. Moreover, $g_{t'}^{s}l'\in J$ if and only if
 $l'\in   \rad(e_{t'} \Lambda e_{t'})$.
Hence it is sufficient to show that for any $(s,t',r',l')$ satisfying 
 $e_s \Lambda e_{t'}\neq 0$, $r'\leq r$ and $l'\in  \rad^{r'}(e_{t'} \Lambda e_{t'})\setminus  \rad^{r'+1}(e_{t'} \Lambda e_{t'})$,
there are $h\in \Hom_{\Lambda}(U,V)$ and $h'\in \Hom_{\Lambda}(U',V')$ 
such that
\[(U'\surj e_{t'}\Lambda\stackrel{gl'}{\to} e_s \Lambda\inj V)=h\zeta+\eta h',\]
where $g=g_{t'}^{s}$.  

We use an induction on $r-r'$.
First of all, we take $l'_{t'_{q'}}\in \rad^{r'}(e_{t'_{q'}} \Lambda e_{t'_{q'}})$ $(q'=1,2,\dots,\ell')$
 and $l_{t_q}\in \rad^{r'}(e_{t_q} \Lambda e_{t_q})$ $(q=1,2,\dots,\ell)$ as follows:
\begin{enumerate}[\rm (i)] 
\item $l'_{t'}=l'$. 
\item If $\zeta_{q'}^q=\pi_{t_q} \zeta \iota'_{t'_{q'}}\neq 0$, $l'_{t'_{q'}}$ is given and $l_{t_q}$ is not given, 
then we let $l_{t_q}$ such that 
\[ l_{t_q} (\pi_{t_q} \zeta \iota'_{t'_{q'}})=(\pi_{t_q} \zeta \iota'_{t'_{q'}}) l'_{t'_{q'}}
\text{ (see Lemma\;\ref{gpreserve}\;(1) and note that $\Omega=\emptyset$)}.\]
\item If $\zeta_{q'}^q=\pi_{t_q} \zeta \iota'_{t'_{q'}}\neq 0$, $l_{t_q}$ is given and $l'_{t'_{q'}}$ is not given, 
then we let $l'_{t'_{q'}}$ such that 
\[ l_{t_q} (\pi_{t_q} \zeta \iota'_{t'_{q'}})=(\pi_{t_q} \zeta \iota'_{t'_{q'}}) l'_{t'_{q'}}
\text{ (see Lemma\;\ref{gpreserve}\;(1) and note that $\Omega=\emptyset$)}.\]
\item If there is no walk from $t'$ to $t'_{q'}$ in $Q^{\zeta}$, then we let $l_{t'_{q'}}=0$.
 \item If there is no walk from $t'$ to $t_q$ in $Q^{\zeta}$, then we let $l'_{t_q}=0$.
\end{enumerate} 
(Actually, the cases (iv) and (v) do not occur because $T$ is indecomposable.)
The reason why we can take $l_{t'_{q'}}$ $(q'=1,2,\dots,\ell')$ and $l'_{t_q}$ $(q=1,2,\dots,\ell)$ as above is
that $Q^{\zeta}$ is tree. 
By hypothesis and Lemma\;\ref{kertensor}, there are $h_0:U\to V$ and $h_1:U' \to V'$ such that
\[\epsilon:=\varphi^{(s,t')}-(h_0\zeta+\eta h_1)\in \widetilde{J}.\]
%i.e. $(e_{t'_{q'}}\Lambda \inj U'\stackrel{\epsilon}{\to} 
%V\surj e_{s_p} \Lambda)\in e_{s_p}Je_{t'_{q'}}$ for any $p,q'$.

We first consider the case $r-r'=0$. By Lemma\;\ref{gpreserve} (2), 
$\epsilon (\iota'_{t'_{q'}} l'_{t'_{q'}} \pi'_{t'_{q'}})=0$ for any $q'\in \{1,\dots, \ell'\}$.
 Hence we have
\[\begin{array}{lll}
(U'\surj e_{t'}\Lambda\stackrel{gl'}{\to} e_s \Lambda\inj V)  
&=& (U'\surj e_{t'}\Lambda\stackrel{l'}{\to} e_{t'} \Lambda\inj U' \surj e_{t'}\Lambda\stackrel{g}{\to} e_s \Lambda\inj V)\\\\
&=& \displaystyle{ (h_0\zeta+\eta h_1+\epsilon)\iota'_{t'} l'_{t'}\pi'_{t'} }\\\\
       &=& \displaystyle{h_0 \zeta  \iota'_{t'} l'_{t'}\pi'_{t'} 
       + \eta h_1 \iota'_{t'} l'_{t'}\pi'_{t'} }\\\\
       &=& \displaystyle{ \sum_q h_0 \iota_{t_q}(\pi_{t_q} \zeta  \iota'_{t'}) l'_{t'}\pi'_{t'} 
       + \eta h_1 \iota'_{t'} l'_{t'}\pi'_{t'} }\\\\
       &=&\displaystyle{ \sum_q h_0 \iota_{t_q} l_{t_q}(\pi_{t_q} \zeta  \iota'_{t'})\pi'_{t'}  
       + \eta h_1 \iota'_{t'} l'_{t'}\pi'_{t'}}\\\\
       &=&\displaystyle{ \sum_q h_0 \iota_{t_q} l_{t_q}\pi_{t_q} \zeta
         - \sum_q\sum_{\substack{ q'\\t'_{q'}\neq t'}} 
         h_0 \iota_{t_q} l_{t_q} (\pi_{t_q} \zeta \iota'_{t'_{q'}})\pi'_{t'_{q'}}
         + \eta h_1 \iota'_{t'} l'_{t'}\pi'_{t'}}\\
       &=&\displaystyle{ \sum_q h_0 \iota_{t_q} l_{t_q}\pi_{t_q} \zeta
        - \sum_q\sum_{\substack{ q'\\t'_{q'}\neq t'}} 
         h_0 \iota_{t_q}  (\pi_{t_q} \zeta \iota'_{t'_{q'}})l'_{t'_{q'}}\pi'_{t'_{q'}}
         + \eta h_1 \iota'_{t'} l'_{t'}\pi'_{t'}}\\
       &=&\displaystyle{ \sum_q h_0 \iota_{t_q} l_{t_q}\pi_{t_q} \zeta
        - \sum_{\substack{ q'\\t'_{q'}\neq t'}} 
         h_0 \zeta \iota'_{t'_{q'}}l'_{t'_{q'}}\pi'_{t'_{q'}}
         + \eta h_1 \iota'_{t'} l'_{t'}\pi'_{t'}}\\
       &=&\displaystyle{ \sum_q h_0 \iota_{t_q} l_{t_q}\pi_{t_q} \zeta
        - \sum_{\substack{ q'\\t'_{q'}\neq t'}} 
         (\varphi^{(t',s)}-\eta h_1-\epsilon)
          \iota'_{t'_{q'}}l'_{t'_{q'}}\pi'_{t'_{q'}}+ \eta h_1 \iota'_{t'} l'_{t'}\pi'_{t'}}\\
       &=&\displaystyle{ \sum_q h_0 \iota_{t_q} l_{t_q}\pi_{t_q} \zeta
        + \sum_{\substack{ q'\\t'_{q'}\neq t'}} 
         \eta h_1\iota'_{t'_{q'}}l'_{t'_{q'}}\pi'_{t'_{q'}}
         + \eta h_1 \iota'_{t'} l'_{t'}\pi'_{t'}}.\\
\end{array}\]
Therefore, $h=\underset{q}\sum h_0 \iota_{t_q} l_{t_q}\pi_{t_q}$ and 
$h'=\displaystyle{\sum_{\substack{ q'\\t'_{q'}\neq t'}} 
          h_1\iota'_{t'_{q'}}l'_{t'_{q'}}\pi'_{t'_{q'}}+h_1 \iota'_{t'} l'_{t'}\pi'_{t'}}$
          satisfy
 \[(U'\surj e_{t'}\Lambda\stackrel{gl'}{\to} e_s \Lambda\inj V)=h\circ \zeta+\eta\circ h'.\]
Moreover, if $r'>0$, then $h\in \widetilde{J}(U,V)$ and $h'\in \widetilde{J}(U',V')$.         
%where $\widetilde{g}=(U'\surj e_{t'}\Lambda\stackrel{g}{\to} e_s \Lambda\inj V)$.

We assume that the assertion holds for the case $r-r'<N$ and consider the case $r-r'=N$.
By Lemma\;\ref{gpreserve} (2) and the hypothesis of induction, we have that 
$\epsilon (\iota_{t'_{q'}} l'_{t'_{q'}} \pi_{t'_{q'}})=h_{t'_{q'}}\zeta +\eta h'_{t'_{q'}}$ 
for some $h_{t'_{q'}}\in \widetilde{J}(U, V)$
 and $h'_{t'_{q'}}\in \widetilde{J}(U',V')$.
Hence, we have
\[\begin{array}{lll}
(U'\surj e_{t'}\Lambda\stackrel{gl'}{\to} e_s \Lambda\inj V)  &=&
 \displaystyle{ (h_0\zeta+\eta h_1+\epsilon)\iota'_{t'} l'_{t'}\pi'_{t'} }\\\\
       &=& \displaystyle{h_0 \zeta  \iota'_{t'} l'_{t'}\pi'_{t'} 
       + \eta h_1 \iota'_{t'} l'_{t'}\pi'_{t'} }+h_{t'}\zeta +\eta h'_{t'}.\\
   \end{array}  
 \]
On the other hand, we have   
 \[\begin{array}{lll} 
   \displaystyle{h_0 \zeta  \iota'_{t'} l'_{t'}\pi'_{t'}}    
   &=& \displaystyle{ \sum_q h_0 \iota_{t_q}(\pi_{t_q} \zeta  \iota'_{t'}) l'_{t'}\pi'_{t'}}\\\\
   &=&\displaystyle{ \sum_q h_0 \iota_{t_q} l_{t_q}(\pi_{t_q} \zeta  \iota'_{t'})\pi'_{t'}}\\\\ 
   &=&\displaystyle{ \sum_q h_0 \iota_{t_q} l_{t_q}\pi_{t_q} \zeta
         - \sum_q\sum_{\substack{ q'\\t'_{q'}\neq t'}} 
         h_0 \iota_{t_q} l_{t_q} (\pi_{t_q} \zeta \iota'_{t'_{q'}})\pi'_{t'_{q'}}}\\
       &=&\displaystyle{ \sum_q h_0 \iota_{t_q} l_{t_q}\pi_{t_q} \zeta
        - \sum_q\sum_{\substack{ q'\\t'_{q'}\neq t'}} 
         h_0 \iota_{t_q}  (\pi_{t_q} \zeta \iota'_{t'_{q'}})l'_{t'_{q'}}\pi'_{t'_{q'}}}\\
       &=&\displaystyle{ \sum_q h_0 \iota_{t_q} l_{t_q}\pi_{t_q} \zeta
        - \sum_{\substack{ q'\\t'_{q'}\neq t'}} 
         h_0 \zeta \iota'_{t'_{q'}}l'_{t'_{q'}}\pi'_{t'_{q'}}}\\
       &=&\displaystyle{ \sum_q h_0 \iota_{t_q} l_{t_q}\pi_{t_q} \zeta
        - \sum_{\substack{ q'\\t'_{q'}\neq t'}} 
         (\varphi^{(t',s)}-\eta h_1-\epsilon)
          \iota'_{t'_{q'}}l'_{t'_{q'}}\pi'_{t'_{q'}}}\\
       &=&\displaystyle{ \sum_q h_0 \iota_{t_q} l_{t_q}\pi_{t_q} \zeta
        + \sum_{\substack{ q'\\t'_{q'}\neq t'}} 
         (\eta h_1\iota'_{t'_{q'}}l'_{t'_{q'}}\pi'_{t'_{q'}}+h_{t'_{q'}}\zeta +\eta h'_{t'_{q'}})}.\\
\end{array}\]
Hence, $h=h_{t'}+\underset{q}{\sum} h_0 \iota_{t_q} l_{t_q}\pi_{t_q}$ and 
$h'= h_1 \iota'_{t'} l'_{t'}\pi'_{t'} + \underset{q'}{\sum} h'_{t'_{q'}}+
\displaystyle{\sum_{\substack{ q'\\t'_{q'}\neq t'}} 
          h_1\iota'_{t'_{q'}}l'_{t'_{q'}}\pi'_{t'_{q'}}}$
          satisfy
 \[(U'\surj e_{t'}\Lambda\stackrel{gl'}{\to} e_s \Lambda\inj V)=h\circ \zeta+\eta\circ h'.\]

Furthermore, these constructions of $h$ and $h'$ show the assertion (3). 
\end{pfclaim}
Then the assertion (1) follows from the previous claim.

%%%%%%%%%%%%%%%%%%%%%%%%%%%%%%%%%%%%%%%%%%%%%%%%%%%%%%%%%%%%%%%%%%%%%%%%%%%%%%%%%%%%%%%%%
%%%%%%%%%%%%%%%%%%%%%%%%%%%%%%%%%%%%%%%%%%%%%%%%%%%%%%%%%%%%%%%%%%%%%%%%%%%%%%%%%%%%%%%%%

We show (2).  
For each pair $(t',t)=(t'_{q'},t_q)$ such that $G_{t'}^{t}\neq \emptyset$, we define 
\[\varphi^{(t',t)}:=(U'\surj e_{t'}\Lambda \stackrel{g_{t'}^{t}}{\to} e_t \Lambda \inj U),\]
where $g_{t'}^{t}$ is taken from $e_t \Lambda e_{t'}\setminus e_t J e_{t'}$.

Since $T\in \tpsilt \Lambda$, there are $h:U\to U$ and $h':U'\to U'$ such that
\[\varphi^{(t',t)}=h\circ \zeta+\zeta \circ h'=h\circ \widetilde{\zeta}+\widetilde{\zeta}\circ h'+h\circ \epsilon+\epsilon\circ h',\]
where $\epsilon:=\sum_{(q,q')\in \Omega}(U'\surj U'_{-q'}\stackrel{\zeta_{q'}^q}{\to} U_q \inj U))\in \widetilde{J}(U',U)$.
Hence, we have
\[\varphi^{(t',t)}-(h\circ\widetilde{\zeta}+\widetilde{\zeta}\circ h')\in \widetilde{J}(U',U).\]
By using (1), we obtain that $T'=[U'\stackrel{\widetilde{\zeta}}{\to}U]$ is presilting.
Then the assertion follows from Theorem\;\ref{gvector}.
\end{proof}

We consider a full functor 
\[-\otimes_{\Lambda}\overline{\Lambda} :\proj \Lambda \to \proj \overline{\Lambda}\]
and denote it by $\overline{(-)}$. 
\begin{corollary}
\label{indecomposability}
Let $\Lambda$ be as in Assumption\;\ref{assumption}.
\begin{enumerate}[{\rm (1)}]
\item If $T=[T_{-1}\stackrel{\zeta}{\to} T_0]$ be an indecomposable two-term presilting object of $\Kb(\proj \Lambda)$, then
 $\overline{T}:=[\overline{T}_{-1}\stackrel{\overline{\zeta}}{\to} \overline{T}_0]$ is indecomposable two-term presilting object of $\Kb(\proj \overline{\Lambda})$.
\item  %\tpsilt \Lambda\to \tpsilt \overline{\Lambda}$ 
We have a poset isomorphism
\[\overline{(-)}:\tsilt \Lambda \stackrel{\sim}{\to} \tsilt \overline{\Lambda}.\]
Moreover, $\overline{\Lambda}$ also satisfies Assumption\;\ref{assumption}.
\end{enumerate}
\end{corollary}
\begin{proof}
Let $T=[T_{-1}\stackrel{\zeta}{\to} T_0]\in \ind(\tpsilt \Lambda)$.
We first show that $\overline{T}$ is indecomposable. By Lemma\;\ref{teqlemma}\;(2),
we may assume that $\Omega(\zeta)=\emptyset$. 

Suppose that $\overline{T}$ is not indecomposable.
Then $\End_{\Kb(\proj \overline{\Lambda})} (\overline{T})$ is not local.
Thus we can take $\phi=(\phi_{-1},\phi_0)\in \End_{\Kb(\proj \Lambda)}(\overline{T})
\setminus \rad (\End_{\Kb(\proj \Lambda)}(\overline{T}))$ 
 such that $\phi$ is not an isomorphism. Let $\varphi_0:T_0\to T_0$ and $\varphi_{-1}:T_{-1}\to T_{-1}$
 such that $\overline{\varphi}_0=\phi_0$ and $\overline{\varphi}_{-1}=\phi_{-1}$.
Now consider $\epsilon:=\varphi_0\circ \zeta-\zeta\circ \varphi_{-1}$.
By Lemma\;\ref{kertensor}, we have $\epsilon\in \widetilde{J}(T_{-1},T_0)$. Since $T\in \tpsilt \Lambda$, Lemma\;\ref{teqlemma}\;(3) implies that
there are $h_0\in \widetilde{J}(T_0,T_0)$ and $h_{-1}\in \widetilde{J}(T_{-1},T_{-1})$
such that $\epsilon=-h_0\circ\zeta+\zeta\circ h_{-1}$.
In particular, we obtain 
\[(\varphi_0+h_0)\circ\zeta-\zeta\circ(\varphi_{-1}+h_{-1})=0.\] 
It follows from Lemma\;\ref{kertensor} that $\overline{\varphi_0+h_0}=\phi_0$ and $\overline{\varphi_{-1}+h_{-1}}=\phi_{-1}$.
Thus there exists $\varphi\in \End_{\Kb(\proj \Lambda)}(T)$ such that 
$\overline{\varphi}=\phi$. Since $\phi$ is not in the radical, $\varphi$ is also not in the radical.
By indecomposability of $T$, we have that $\varphi$ is an isomorphism. 
This implies that $\phi$ is an isomorphism which leads to a contradiction.

Let $T$ and $T'$ be in $\ind(\tpsilt \Lambda)$. Since the tensor functor 
$-\otimes_{\Lambda}\overline{\Lambda}:\proj \Lambda\to \proj \overline{\Lambda}$ is full,
it follows from Lemma\;\ref{teqlemma}\;(1) and (2) that
\[\Hom_{\Kb(\proj \Lambda)}(T,T'[1])=0\Leftrightarrow \Hom_{\Kb(\proj \overline{\Lambda})}(\overline{T},\overline{T'}[1])=0.\]
In particular, we obtain (1). By Theorem\;\ref{gvector}, $\overline{(-)}$ induces a poset isomorphism 
\[\tsilt \Lambda \simeq \P,\]
where $\P:=\overline{\tsilt \Lambda}$ be the image of $\tsilt \Lambda$ under $\overline{(-)}$.
%Suppose that there are $T$ and $T'$ satisfying
%\begin{itemize}
%\item there is an arrow $\overline{T}\to \overline{T'}$ in $\H(\P)$;
%\item there is no arrow $\overline{T}\to \overline{T'}$ in $\H(\tsilt \overline{\Lambda})$.
%\end{itemize}
%Since $\dis(\overline{T})=\{\overline{T''}\mid T''\in \dis(T)\}$, it follows from Theorem\;\ref{basicprop}\;(3) that
%there exists $T^{(1)}\in \dis(T)$ such that $\overline{T}^{(1)}\geq \overline{T'}$. 
%Since there is no arrow $\overline{T}\to \overline{T'}$ in $\H(\tsilt \overline{\Lambda})$, we have 
%\[\overline{T}>\overline{T}^{(1)}> \overline{T'}.\]
%This contradicts to that $\overline{T'}$ is a direct successor of $\overline{T}$ in $\P$.
Therefore, for $X,Y\in \tsilt \Lambda$, we have
\[\begin{array}{ccl}
\overline{X}\to \overline{Y}\text{ in $\H(\P)$}& \Leftrightarrow & X\to Y\text{ in $\H(\tsilt \Lambda)$}\\
                                               & \Leftrightarrow & X>Y \text{ and}\ |\add X\cap \add Y|=|\Lambda|-1\\
                                               & \Leftrightarrow & \overline{X}>\overline{Y} \text{ and}\ |\add \overline{X}\cap \add \overline{Y}|=|\overline{\Lambda}|-1\\
                                               & \Leftrightarrow & \overline{X}\to \overline{Y}\text{ in $\H(\tsilt \overline{\Lambda})$},
\end{array}\]
where, $|\chi|$ denotes the cardinality of $\ind \chi$.
In particular, $\P\simeq \tsilt \Lambda$ is a strongly full subposet of $\tsilt \overline{\Lambda}$.
Then the assertion follows from Theorem\;\ref{basicprop}\;(4).
\end{proof}

In the rest of this subsection, we always assume that $\overline{\Lambda}$
satisfies Assumption\;\ref{assumption}.
For an indecomposable two-term silting object $T$ in $\Kb(\proj \overline{\Lambda})$, 
we denote by $[T]$
 the isomorphism class of $T$. 
    We fix a morphism $f_{[T]}:T^{-1}\to T^0$ with $\add T^{-1}\cap \add T^0=\{0\}$ and indecomposable decompositions
$\mathbf{u}_{[T]}$ of $T^0$, $\mathbf{u}_{[T]}'$ of $T^1$   
 such that $T_{f_{[T]}}:=[T^{-1}\stackrel{f_{[T]}}{\to} T^0]\simeq T$
   and $Q^{(f_{[T]},\mathbf{u}_{[T]},\mathbf{u}_{[T]}')}$ is a tree. 
 Then we set
\[\overline{\Phi}:=\{f_{[T]}\mid T\in \tpsilt \overline{\Lambda}\}.\] 

For each element $ \overline{\Phi}\ni f=f_{[T]}:T^{-1}\to T^0$, 
with $\mathbf{u}_{[T]}=(T^0_1,\dots,T^0_{\ell})$ and $\mathbf{u}_{[T]}'=(T^{-1}_1,\dots,T^{-1}_{\ell'})$, 
we let $\widehat{T}_1^0,\dots,\widehat{T}_\ell^1$, $\widehat{T}_1^{-1},\dots,\widehat{T}_{\ell'}^{-1}$
such that $\widehat{T}_t^0\otimes_{\Lambda} \overline{\Lambda}=T_t^0$ and $\widehat{T}_{t'}^{-1}\otimes_{\Lambda} \overline{\Lambda}=T^{-1}_{t'}$. 
Then we choose $\widehat{f}=\widehat{f}_{[T]}:\widehat{T}^{-1}=\underset{1\leq t'\leq \ell'}{\bigoplus}\widehat{T}^{-1}_{t'}\to 
\underset{1\leq t \leq \ell}{\bigoplus}\widehat{T}^0_t=\widehat{T}^0$ with $\widehat{f}\otimes_{\Lambda}\overline{\Lambda}=f$
satisfying the following implication:
\[(T^{-1}_{t'}\inj T^{-1}\stackrel{f}{\to} T^0\surj T^0_t)=0
\Rightarrow (\widehat{T}^{-1}_{t'}\inj \widehat{T}^{-1}\stackrel{\widehat{f}}{\to} \widehat{T}^0\surj \widehat{T}^0_t)=0.\]
We set
\[\Phi:=\{\widehat{f}_{[T]}\mid T\in \tpsilt \overline{\Lambda}\}.\]
Now for any $T\in \tpsilt \overline{\Lambda}$, we denote by $\widehat{T}$ the 
two-term complex in $\Kb(\proj \Lambda)$ given by $\widehat{f}_{[T]}$.  
  
\begin{lemma}
\label{lemmaforreduction}
Let $T$, $S$ be in $\ind(\tpsilt \overline{\Lambda})$. Then we have
\[\Hom_{\Kb(\proj \overline{\Lambda})}(T,S[1])=0\Leftrightarrow \Hom_{\Kb(\proj \Lambda)}(\widehat{T},\widehat{S}[1])=0.\]
\end{lemma}
\begin{proof}
We suppose $\Hom_{\Kb(\proj \overline{\Lambda})}(T,S[1])=0$.
Without loss of generality, we may assume that $T=[T^{-1}\stackrel{f}{\to} T^0]$ and $S=[S^{-1}\stackrel{f'}{\to} S^0]$ with $f,f'\in 
\overline{\Phi}$ and 
$\mathbf{u}_{[T]}=(e_{t_1} \overline{\Lambda},\dots,e_{t_{\ell}} \overline{\Lambda})$,
$\mathbf{u'}_{[T]}=(e_{t'_1} \overline{\Lambda},\dots,e_{t'_{\ell'}} \overline{\Lambda})$,
$\mathbf{u}_{[S]}=(e_{s_1} \overline{\Lambda},\dots,e_{s_{m}} \overline{\Lambda})$,
$\mathbf{u'}_{[S]}=(e_{s'_1} \overline{\Lambda},\dots,e_{s'_{m'}} \overline{\Lambda})$.
We also may assume  $\widehat{T}=[\widehat{T}^{-1}\stackrel{\widehat{f}}{\to}\widehat{T}]$ and
$\widehat{S}=[\widehat{S}^{-1}\stackrel{\widehat{f'}}{\to}\widehat{S}^0]$ with
\[\begin{array}{llllclc}
\widehat{T}^0&=& e_{t_1} \Lambda &\oplus\cdots\oplus & e_{t_{\ell}}\Lambda,\\
\widehat{T}^{-1}&=&e_{t'_1} \Lambda &\oplus\cdots\oplus & e_{t'_{\ell'}}\Lambda,\\
\widehat{S}^0&=&e_{s_1} \Lambda &\oplus\cdots\oplus & e_{s_{m}}\Lambda, \\
\widehat{S}^{-1}&=&e_{s'_1} \Lambda &\oplus\cdots\oplus & e_{s'_{m'}}\Lambda.\\
\end{array}
\]

Let $g\in \Hom_{\Lambda}(\widehat{T}_{p'}^{-1},\widehat{S}_q^0)$. Since $\Hom_{\Kb(\proj \overline{\Lambda})}(T,S[1])=0$ and
$\overline{(-)}:\proj \Lambda\to \proj \overline{\Lambda}$ is full, there are 
$h_0\in \Hom_{\Lambda}(\widehat{T}^0,\widehat{S}^{0})$ and 
$h_1\in \Hom_{\Lambda}(\widehat{T}^{-1},\widehat{S}^{-1})$ such that
\[(\widehat{T}^{-1}\surj \widehat{T}_{p'}^{-1}\stackrel{g}{\to} \widehat{S}_q^0\inj \widehat{S})\otimes_{\Lambda} \overline{\Lambda}
=(h_0\widehat{f}+\widehat{f'} h_1)\otimes_{\Lambda}\overline{\Lambda}.\]
Let $\epsilon=
(\widehat{T}^{-1}\surj \widehat{T}_{p'}^{-1}\stackrel{g}{\to} \widehat{S}_q^0\inj \widehat{S})-(h_0\widehat{f}+\widehat{f'} h_1)$.
 %Then we have
 %\[(e_{t'_p}\Lambda \inj \Upsilon'\stackrel{\epsilon}{\to} 
%\Pi\surj e_{s_q} \Lambda)\in e_{s_q}Je_{t'_p}\]
%for any $p,q$. 
Since $\overline{\epsilon}=0$, we obtain $\epsilon\in \widetilde{J}(\widehat{T}^{-1},\widehat{S}^0)$ from Lemma\;\ref{kertensor}.
Therefore, we can apply Lemma\;\ref{teqlemma}(1) and obtain 
\[\Hom_{\Kb(\proj \Lambda)}(\widehat{T},\widehat{S}[1])=0.\] 

Assume that $\Hom_{\Kb(\proj \Lambda)}(\widehat{T},\widehat{S}[1])=0$. Since $\overline{(-)}$ is full, it is easy to check
\[\Hom_{\Kb(\proj \overline{\Lambda})}(T,S[1])=0.\]
This finishes a proof.
\end{proof}

We prove Theorem\;\ref{reductiontominimal}.
The assertion (1) follows from Corollary\;\ref{indecomposability}.
Let $\mathbb{P}:=\tsilt \Lambda\cap \add \underset{T\in \tsilt \overline{\Lambda}}{\bigoplus} \widehat{T}$.
By Lemma\;\ref{lemmaforreduction}, $\mathbb{P}$ is  isomorphic to $\tsilt \overline{\Lambda}$ and a strongly full subposet of $\tsilt \Lambda$. Since $\tsilt \overline{\Lambda}$ is a finite poset,
we have the assertion (2) from Theorem\;\ref{basicprop} and Theorem\;\ref{bijection}.

\subsection{Applications to Nakayama algebras}
\label{subsect4.2}
\begin{definition}
A module $M$ is said to be {\bf uniserial} if it has the unique composition series.
If every indecomposable projective modules and every indecomposable injective modules of $\Lambda$ are uniserial,
 then we call $\Lambda$ a {\bf Nakayama algebra}. 
\end{definition}
Nakayama algebras are characterized as follows.
\begin{theorem}[{\cite[Chapter\;V, Theorem\;3.2]{ASS}}]
$\Lambda$ is a Nakayama algebra if and only if $Q$ is either a quiver of type $A_n$ with a linear orientation or a 
cyclic quiver. 
\end{theorem}
\begin{proposition}
\label{taurigidnakayama}
 Let $\Lambda$ be a Nakayama algebra and $M\in \ind\Lambda$.
\begin{enumerate}[{\rm (1)}]
\item{\cite[Chapter\;V, Theorem\;3.5]{ASS}} There are $i\in Q_0$ and $r\in \Z_{\geq 0}$ such that
$M\simeq P_i/\rad^r P_i$. In particular, $\Lambda$ is representation-finite.
\item\cite[Proposition\;2.5]{A} Assume that $M$ is non-projective. Then $M$ is $\tau$-rigid if and only
 if its Loewy length $\ell\ell(M)$ is less than $n$.
\end{enumerate}
\end{proposition}
By using Proposition\;\ref{taurigidnakayama}, each indecomposable two-term presilting object of a Nakayama algebra $\Lambda$
has one of the following forms:
\[[0\to P_i],\ [P_j\stackrel{g_j^i\cdot -}{\longrightarrow}P_i],\ [P_j\to 0],\]
where $g_j^i$ is a shortest path from $i$ to $j$ on $Q$. 
In particular, if $\ell\ell(P_i)\geq n$ holds for each $i\in Q_0$, then
we have
\[\ind(\tpsilt \Lambda)=\{[0\to P_i]\mid i\in Q_0\}\sqcup\{[P_j\stackrel{g_j^i\cdot -}{\longrightarrow}P_i]\mid i\neq j\in Q_0\}\sqcup\{[P_j\to 0]\mid j\in Q_0\}. \]
\begin{theorem}[{\cite[Theorem\;3.11]{A}}]
\label{adachireduction}

Let $\Lambda$ be a Nakayama algebra.
Assume that $\ell\ell(P_i)\geq n$ holds for any $i\in Q_0$. Then we have a poset isomorphism
\[\sttilt \Lambda\simeq \sttilt kC/R^n,\]
where $C$ is a cyclic quiver with $C_0:=\{1,\dots,n\}$ and $R=R_n:=\rad kC$.
\end{theorem}
We  generalize Theorem\;\ref{adachireduction} by applying Corollary\;\ref{application}.
\begin{proposition}
\label{nakayamacase}
Let $\Lambda=kC/I$ be a Nakayama algebra. Then
 we have $\T(\Lambda)=\T'(\Lambda)$.
% $\sttilt \Lambda\simeq \sttilt \Gamma$ if and only if $\Lambda$ satisfies $(\mathrm{a})$ Condition\;\ref{cd} and
%\begin{enumerate}[{\rm (a)}]
%\item
%$(\mathrm{b})$ there is a quiver isomorphism $\sigma:Q^{\circ}\stackrel{\sim}{\to} C$ satisfying 
%$\supp(e_{\sigma(i)} \Gamma )=\sigma(\supp(e_i \Lambda))$ for any $i\in Q_0$.
%\item $x\Lambda e_j=e_i\Lambda e_j=e_i\Lambda x$ holds for any arrow $x\in Q^{\circ}_1$.
%\item $\supp e_i \Lambda=Q_0$ holds for any $i\in Q_0$.
%\end{enumerate}
\end{proposition}
%%%%%%%%%%%%%%%%%%%%%%%%%%%%%%%%%%%%%%%%%%%%%%%%%%%%%%%%%%%%%%%%%%%%%%%%%%%%%%%%%%%%%%%%%%%%%%%%%%%%%%%%%%%%%%%%%%
%%%%%%%%%%%%%%%%%%%%%%%%%%%%%%%%%%%%%%%%%%%%%%%%%%%%%%%%%%%%%%%%%%%%%%%%%%%%%%%%%%%%%%%%%%%%%%%%%%%%%%%%%%%%%%%%%%
 \subsection{Applications to Brauer tree algebras}
 \label{subsect4.3}
 Let $\Tree$ be a tree, $\m: \Tree_0\to \Z_{>0}$ a map from the set of vertices of $\Tree$ to the set of positive
 integers and $\co_v$ a cyclic ordering of the set of edges of $\Tree$ adjacent to a vertex $v$.
 Where cyclic ordering of a finite set $E$ is defined to be a bijection $c:E\to E$ such that
 $\{c^m(e)\mid m\in \Z\}=E$ for any $e\in E$, i.e., $c\in \Sym_E$ has the form $(e_0,e_1,\dots,e_{|E|-1})$.  
 For $(e,e')\in E\times E$ with $e'=c^m(e)$ ($0\leq m\leq |E|-1$), we set $[e,e']_c:=\{c^\ell(e)\mid 0\leq \ell\leq m\}$.
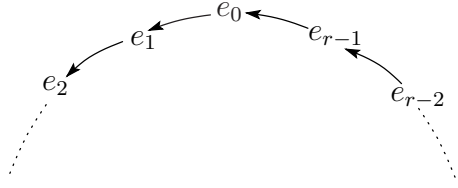
 \[%WinTpicVersion4.30a
{\unitlength 0.1in%
\begin{picture}( 26.8000, 11.5000)(  7.7000,-20.2000)%
% STR 2 0 3 0 Black White
% 4 770 2050 770 2150 2 0 0 0
% $E=\{e_m\mid m\in \Z/r\Z\}$, $r=\#E$, $c(e_m)=e_{m+1}$
\put(7.7000,-21.5000){\makebox(0,0)[lb]{$E=\{e_m\mid m\in \Z/r\Z\}$, $r=\#E$, $c(e_m)=e_{m+1}$}}%
% STR 2 0 3 0 Black White
% 4 2200 900 2200 1000 2 0 0 0
% $e_0$
\put(22.0000,-10.0000){\makebox(0,0)[lb]{$e_0$}}%
% SPLINE 2 0 3 0 Black White
% 3 2170 970 2000 1000 1850 1050
% 
\special{pn 8}%
\special{pa 2170 970}%
\special{pa 2138 974}%
\special{pa 2106 979}%
\special{pa 2075 984}%
\special{pa 2043 990}%
\special{pa 2012 997}%
\special{pa 1981 1005}%
\special{pa 1951 1014}%
\special{pa 1920 1024}%
\special{pa 1860 1046}%
\special{pa 1850 1050}%
\special{fp}%
% SARROW 2 0 3 1 Black White
% 2 1860 1046 1850 1050
% 
\special{pn 8}%
\special{pa 1860 1046}%
\special{pa 1850 1050}%
\special{fp}%
\special{sh 1}%
\special{pa 1850 1050}%
\special{pa 1919 1044}%
\special{pa 1900 1030}%
\special{pa 1904 1007}%
\special{pa 1850 1050}%
\special{fp}%
% STR 2 0 3 0 Black White
% 4 1750 1050 1750 1150 2 0 0 0
% $e_1$
\put(17.5000,-11.5000){\makebox(0,0)[lb]{$e_1$}}%
% SPLINE 2 0 3 0 Black White
% 3 1710 1110 1510 1210 1420 1290
% 
\special{pn 8}%
\special{pa 1710 1110}%
\special{pa 1650 1134}%
\special{pa 1620 1147}%
\special{pa 1591 1161}%
\special{pa 1563 1176}%
\special{pa 1536 1192}%
\special{pa 1510 1210}%
\special{pa 1485 1230}%
\special{pa 1461 1251}%
\special{pa 1437 1273}%
\special{pa 1420 1290}%
\special{fp}%
% SARROW 2 0 3 1 Black White
% 2 1437 1273 1420 1290
% 
\special{pn 8}%
\special{pa 1437 1273}%
\special{pa 1420 1290}%
\special{fp}%
\special{sh 1}%
\special{pa 1420 1290}%
\special{pa 1481 1257}%
\special{pa 1458 1252}%
\special{pa 1453 1229}%
\special{pa 1420 1290}%
\special{fp}%
% STR 2 0 3 0 Black White
% 4 1290 1290 1290 1390 2 0 0 0
% $e_2$
\put(12.9000,-13.9000){\makebox(0,0)[lb]{$e_2$}}%
% SPLINE 2 2 3 0 Black White
% 3 1310 1440 1210 1600 1120 1800
% 
\special{pn 8}%
\special{pn 8}%
\special{pa 1310 1440}%
\special{pa 1306 1447}%
\special{fp}%
\special{pa 1285 1477}%
\special{pa 1281 1484}%
\special{fp}%
\special{pa 1261 1514}%
\special{pa 1257 1521}%
\special{fp}%
\special{pa 1238 1552}%
\special{pa 1234 1559}%
\special{fp}%
\special{pa 1215 1590}%
\special{pa 1211 1597}%
\special{fp}%
\special{pa 1195 1630}%
\special{pa 1191 1637}%
\special{fp}%
\special{pa 1176 1670}%
\special{pa 1172 1677}%
\special{fp}%
\special{pa 1157 1711}%
\special{pa 1154 1718}%
\special{fp}%
\special{pa 1140 1751}%
\special{pa 1137 1759}%
\special{fp}%
\special{pa 1123 1793}%
\special{pa 1120 1800}%
\special{fp}%
% SPLINE 2 0 3 0 Black White
% 3 2360 960 2530 990 2680 1040
% 
\special{pn 8}%
\special{pa 2360 960}%
\special{pa 2392 964}%
\special{pa 2424 969}%
\special{pa 2455 974}%
\special{pa 2487 980}%
\special{pa 2518 987}%
\special{pa 2549 995}%
\special{pa 2579 1004}%
\special{pa 2610 1014}%
\special{pa 2670 1036}%
\special{pa 2680 1040}%
\special{fp}%
% SARROW 2 0 3 1 Black White
% 2 2392 964 2360 960
% 
\special{pn 8}%
\special{pa 2392 964}%
\special{pa 2360 960}%
\special{fp}%
\special{sh 1}%
\special{pa 2360 960}%
\special{pa 2424 988}%
\special{pa 2413 967}%
\special{pa 2429 948}%
\special{pa 2360 960}%
\special{fp}%
% STR 2 0 3 0 Black White
% 4 2970 1040 2970 1140 3 0 0 0
% $e_{r-1}$
\put(29.7000,-11.4000){\makebox(0,0)[rb]{$e_{r-1}$}}%
% SPLINE 2 0 3 0 Black White
% 3 2880 1150 3080 1250 3170 1330
% 
\special{pn 8}%
\special{pa 2880 1150}%
\special{pa 2940 1174}%
\special{pa 2970 1187}%
\special{pa 2999 1201}%
\special{pa 3027 1216}%
\special{pa 3054 1232}%
\special{pa 3080 1250}%
\special{pa 3105 1270}%
\special{pa 3129 1291}%
\special{pa 3153 1313}%
\special{pa 3170 1330}%
\special{fp}%
% SARROW 2 0 3 1 Black White
% 2 2910 1162 2880 1150
% 
\special{pn 8}%
\special{pa 2910 1162}%
\special{pa 2880 1150}%
\special{fp}%
\special{sh 1}%
\special{pa 2880 1150}%
\special{pa 2934 1193}%
\special{pa 2930 1170}%
\special{pa 2949 1156}%
\special{pa 2880 1150}%
\special{fp}%
% STR 2 0 3 0 Black White
% 4 3400 1360 3400 1460 3 0 0 0
% $e_{r-2}$
\put(34.0000,-14.6000){\makebox(0,0)[rb]{$e_{r-2}$}}%
% SPLINE 2 2 3 0 Black White
% 3 3260 1460 3360 1620 3450 1820
% 
\special{pn 8}%
\special{pn 8}%
\special{pa 3260 1460}%
\special{pa 3264 1467}%
\special{fp}%
\special{pa 3285 1497}%
\special{pa 3289 1504}%
\special{fp}%
\special{pa 3309 1534}%
\special{pa 3313 1541}%
\special{fp}%
\special{pa 3332 1572}%
\special{pa 3336 1579}%
\special{fp}%
\special{pa 3355 1610}%
\special{pa 3359 1617}%
\special{fp}%
\special{pa 3375 1650}%
\special{pa 3379 1657}%
\special{fp}%
\special{pa 3394 1690}%
\special{pa 3398 1697}%
\special{fp}%
\special{pa 3413 1731}%
\special{pa 3416 1738}%
\special{fp}%
\special{pa 3430 1771}%
\special{pa 3433 1779}%
\special{fp}%
\special{pa 3447 1813}%
\special{pa 3450 1820}%
\special{fp}%
\end{picture}}%
\vspace{5pt}\]

  Then $(\Tree,\m,\co)$ is said to be a {\bf generalized Brauer tree}. In this subsection, we assume that
  each generalized Brauer tree $(\Tree,\m,\co)$ satisfies that $\Tree$ is connected and $\#\Tree_1\geq 2$.   
 \begin{definition}
 \label{bta} Let $(\Tree,\m,\co)$ be a generalized Brauer tree. A basic algebra $\Lambda$ is said to be a
 generalized Brauer tree algebra associated with $(\Tree,\m,\co)$ if there is an assignment $i\mapsto S_i$ from
 edges of $\Tree$ to simple $\Lambda$-modules satisfying the following conditions:
 \begin{enumerate}[{\rm (i)}]
 \item $S_i\ (i\in \Tree_1)$ gives a complete set of representatives of isomorphism classes of simple
 $\Lambda$-modules.
 \item Let $P_i$ be the projective cover of $S_i$. Then $\top P_i\simeq \soc P_i\simeq S_i$.
 \item If $u\overset{i}{-}v$ with $\co_u=(i,i^{(u)}_1,\cdots, i^{(u)}_r)$ and $\co_v=(i,i^{(v)}_1,\cdots, i^{(v)}_s)$,
  then there is a direct sum decomposition
 \[\rad P_i/\soc P_i\simeq U_i\oplus V_i\] 
 satisfying $U_i$ and $V_i$ are uniserial modules with 
 \[U_i\simeq {\scriptsize \begin{array}{c}
 S_{i^{(u)}_1}\\
 \vdots\\
 S_{i^{(u)}_r}\vspace{2pt}\\
 S_i\\
 \vdots\\
 S_{i^{(u)}_1}\\
 \vdots\\
 S_{i^{(u)}_r}\\
 \end{array},\ 
 V_i\simeq \begin{array}{c}
 S_{i^{(v)}_1}\\
 \vdots\\
 S_{i^{(v)}_s}\vspace{2pt}\\
 S_i\\
 \vdots\\
 S_{i^{(v)}_1}\\
 \vdots\\
 S_{i^{(v)}_s}\\
 \end{array}}, 
 \] 
 where $S_i$ appears $\m(u)-1$ (resp. $\m(v)-1$) times in $U_i$ (resp. $V_i$).    
 \end{enumerate} 
  \end{definition}
\begin{remark}
\label{rembta}
 Let $\Lambda=KQ/I$ be a generalized Brauer tree algebra associated with a generalized Brauer tree $(\Tree,\m,\co)$.
 We may assume that $Q_0=\Tree_1$ via the assignment $i\mapsto S_i$.
We write $i\stackrel{v}{\sim} j$ if $i$ and $j$ adjacent to $v$.   
By Definition\;\ref{bta}, we see the following statements:
\begin{enumerate}[{\rm (1)}]
\item Let $i\neq j\in Q_0$. Then there is an arrow from $i$ to $j$ if and only if
there exists a vertex $v$ of $\Tree$ such that $i\stackrel{v}{\sim}j$ and
$j=\co_v(i)$. Further more, since $\Tree$ is a tree, the number of arrows from $i$ to $j$ is at most one for each $(i,j)\in Q_0\times Q_0$. 
%$j$ is the direct successor of $i$ in $\co(v$.
\item Consider $i_1, i_2, i_3\in Q_0$ such that $i_2=\co_u(i_1)$ and 
$i_3=\co_v(i_2)$ for some $u,v\in \Tree_0$. Let $\alpha$ (resp. $\beta$)
be the arrow of $Q$ corresponding to $(i_1,i_2)$ (resp. $(i_2,i_3)$).
Then $\alpha\beta=0$ if $u\neq v$. 
\item $e_i \Lambda e_j\neq 0$ if and only if there exists a vertex $v$ of $\Tree$ such that $i\stackrel{v}{\sim}j$ 
(since $\Tree$ is a tree, $v$ is unique if exists).
Moreover, for cyclic ordering $\co_v=(i=i_0, i_1, \dots, i_t=j, \dots, i_r)$, it follows from (2) that
\[e_i \Lambda e_j=\sum_{\ell\in \Z_{\geq 0}}K(\alpha_1\alpha_2\cdots\alpha_r)^\ell\alpha_1\cdots\alpha_j,\]
where $\alpha_t$ is an arrow from $i_{t-1}$ to $i_t$ corresponding to $(i_{t-1},i_t)$.
%(Note that $\Tree$ is a tree.)
 In this case, we denote by $g^i_j$ the path $\alpha_1\cdots\alpha_t$ and obtain $G^i_j(\Lambda)=\{g^i_j\}$. 
%\item Let $i,j,k$ be exactly three elements in $Q_0=\Tree_1$ such that $i\stackrel{v}{\sim} j\stackrel{v}{\sim} k$. Then
%\[g_j^i=g^i_{k}g^{k}_j\Leftrightarrow k\in [i,j]_{\co_v}.\] 
\end{enumerate}
\end{remark}

% Remark\;\ref{rembta} implies that if $e_i\Lambda e_j\neq 0$ then
 % $G_j^i(\Lambda)=\{g^i_j\}\neq \emptyset$. % and $G(\Lambda)$ is closed under taking subpath. 

 It is well-known that a generalized Brauer tree algebra is a special biserial algebra.
 There is a nice description of indecomposable modules.
\begin{theorem}[{\cite{WW}}] 
\label{WW}
Let $\Lambda$ be a special biserial algebra.
\begin{enumerate}[{\rm (1)}]
\item Each indecomposable $\Lambda$-module is either a string module, a band module 
or a non-uniserial projective-injective modules.
\item Let $M$ be a string module and $P^{(1)}\stackrel{f}{\to} P^{(0)}\to M\to 0$
the minimal projective presentation. Then $P^{(1)}\stackrel{f}{\to} P^{(0)}$ has following
form$:$
\[%WinTpicVersion4.30a
{\unitlength 0.1in%
\begin{picture}( 16.5000, 36.7000)( 22.7000,-44.6000)%
% STR 2 0 3 0 Black White
% 4 2400 1300 2400 1400 2 0 0 0
% $P_{j_0}$
\put(24.0000,-14.0000){\makebox(0,0)[lb]{$P_{j_0}$}}%
% STR 2 0 3 0 Black White
% 4 2400 2240 2400 2340 2 0 0 0
% $P_{j_1}$
\put(24.0000,-23.4000){\makebox(0,0)[lb]{$P_{j_1}$}}%
% STR 2 0 3 0 Black White
% 4 3920 1780 3920 1880 2 0 0 0
% $P_{i_1}$
\put(39.2000,-18.8000){\makebox(0,0)[lb]{$P_{i_1}$}}%
% STR 2 0 3 0 Black White
% 4 2330 3480 2330 3580 2 0 0 0
% $P_{j_{m-1}}$
\put(23.3000,-35.8000){\makebox(0,0)[lb]{$P_{j_{m-1}}$}}%
% STR 2 0 3 0 Black White
% 4 2400 4440 2400 4540 2 0 0 0
% $P_{j_m}$
\put(24.0000,-45.4000){\makebox(0,0)[lb]{$P_{j_m}$}}%
% STR 2 0 3 0 Black White
% 4 3920 3980 3920 4080 2 0 0 0
% $P_{i_m}$
\put(39.2000,-40.8000){\makebox(0,0)[lb]{$P_{i_m}$}}%
% STR 2 0 3 0 Black White
% 4 3900 3060 3900 3160 2 0 0 0
% $P_{i_{m-1}}$
\put(39.0000,-31.6000){\makebox(0,0)[lb]{$P_{i_{m-1}}$}}%
% DOT 2 0 3 0 Black White
% 3 3320 2540 3320 2740 3320 2940
% 
\special{pn 4}%
\special{sh 1}%
\special{ar 3320 2540 8 8 0  6.28318530717959E+0000}%
\special{sh 1}%
\special{ar 3320 2740 8 8 0  6.28318530717959E+0000}%
\special{sh 1}%
\special{ar 3320 2940 8 8 0  6.28318530717959E+0000}%
% STR 2 0 3 0 Black White
% 4 3220 1510 3220 1610 2 0 0 0
% $f_{j_0}^{i_1}$
\put(32.2000,-16.1000){\makebox(0,0)[lb]{$f_{j_0}^{i_1}$}}%
% STR 2 0 3 0 Black White
% 4 3220 2040 3220 2140 2 0 0 0
% $f_{j_1}^{i_1}$
\put(32.2000,-21.4000){\makebox(0,0)[lb]{$f_{j_1}^{i_1}$}}%
% STR 2 0 3 0 Black White
% 4 2270 820 2270 920 2 0 0 0
% $(-1\mathrm{th})$
\put(22.7000,-9.2000){\makebox(0,0)[lb]{$(-1\mathrm{th})$}}%
% STR 2 0 3 0 Black White
% 4 3850 820 3850 920 2 0 0 0
% $(0\mathrm{th})$
\put(38.5000,-9.2000){\makebox(0,0)[lb]{$(0\mathrm{th})$}}%
% LINE 2 0 3 0 Black White
% 2 2700 1380 3080 1490
% 
\special{pn 8}%
\special{pa 2700 1380}%
\special{pa 3080 1490}%
\special{fp}%
% VECTOR 2 0 3 0 Black White
% 2 3600 1660 3860 1750
% 
\special{pn 8}%
\special{pa 3600 1660}%
\special{pa 3860 1750}%
\special{fp}%
\special{sh 1}%
\special{pa 3860 1750}%
\special{pa 3804 1709}%
\special{pa 3810 1733}%
\special{pa 3790 1747}%
\special{pa 3860 1750}%
\special{fp}%
% LINE 2 0 3 0 Black White
% 2 2720 2260 3100 2150
% 
\special{pn 8}%
\special{pa 2720 2260}%
\special{pa 3100 2150}%
\special{fp}%
% VECTOR 2 0 3 0 Black White
% 4 3600 1990 3860 1910 3790 1700 3790 1700
% 
\special{pn 8}%
\special{pa 3600 1990}%
\special{pa 3860 1910}%
\special{fp}%
\special{sh 1}%
\special{pa 3860 1910}%
\special{pa 3790 1910}%
\special{pa 3809 1926}%
\special{pa 3802 1949}%
\special{pa 3860 1910}%
\special{fp}%
\special{pa 3790 1700}%
\special{pa 3790 1700}%
\special{fp}%
% STR 2 0 3 0 Black White
% 4 3180 3760 3180 3860 2 0 0 0
% $f_{j_{m-1}}^{i_m}$
\put(31.8000,-38.6000){\makebox(0,0)[lb]{$f_{j_{m-1}}^{i_m}$}}%
% LINE 2 0 3 0 Black White
% 2 2700 3580 3080 3690
% 
\special{pn 8}%
\special{pa 2700 3580}%
\special{pa 3080 3690}%
\special{fp}%
% VECTOR 2 0 3 0 Black White
% 2 3600 3860 3860 3950
% 
\special{pn 8}%
\special{pa 3600 3860}%
\special{pa 3860 3950}%
\special{fp}%
\special{sh 1}%
\special{pa 3860 3950}%
\special{pa 3804 3909}%
\special{pa 3810 3933}%
\special{pa 3790 3947}%
\special{pa 3860 3950}%
\special{fp}%
% LINE 2 0 3 0 Black White
% 2 2720 4460 3100 4350
% 
\special{pn 8}%
\special{pa 2720 4460}%
\special{pa 3100 4350}%
\special{fp}%
% VECTOR 2 0 3 0 Black White
% 4 3600 4190 3860 4110 3790 3900 3790 3900
% 
\special{pn 8}%
\special{pa 3600 4190}%
\special{pa 3860 4110}%
\special{fp}%
\special{sh 1}%
\special{pa 3860 4110}%
\special{pa 3790 4110}%
\special{pa 3809 4126}%
\special{pa 3802 4149}%
\special{pa 3860 4110}%
\special{fp}%
\special{pa 3790 3900}%
\special{pa 3790 3900}%
\special{fp}%
% STR 2 0 3 0 Black White
% 4 3190 4290 3190 4390 2 0 0 0
% $f_{j_m}^{i_m}$
\put(31.9000,-43.9000){\makebox(0,0)[lb]{$f_{j_m}^{i_m}$}}%
% LINE 2 0 3 0 Black White
% 2 2720 3460 3100 3350
% 
\special{pn 8}%
\special{pa 2720 3460}%
\special{pa 3100 3350}%
\special{fp}%
% VECTOR 2 0 3 0 Black White
% 4 3600 3190 3860 3110 3790 2900 3790 2900
% 
\special{pn 8}%
\special{pa 3600 3190}%
\special{pa 3860 3110}%
\special{fp}%
\special{sh 1}%
\special{pa 3860 3110}%
\special{pa 3790 3110}%
\special{pa 3809 3126}%
\special{pa 3802 3149}%
\special{pa 3860 3110}%
\special{fp}%
\special{pa 3790 2900}%
\special{pa 3790 2900}%
\special{fp}%
% STR 2 0 3 0 Black White
% 4 3190 3280 3190 3380 2 0 0 0
% $f_{j_{m-1}}^{i_{m-1}}$
\put(31.9000,-33.8000){\makebox(0,0)[lb]{$f_{j_{m-1}}^{i_{m-1}}$}}%
\end{picture}}%
\vspace{5pt}\] 
where  $0\neq f_j^i\in e_i \Lambda e_j$ and $P_{j_0}$ and $P_{j_m}$ are possibly zero.
Moreover, if $i_t\stackrel{v}{\sim} j_{t-1}$ and $i_t\stackrel{u}{\sim} j_t$, then $v\neq u$.
Also, if $i_t\stackrel{u}{\sim} j_t$ and $i_{t+1}\stackrel{v}{\sim} j_t$, then $v\neq u$.
\item Each band module is $\tau$-stable.
\end{enumerate} 
\end{theorem}

Let $M$ be an indecomposable $\tau$-rigid module. Then Theorem\;\ref{WW} implies that
$M$ is a string module. 
Let $P^{(1)}\stackrel{f}{\to} P^{(2)}$ be as in Theorem\;\ref{WW} (2). 
Since $\Tree$ is a tree and $j_s\neq i_t$ for any $s,t$ (see Lemma\;\ref{factformpp}), 
it is easy to check that $j_s\neq j_{s'}$ ($s\neq s'$) and $i_t\neq i_{t'}$ ($t\neq t'$). 
In particular,  generalized Brauer Tree algebras are $\tau$-tilting finite.
Thus we can apply Corollary\;\ref{application} to generalized Brauer tree algebras.
\begin{proposition}
\label{btacase}
Let $\Lambda$ be a generalized Brauer tree algebra associated with $(\Tree,\m,\co)$. Then $\T(\Lambda)=\T'(\Lambda)$.
In particular, $\sttilt \Lambda$ does not depend on $\m$.
\end{proposition} 
%%%%%%%%%%%%%%%%%%%%%%%%%%%%%%%%%%%%%%%%%%%%%%%%%%%%%%%%%%%%%%%%%%%%%%%%%%
\subsection{Applications to preprojective algebras of type $A$}
\label{subsect:preproja}
Let $\Pi$ be the preprojective algebra of type $A_n$, i.e., $\Pi$ is 
given by the following quiver and relations:
\[\begin{xy}
(0,0)*[o]+{1}="A",(14,0)*[o]+{2}="B",(30,0)*[o]+{\cdots}="C",(46,0)*[o]+{n}="D",
\ar @<3pt> "A";"B"^{\alpha_1}
\ar @<3pt> "B";"A"^{\alpha_1^*}
\ar @<3pt> "B";"C"^{\alpha_2}
\ar @<3pt> "C";"B"^{\alpha_2^*}
\ar @<3pt> "C";"D"^{\alpha_{n-1}}
\ar @<3pt> "D";"C"^{\alpha_{n-1}^*}
\end{xy}
\]
\[
\alpha_1\alpha_1^*=0,\ \alpha_i^*\alpha_i=\alpha_{i+1}\alpha_{i+1}^*\;(1\leq i \leq n-2),
\ \alpha_{n-1}^*\alpha_{n-1}=0. 
\]

Then it is known that each indecomposable two-term silting object has
the following form (see \cite[Section\;6.1]{IRRT}, \cite[Lemma\;6.7]{K2} for example):
\[\begin{xy}
(0,0)*[o]+{\text{(-1th})}="A",(30,0)*[o]+{\text{(0th)}}="B",
(0,-10)*[o]+{P_{j_0}}="C",(30,-20)*[o]+{P_{i_1}}="D",
(0,-30)*[o]+{P_{j_1}}="E",(30,-40)*[o]+{P_{i_{m-1}}}="F",
(0,-50)*[o]+{P_{j_{m-1}}}="G",(30,-60)*[o]+{P_{i_m}}="H",
(0,-70)*[o]+{P_{j_{m}}}="I",(15,-35)*[o]+{\vdots}="K"
\ar "C";"D"
\ar "E";"D"
\ar "G";"F"
\ar "G";"H"
\ar "I";"H"
\end{xy}
\] 
where $0\leq j_0< i_0 <j_1<\cdots < i_{m-1}<j_{m-1}<i_m<j_m \leq n+1$
and $P_0=0=P_{n+1}$.
Hence we can apply Corollary\;\ref{application} to preprojective algebras of type $A$
and recover \cite[Theorem\;3.3]{K2}.
%%%%%%%%%%%%%%%%%%%%%%%%%%%%%%%%%%%%%%%%%%%%%%%%%%%%%%%%%%%%%%%%%%%%%%%%%%%%%%
%\subsection{Applications to $\tau$-tilting finite radical squer zero algebras}

%%%%%%%%%%%%%%%%%%%%%%%%%%%%%%%%%%%%%%%%%%%%%%%%%%%%%%%%%%%%%%%%%%%%%%%%%%%%%%%%%%%%%%%%%%%%%%%%%%%%%%%%%%%
%%%%%%%%%%%%%%%%%%%%%%%%%%%%%%%%%%%%%%%%%%%%%%%%%%%%%%%%%%%%%%%%%%%%%%%%%%%%%%%%%%%%%%%%%%%%%%%%%%%%%%%%%%% 
\section{Finite support $\tau$-tilting posets of $2$-point algebras}
\label{subsect4.1}
Let $\Lambda$ be a $\tau$-tilting finite algebra.
Take a full subquiver $u\in \U^+_2$
 of $\H(\sttilt \Lambda)$ and $X:=\kappa^+(u)\in \trigidp \Lambda$, i.e.
$u$ has the following form 
\[
\begin{xy}
(0,0)*[o]+{T}="A", (-10,10)*[o]+{T'}="B", (10,10)*[o]+{T''}="C",
\ar "B";"A"
\ar "C";"A"
\end{xy}
\] 
and $\add X=\add \widetilde{T} \cap \add \widetilde{T}'\cap \add \widetilde{T}''$, where $\widetilde{T}$ (resp. $\widetilde{T}',\widetilde{T}'')$ is
the $\tau$-tilting pair corresponding to $T$ (resp. $T', T''$).
 
Then it follows from Corollary\;\ref{preservingsupport} that
\[[T,T' \vee T'']\simeq \sttilt_X \Lambda.\]  
On the other hand, Theorem\;\ref{reduction theorem} implies that $\sttilt_X \Lambda$ is isomorphic to some $2$-regular finite support $\tau$-tilting poset $\P$. 
%Hence \[\P\simeq \sttilt_M \Lambda=[T,Z]=[T,T'\vee T'']\] follows from Lemma\;\ref{updown}. 
\begin{comment}

\[
\begin{xy}
(0,20)*[o]+{\Lambda}="A", (0,-60)*[o]+{0}="B",(-40,10)*{}="C", (40,10)*{}="D",
(0,-34)*[o]+{\scalebox{0.6}{$T$}}="E", (-5,-26)*[o]+{\scalebox{0.6}{$T'$}}="F", (5,-26)*[o]+{{\scalebox{0.6}{$T''$}}}="G",
(0,0)*[o]+{\scalebox{0.6}{$\displaystyle T' \vee T''$}}="H",(-5,-8)*[o]+{\circ }="I", (5,-8)*[o]+{\circ }="J",
(-5,-16)*{\vdots}="K",(5,-16)*{\vdots}="L",(-5,-19)*[o]+{}="M",(5,-19)*[o]+{}="N",
(-40,-50)*{}="O", (40,-50)*{}="P",
(10,-17)*{\simeq \P}="Q"
\ar @{-} "A";"C"
\ar @{-} "A";"D"
\ar @{-} "C";"O"
\ar @{-} "D";"P"
\ar @{-} "O";"B"
\ar @{-} "P";"B"
\ar "F";"E"
\ar "G";"E"
\ar "H";"I"
\ar "H";"J"
\ar "I";"K"
\ar "J";"L"
\ar "M";"F"
\ar "N";"G"

\end{xy}
\]
\end{comment}
It shows that any support $\tau$-tilting poset is a "union" of $2$-regular finite support $\tau$-tilting posets. Therefore to determine possible shapes of $2$-regular finite support $\tau$-tilting posets
is an interesting problem.

Let $\Lambda=KQ/I$ be a $\tau$-tilting finite algebra with $Q_0=\{1,2\}$.
Since $\sttilt \Lambda$ is connected, $2$-regular and a finite lattice, $\sttilt \Lambda$ is isomorphic to
$\P_{\ell,\ell'}$ for some $\ell,\ell'\in \Z_{\geq 1}$, where $\P_{\ell,\ell'}$ is a poset given by the following quiver:  
%\[\input{kase_7th_shapetworeg.tex}\]
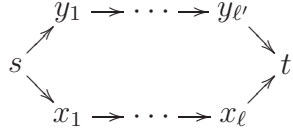
\[\begin{xy}
(0,0)*[o]+{s}="A", (7,7)*[o]+{y_1}="B", (18,7)*[o]+{\cdots}="C", (29,7)*[o]+{y_{\ell'}}="D",
 (7,-7)*[o]+{x_1}="E", (18,-7)*[o]+{\cdots}="F", (29,-7)*[o]+{x_{\ell}}="G", (36,0)*[o]+{t}="H",
\ar "A";"B"
\ar "B";"C"
\ar "C";"D"
\ar "D";"H"
\ar "A";"E"
\ar "E";"F"
\ar "F";"G"
\ar "G";"H"
\end{xy}\]
Conversely, each $\P_{\ell,\ell'}$ is realized as a support $\tau$-tilting poset.
\begin{proposition}
\label{2pointex}
Let $Q^{(\ell,\ell')}$ be a finite quiver with two vertices $1,2$ and 
\[Q^{(\ell,\ell')}_1=\left\{\begin{array}{ll} 
\begin{array}{lll}
 \{a_0:1\to 2\} &\cup&  \{a_i:1\to 1\mid 1\leq i\leq \ell-2\}\\
& \cup & \{b_0:2\to 1\}  \\
& \cup & \{b_{i'}:2\to 2\mid 1\leq i'\leq \ell'-2\} \\
\end{array} &  \text{if\ $\ell,\ell'\geq 2$}  \\\\
  \{b_0:2\to 1\}\cup \{b_{i'}:2\to 2\mid 1\leq i'\leq \ell'-2\} &   
  \text{if\ $\ell=1,\ell'\geq 2$} \\\\
 \{a_0:1\to 2\}\cup \{a_i:1\to 1\mid 1\leq i\leq \ell-2\} & 
  \text{if\ $\ell\geq 2,\ell'=1$}\\\\
 \emptyset & 
  \text{if\ $\ell=\ell'=1$}\\\\
 \end{array}\right.
\]
$I^{(\ell,\ell')}$ denotes an admissible ideal of $KQ^{(\ell,\ell')}$ generated by
\[a_ia_j\ (i-j\neq 1),\ b_{i'}b_{j'}\ (i'-j' \neq 1),\ a_ib_{i'},\ b_{i'}a_{i}\ (\forall i,i').\]
We set $\Lambda^{(\ell,\ell')}=KQ^{(\ell,\ell')}/I^{(\ell,\ell')}$. Then we have a poset isomorphism 
\[\sttilt \Lambda^{(\ell,\ell')}\simeq \P_{\ell,\ell'}.\]
\end{proposition}
\begin{proof}
We set 
\[X^{(r)}:=e_1\Lambda/(\overset{r-1}{\underset{t=0}{\Sigma}}a_t\cdots a_0 \Lambda)\ \text{and\ }Y^{(s)}:=e_2\Lambda/(\overset{s-1}{\underset{t=0}{\Sigma}}b_t\cdots b_0 \Lambda)\]
for any $r\in \{0,1,\dots, \ell-1\}$ and 
$s\in \{0,1,\dots, \ell'-1\}$. Since $e_1\Lambda e_2 \Lambda$ (resp. $e_2 \Lambda e_1 \Lambda$) 
is spanned by $\{a_t\cdots a_0\mid 0\leq t\leq r-1\}$ (resp. $\{b_t\cdots b_0\mid 0\leq t\leq s-1\}$),
we have $X^{(\ell-1)}=X_1$ and  $Y^{(\ell'-1)}=X_2$. It is also easy to check 
\[\sum_{t=0}^{r-1}a_t\cdots a_0 \Lambda=\bigoplus_{t=0}^{r-1}a_t\cdots a_0 \Lambda\ \text{and}\ 
\sum_{t=0}^{s-1}b_t\cdots b_0 \Lambda=\bigoplus_{t=0}^{s-1}b_t\cdots b_0 \Lambda.\]
In particular, $X^{(r)}\not\simeq X^{(r')}$ (resp. $Y^{(s)}\not\simeq Y^{(s')}$) if $r\neq r'$ (resp. $s\neq s'$).    
Let $f_t:P_2\to P_1$ be the left multiplication by $a_t\cdots a_0$ and $g_t:P_1\to P_2$ the left multiplication by $b_t\cdots b_0$.
Then a minimal projective presentation of $X^{(r)}$ is given by 
\[d_X^{(r)}:=(f_t)_{t=0}^{r-1}:\bigoplus_{t=0}^{r-1}P_2^{(t)}=P_2^{\oplus r}\to P_1\]
and a minimal projective presentation of $Y^{(s)}$ is given by 
\[d_Y^{(s)}:=(g_t)_{t=0}^{s-1}:\bigoplus_{t=0}^{s-1}P_1^{(t)}=P_1^{\oplus s}\to P_2.\]
Thus we have $\mathbf{S}(X^{(r)})=[P_2^{\oplus r}\stackrel{d_X^{(r)}}{\longrightarrow} P_1]$ and 
$\mathbf{S}(Y^{(s)})=[P_1^{\oplus s}\stackrel{d_Y^{(s)}}{\longrightarrow} P_2]$.
One sees that if $r\leq r'$ (resp. $s\leq s'$), then we have
\[\Hom_{\Kb(\proj \Lambda)}(\mathbf{S}(X^{(r)}), \mathbf{S}(X^{(r')})[1])=0\ (\text{resp}.\;\Hom_{\Kb(\proj \Lambda)}(\mathbf{S}(Y^{(s)}), \mathbf{S}(Y^{(s')})[1])=0).\] 
 We show  $\Hom_{\Kb(\proj \Lambda)}(\mathbf{S}(X^{(r)}), \mathbf{S}(X^{(r-1)})[1])=0$.
Denote by $f_p^{(t)}$ the composition map $f_p\circ \pi_t:\bigoplus_{t=0}^{r-1}P_2^{(t)}=P_2^{\oplus r}\to P_1$,
where $\pi_t$ is the canonical surjection $\bigoplus_{t=0}^{r-1}P_2^{(t)}\surj P_2^{(t)}$.
We regard $f_p^{(t)}$ as a morphism in $\Hom_{\Kb(\proj \Lambda)}(\mathbf{S}(X^{(r)}), \mathbf{S}(X^{(r-1)})[1])$ by the natural way.
Then it is sufficient to check that $f_p^{(t)}=0$ in $\Kb(\proj \Lambda)$ for any $p\in \{0,1,\dots,\ell-1\}$ and 
$t\in \{0,1,\dots, r-1\}$. If $p\leq r-2$, then we can easily check $f_p^{(t)}=0$.
Therefore, we may assume that $p\geq r-1$. Assume either $t\leq r-2$ or $p\geq r$ holds and let $h:P_1\to P_1$
be a left multiplication by $a_pa_{p-1}\cdots a_{t+1}$. In this case, it is easy to check that
\[f_p^{(t)}=h\circ d_X^{(r)}.\]
Hence we obtain that $f_p^{(t)}=0$ in $\Kb(\proj \Lambda)$. 
We consider the remaining case i.e., $t=r-1$ and $p=r-1$. Let $h=\mathrm{id}_{P_1}$ and 
$h'=\overset{r-2}{\underset{t=0}{\sum}}\iota_t\circ \mathrm{id}_{P_2^{(t)}} \circ \pi_t:\bigoplus_{t=0}^{r-1}P_2^{(t)}\to 
\bigoplus_{t=0}^{r-2}P_2^{(t)}$,
where $\iota_t$ be the canonical inclusion $P_2^{(t)}\inj \bigoplus_{t=0}^{r-2}P_2^{(t)}=P_2^{\oplus r-1}$.
Then we have
\[f_p^{(t)}=f_{r-1}^{(r-1)}=h\circ d_X^{(r)}-d_X^{(r-1)}\circ h'.\]
In particular, $f_{r-1}^{(r-1)}=0$ in $\Kb(\proj \Lambda)$.

Now Theorem\;\ref{bijection} implies that there is a path
\[\Lambda=X^{(0)}\oplus Y^{(0)}\to X^{(1)}\oplus X^{(0)}\to X^{(2)}\oplus X^{(1)}\to \cdots X^{(\ell-1)}\oplus X^{(\ell-2)}
\to X^{(\ell-1)}=X_1\to 0\]
in $\H(\sttilt \Lambda^{(\ell,\ell')})$.
Similarly, we obtain a path 
\[\Lambda=Y^{(0)}\oplus X^{(0)}\to Y^{(1)}\oplus Y^{(0)}\to Y^{(2)}\oplus Y^{(1)}\to \cdots Y^{(\ell'-1)}\oplus Y^{(\ell'-2)}\to Y^{(\ell'-1)}=X_2\to 0\]
in $\H(\sttilt \Lambda^{(\ell,\ell')})$. Thus $\sttilt \Lambda^{(\ell,\ell')}\simeq \P_{\ell,\ell'}$.
\end{proof}
In the case $\ell,\ell'\leq 2$, \cite[Proposition\;3.2]{AK} gives a characterization of algebras $\Lambda$ satisfying 
$\sttilt \Lambda\simeq \P_{\ell,\ell'}(\Leftrightarrow \Lambda\in \T(\Lambda^{(\ell,\ell')}))$.
Furthermore, Question\;\ref{cj} holds true if $|\Lambda|=2$.

Proposition\;\ref{2pointex} says that each connected $2$-regular finite lattice
is realized as a support $\tau$-tilting poset.
However, we have the following result.
\begin{proposition}
For each $n>2$, there exists a connected $n$-regular finite lattice
which is not realized as a support $\tau$-tilting poset.
\end{proposition}
\begin{proof}
For two posets $(\P,\leq_{\P}),(\P',\leq_{\P'})$, we always
regard $\P\times \P'$ as a poset via the following partial order:
\[(a,a')\leq (b,b'):\Leftrightarrow a\leq_{\P} b,\ a'\leq_{\P'} b'.\] 

Let $\P$ be a poset given by the following quiver:
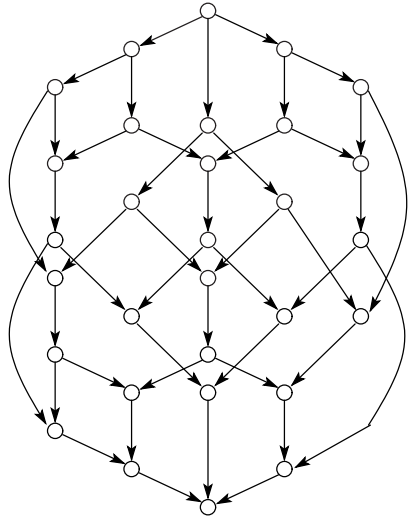
\[\input{kase_7th_counterex.tex}\]
We denote by $\B_m:=\underbrace{\{0<1\}\times \{0<1\}\times\cdots \{0<1\}}_{m}$.
Since $\P$ is a connected $3$-regular finite lattice and $\B_m$
is a connected $m$-regular finite lattice,
 $\P\times \B_m$ is a connected $(m+3)$-regular finite lattice.
 Hence it is sufficient to show that $\P\times \B_m$ is 
 not realized as a support $\tau$-tilting poset.
 
 Suppose that $\P\times \B_m\simeq \sttilt \Lambda$, where $\Lambda=KQ/I$.
 By Theorem\;\ref{reduction theorem}, we have
\[\sttilt_{(e_i+e_j) \Lambda} \Lambda\simeq \sttilt_{(e_i+e_j)\Lambda^-} \Lambda\]
for any $i\neq j\in Q_0$.
 From results in Subsection\;\ref{subsec:tau-rigid in poset}, we have
 \[\#\{\{i,j\}\mid i\neq j,\ \#\sttilt_{(e_i+e_j)\Lambda} \Lambda=6\}=3\neq 
 2=\#\{\{i,j\}\mid i\neq j,\ \#\sttilt_{(e_i+e_j)\Lambda^-} \Lambda=6\}.\] 
 This is a contradiction.
\end{proof}

%%%%%%%%%%%%%%%%%%%%%%%%%%%%%%%%%%%%%%%%%%%%%%%%%%%%%%%%%%%%%%%%%%%%%%%%%%%%%%%%%%%%%%%%%%%%%%%%%%%%
%%%%%%%%%%%%%%%%%%%%%%%%%%%%%%%%%%%%%%%%%%%%%%%%%%%%%%%%%%%%%%%%%%%%%%%%%%%%%%%%%%%%%%%%%%%%%%%%%%%%%%

\section{3-point algebras in $\Theta$}%having 3 isomorphism classes of simple modules}
  \label{subsect4.4}
%Assume that $\Lambda$ satisfies Condition\;\ref{cd} and $|\Lambda|=3$.
Let $\Theta_3:=\{\Lambda\in \Theta\mid \Lambda \text{ is connected with } |\Lambda|=3 \}$, $\overline{\Theta}_3:=\{\overline{\Lambda}\mid \Lambda\in \Theta_3\}$.
We denote by $\Theta'$ the set of (isomorphism classes of) basic connected algebras satisfying Condition\;\ref{cd} and
define $\Theta'_3:=\{\Lambda\in \Theta'\mid |\Lambda|=3\}$, $\overline{\Theta}'_3:=\{\overline{\Lambda}\mid \Lambda\in \Theta'_3\}$.
Then Figure\;\ref{fig:one} gives a complete list of algebras in $\overline{\Theta}'_3$. 

We can directly compute support $\tau$-tilting posets of algebras listed in Figure\;\ref{fig:one}. 
Such posets are available at
authors homepage (https://sites.google.com/site/ryoichikase/papers).
In particular, we have the following proposition.
\begin{proposition}
	\label{3simplecase}
	Each algebra in Figure\;\ref{fig:one} satisfies Assumption\;\ref{assumption}. In particular, 
	$\Theta_3=\Theta'_3$ and $\T(\Lambda)=\T'(\Lambda)$ holds for any $\Lambda\in \Theta_3$.
	Furthermore, for each $\Lambda\in \Theta_3$, we have
	\[ \# \sttilt \Lambda\in \{12,14,16,18,20,22,24,26,28,32\} \]
\end{proposition}  
\begin{figure}[h]
\[\scalebox{0.77}{$\begin{array}{llllll}
\begin{xy}
(0,0)*[o]+{1}="A",(13,0)*[o]+{2}="B",(26,0)*[o]+{3}="C",
(-7,0)*[0]={\Lambda_1:}="E"
\ar @<2pt> "A";"B"^{\alpha}
\ar @<2pt> "B";"A"^{\alpha^*}
\ar @<2pt> "B";"C"^{\beta}
\ar @<2pt> "C";"B"^{\beta^*}
\ar @(ld,rd) @<1pt>"C";"A"^{\gamma}
\ar @(ru,lu) @<1pt> "A";"C"^{\gamma^*}
\end{xy}
&
\scalebox{0.6}{$\begin{array}{l}
                        \alpha\alpha^*+\alpha^*\alpha=0\\
                        \beta\beta^*+\beta^*\beta=0\\
                        \gamma\gamma^*+\gamma^*\gamma=0\\
                        %\alpha\beta\gamma+\beta\gamma\alpha+\gamma\alpha\beta=0\\
                        %\gamma^*\beta^*\alpha^*+\beta^*\alpha^*\gamma^*+\alpha^*\gamma^*\beta^*=0\\
                        \alpha\beta+\beta\gamma+\gamma\beta=0\\
                        \gamma^*\beta^*+\beta^*\alpha^*+\alpha^*\gamma^*=0\\
                        \end{array}
                        $}
&\ \ 
\begin{xy}
(0,0)*[o]+{1}="A",(13,0)*[o]+{2}="B",(26,0)*[o]+{3}="C",
(-7,0)*[0]={\Lambda_2:}="E"
\ar @<2pt> "A";"B"^{\alpha}
\ar @<2pt> "B";"A"^{\alpha^*}
\ar @<2pt> "B";"C"^{\beta}
\ar @<2pt> "C";"B"^{\beta^*}
%\ar @(ld,rd) @<1pt>"C";"A"^{\gamma}
\ar @(ru,lu) @<1pt> "A";"C"^{\gamma^*}
\end{xy}
&
\scalebox{0.6}{$\begin{array}{l}
                        \alpha\alpha^*+\alpha^*\alpha=0\\
                        \beta\beta^*+\beta^*\beta=0\\
                        %\gamma\gamma^*+\gamma^*\gamma=0\\
                        %\alpha\beta\gamma+\beta\gamma\alpha+\gamma\alpha\beta=0\\
                        %\gamma^*\beta^*\alpha^*+\beta^*\alpha^*\gamma^*+\alpha^*\gamma^*\beta^*=0\\
                        \alpha\beta=0\\
                        \gamma^*\beta^*+\alpha^*\gamma^*=0\\
                        \end{array}
                        $}
&\ \ 
\begin{xy}
(0,0)*[o]+{1}="A",(13,0)*[o]+{2}="B",(26,0)*[o]+{3}="C",
(-7,0)*[0]={\Lambda_3:}="E"
\ar @<2pt> "A";"B"^{\alpha}
\ar @<2pt> "B";"A"^{\alpha^*}
\ar @<2pt> "B";"C"^{\beta}
\ar @<2pt> "C";"B"^{\beta^*}
%\ar @(ld,rd) @<1pt>"C";"A"^{\gamma}
\ar @(ru,lu) @<1pt> "A";"C"^{\gamma^*}
\end{xy}
&
\scalebox{0.6}{$\begin{array}{l}
                        \alpha\alpha^*+\alpha^*\alpha=0\\
                        \beta\beta^*+\beta^*\beta=0\\
                        %\gamma\gamma^*+\gamma^*\gamma=0\\
                        %\alpha\beta\gamma+\beta\gamma\alpha+\gamma\alpha\beta=0\\
                        %\gamma^*\beta^*\alpha^*+\beta^*\alpha^*\gamma^*+\alpha^*\gamma^*\beta^*=0\\
                        \alpha\beta=0\\
                        \gamma^*\beta^*+\beta^*\alpha^*+\alpha^*\gamma^*=0\\
                        \end{array}
                        $}\\
\begin{xy}
(0,0)*[o]+{1}="A",(13,0)*[o]+{2}="B",(26,0)*[o]+{3}="C",
(-7,0)*[0]={\Lambda_4:}="E"
\ar @<2pt> "A";"B"^{\alpha}
\ar @<2pt> "B";"A"^{\alpha^*}
%\ar @<2pt> "B";"C"^{\beta}
\ar @<2pt> "C";"B"^{\beta^*}
%\ar @(ld,rd) @<1pt>"C";"A"^{\gamma}
\ar @(ru,lu) @<1pt> "A";"C"^{\gamma^*}
\end{xy}
&
\scalebox{0.6}{$\begin{array}{l}
                        \alpha\alpha^*+\alpha^*\alpha=0\\
                        %\beta\beta^*+\beta^*\beta=0\\
                        %\gamma\gamma^*+\gamma^*\gamma=0\\
                        %\alpha\beta\gamma+\beta\gamma\alpha+\gamma\alpha\beta=0\\
                        %\gamma^*\beta^*\alpha^*+\beta^*\alpha^*\gamma^*+\alpha^*\gamma^*\beta^*=0\\
                        %\alpha\beta=0\\
                        \gamma^*\beta^*=0\\
                        \end{array}
                        $}
& \ \ 
\begin{xy}
(0,0)*[o]+{1}="A",(13,0)*[o]+{2}="B",(26,0)*[o]+{3}="C",
(-7,0)*[0]={\Lambda_5:}="E"
\ar @<2pt> "A";"B"^{\alpha}
\ar @<2pt> "B";"A"^{\alpha^*}
%\ar @<2pt> "B";"C"^{\beta}
\ar @<2pt> "C";"B"^{\beta^*}
%\ar @(ld,rd) @<1pt>"C";"A"^{\gamma}
\ar @(ru,lu) @<1pt> "A";"C"^{\gamma^*}
\end{xy}
&
\scalebox{0.6}{$\begin{array}{l}
                        \alpha\alpha^*+\alpha^*\alpha=0\\
                        %\beta\beta^*+\beta^*\beta=0\\
                        %\gamma\gamma^*+\gamma^*\gamma=0\\
                        %\alpha\beta\gamma+\beta\gamma\alpha+\gamma\alpha\beta=0\\
                        %\gamma^*\beta^*\alpha^*+\beta^*\alpha^*\gamma^*+\alpha^*\gamma^*\beta^*=0\\
                        %\alpha\beta=0\\
                        \beta^*\alpha^*=0\\
                        \end{array}
                        $}
&\ \ 
\begin{xy}
(0,0)*[o]+{1}="A",(13,0)*[o]+{2}="B",(26,0)*[o]+{3}="C",
(-7,0)*[0]={\Lambda_6:}="E"
\ar @<2pt> "A";"B"^{\alpha}
\ar @<2pt> "B";"A"^{\alpha^*}
%\ar @<2pt> "B";"C"^{\beta}
\ar @<2pt> "C";"B"^{\beta^*}
%\ar @(ld,rd) @<1pt>"C";"A"^{\gamma}
\ar @(ru,lu) @<1pt> "A";"C"^{\gamma^*}
\end{xy}
&
\scalebox{0.6}{$\begin{array}{l}
                        \alpha\alpha^*+\alpha^*\alpha=0\\
                        %\beta\beta^*+\beta^*\beta=0\\
                        %\gamma\gamma^*+\gamma^*\gamma=0\\
                        %\alpha\beta\gamma+\beta\gamma\alpha+\gamma\alpha\beta=0\\
                        %\gamma^*\beta^*\alpha^*+\beta^*\alpha^*\gamma^*+\alpha^*\gamma^*\beta^*=0\\
                        %\alpha\beta=0\\
                        %\gamma^*\beta^*=0\\
                        \alpha^*\gamma^*=0\\
                        \end{array}
                        $}\\
\begin{xy}
(0,0)*[o]+{1}="A",(13,0)*[o]+{2}="B",(26,0)*[o]+{3}="C",
(-7,0)*[0]={\Lambda_7:}="E"
\ar @<2pt> "A";"B"^{\alpha}
\ar @<2pt> "B";"A"^{\alpha^*}
%\ar @<2pt> "B";"C"^{\beta}
\ar @<2pt> "C";"B"^{\beta^*}
%\ar @(ld,rd) @<1pt>"C";"A"^{\gamma}
\ar @(ru,lu) @<1pt> "A";"C"^{\gamma^*}
\end{xy}
&
\scalebox{0.6}{$\begin{array}{l}
                        \alpha\alpha^*+\alpha^*\alpha=0\\
                        %\beta\beta^*+\beta^*\beta=0\\
                        %\gamma\gamma^*+\gamma^*\gamma=0\\
                        %\alpha\beta\gamma+\beta\gamma\alpha+\gamma\alpha\beta=0\\
                        %\gamma^*\beta^*\alpha^*+\beta^*\alpha^*\gamma^*+\alpha^*\gamma^*\beta^*=0\\
                        %\alpha\beta=0\\
                        \beta^*\alpha^*+\alpha^*\gamma^*=0\\
                        \end{array}
                        $}
&\ \ 
\begin{xy}
(0,0)*[o]+{1}="A",(13,0)*[o]+{2}="B",(26,0)*[o]+{3}="C",
(-7,0)*[0]={\Lambda_8:}="E"
\ar @<2pt> "A";"B"^{\alpha}
\ar @<2pt> "B";"A"^{\alpha^*}
\ar @<2pt> "B";"C"^{\beta}
%\ar @<2pt> "C";"B"^{\beta^*}
%\ar @(ld,rd) @<1pt>"C";"A"^{\gamma}
\ar @(ru,lu) @<1pt> "A";"C"^{\gamma^*}
\end{xy}
&
\scalebox{0.6}{$\begin{array}{l}
                        \alpha\alpha^*+\alpha^*\alpha=0\\
                        %\beta\beta^*+\beta^*\beta=0\\
                        %\gamma\gamma^*+\gamma^*\gamma=0\\
                        %\alpha\beta\gamma+\beta\gamma\alpha+\gamma\alpha\beta=0\\
                        %\gamma^*\beta^*\alpha^*+\beta^*\alpha^*\gamma^*+\alpha^*\gamma^*\beta^*=0\\
                        \alpha\beta=0\\
                        %\gamma^*\beta^*=0\\
                        \alpha^*\gamma^*=0\\
                        \end{array}
                        $}
& \ \ 
\begin{xy}
(0,0)*[o]+{1}="A",(13,0)*[o]+{2}="B",(26,0)*[o]+{3}="C",
(-7,0)*[0]={\Lambda_9:}="E"
\ar @<2pt> "A";"B"^{\alpha}
\ar @<2pt> "B";"A"^{\alpha^*}
%\ar @<2pt> "B";"C"^{\beta}
\ar @<2pt> "C";"B"^{\beta^*}
\ar @(ld,rd) @<1pt>"C";"A"^{\gamma}
%\ar @(ru,lu) @<1pt> "A";"C"^{\gamma^*}
\end{xy}
&
\scalebox{0.6}{$\begin{array}{l}
                        \alpha\alpha^*+\alpha^*\alpha=0\\
                        %\beta\beta^*+\beta^*\beta=0\\
                        %\gamma\gamma^*+\gamma^*\gamma=0\\
                        %\alpha\beta\gamma+\beta\gamma\alpha+\gamma\alpha\beta=0\\
                        %\gamma^*\beta^*\alpha^*+\beta^*\alpha^*\gamma^*+\alpha^*\gamma^*\beta^*=0\\
                        \gamma\alpha=0\\
                        \beta^*\alpha^*=0\\
                        \end{array}
                        $}\\ 
\begin{xy}
(0,0)*[o]+{1}="A",(13,0)*[o]+{2}="B",(26,0)*[o]+{3}="C",
(-7,0)*[0]={\Lambda_{10}:}="E"
%\ar @<2pt> "A";"B"^{\alpha}
\ar @<2pt> "B";"A"^{\alpha^*}
%\ar @<2pt> "B";"C"^{\beta}
\ar @<2pt> "C";"B"^{\beta^*}
%\ar @(ld,rd) @<1pt>"C";"A"^{\gamma}
\ar @(ru,lu) @<1pt> "A";"C"^{\gamma^*}
\end{xy}
&
\scalebox{0.6}{$\begin{array}{l}
                        %\alpha\alpha^*+\alpha^*\alpha=0\\
                        %\beta\beta^*+\beta^*\beta=0\\
                        %\gamma\gamma^*+\gamma^*\gamma=0\\
                        %\alpha\beta\gamma+\beta\gamma\alpha+\gamma\alpha\beta=0\\
                        \gamma^*\beta^*\alpha^*+\beta^*\alpha^*\gamma^*+\alpha^*\gamma^*\beta^*=0\\
                        %\alpha\beta=0\\
                        %\gamma^*\beta^*=0\\
                        %\alpha^*\gamma^*=0\\
                        \end{array}
                        $}
&\ \ 
\begin{xy}
(0,0)*[o]+{1}="A",(13,0)*[o]+{2}="B",(26,0)*[o]+{3}="C",
(-7,0)*[0]={\Lambda_{11}:}="E"
%\ar @<2pt> "A";"B"^{\alpha}
\ar @<2pt> "B";"A"^{\alpha^*}
%\ar @<2pt> "B";"C"^{\beta}
\ar @<2pt> "C";"B"^{\beta^*}
%\ar @(ld,rd) @<1pt>"C";"A"^{\gamma}
\ar @(ru,lu) @<1pt> "A";"C"^{\gamma^*}
\end{xy}
&
\scalebox{0.6}{$\begin{array}{l}
                        %\alpha\alpha^*+\alpha^*\alpha=0\\
                        %\beta\beta^*+\beta^*\beta=0\\
                        %\gamma\gamma^*+\gamma^*\gamma=0\\
                        %\alpha\beta\gamma+\beta\gamma\alpha+\gamma\alpha\beta=0\\
                        \gamma^*\beta^*=0\\
                        \beta^*\alpha^*\gamma^*=0\\
                        %\gamma\alpha=0\\
                        %\beta^*\alpha^*=0\\
                        \end{array}
                        $}
&\ \ 
\begin{xy}
(0,0)*[o]+{1}="A",(13,0)*[o]+{2}="B",(26,0)*[o]+{3}="C",
(-7,0)*[0]={\Lambda_{12}:}="E"
%\ar @<2pt> "A";"B"^{\alpha}
\ar @<2pt> "B";"A"^{\alpha^*}
%\ar @<2pt> "B";"C"^{\beta}
\ar @<2pt> "C";"B"^{\beta^*}
%\ar @(ld,rd) @<1pt>"C";"A"^{\gamma}
\ar @(ru,lu) @<1pt> "A";"C"^{\gamma^*}
\end{xy}
&
\scalebox{0.6}{$\begin{array}{l}
                        %\alpha\alpha^*+\alpha^*\alpha=0\\
                        %\beta\beta^*+\beta^*\beta=0\\
                        %\gamma\gamma^*+\gamma^*\gamma=0\\
                        %\alpha\beta\gamma+\beta\gamma\alpha+\gamma\alpha\beta=0\\
                        \beta^*\alpha^*+\gamma^*\beta^*=0\\
                        %\alpha\beta=0\\
                        %\gamma^*\beta^*=0\\
                        %\alpha^*\gamma^*=0\\
                        \end{array}
                        $}\\
\begin{xy}
(0,0)*[o]+{1}="A",(13,0)*[o]+{2}="B",(26,0)*[o]+{3}="C",
(-7,0)*[0]={\Lambda_{13}:}="E"
%\ar @<2pt> "A";"B"^{\alpha}
\ar @<2pt> "B";"A"^{\alpha^*}
%\ar @<2pt> "B";"C"^{\beta}
\ar @<2pt> "C";"B"^{\beta^*}
%\ar @(ld,rd) @<1pt>"C";"A"^{\gamma}
\ar @(ru,lu) @<1pt> "A";"C"^{\gamma^*}
\end{xy}
&
\scalebox{0.6}{$\begin{array}{l}
                        %\alpha\alpha^*+\alpha^*\alpha=0\\
                        %\beta\beta^*+\beta^*\beta=0\\
                        %\gamma\gamma^*+\gamma^*\gamma=0\\
                        %\alpha\beta\gamma+\beta\gamma\alpha+\gamma\alpha\beta=0\\
                        \gamma^*\beta^*+\beta^*\alpha^*+\alpha^*\gamma^*=0\\
                        %\gamma\alpha=0\\
                        %\beta^*\alpha^*=0\\
                        \end{array}
                        $}
&\ \ 
\begin{xy}
(0,0)*[o]+{1}="A",(13,0)*[o]+{2}="B",(26,0)*[o]+{3}="C",
(-7,0)*[0]={\Lambda_{14}:}="E"
\ar @<2pt> "A";"B"^{\alpha}
\ar @<2pt> "B";"A"^{\alpha^*}
\ar @<2pt> "B";"C"^{\beta}
\ar @<2pt> "C";"B"^{\beta^*}
%\ar @(ld,rd) @<1pt>"C";"A"^{\gamma}
%\ar @(ru,lu) @<1pt> "A";"C"^{\gamma^*}
\end{xy}
&
\scalebox{0.6}{$\begin{array}{l}
                        \alpha\alpha^*+\alpha^*\alpha=0\\
                        \beta\beta^*+\beta^*\beta=0\\
                        %\gamma\gamma^*+\gamma^*\gamma=0\\
                        %\alpha\beta\gamma+\beta\gamma\alpha+\gamma\alpha\beta=0\\
                        %\beta^*\alpha^*+\gamma^*\beta^*=0\\
                        %\alpha\beta=0\\
                        %\gamma^*\beta^*=0\\
                        %\alpha^*\gamma^*=0\\
                        \end{array}
                        $}
& \ \ 
\begin{xy}
(0,0)*[o]+{1}="A",(13,0)*[o]+{2}="B",(26,0)*[o]+{3}="C",
(-7,0)*[0]={\Lambda_{15}:}="E"
\ar @<2pt> "A";"B"^{\alpha}
\ar @<2pt> "B";"A"^{\alpha^*}
\ar @<2pt> "B";"C"^{\beta}
\ar @<2pt> "C";"B"^{\beta^*}
%\ar @(ld,rd) @<1pt>"C";"A"^{\gamma}
%\ar @(ru,lu) @<1pt> "A";"C"^{\gamma^*}
\end{xy}
&
\scalebox{0.6}{$\begin{array}{l}
                        \alpha\alpha^*+\alpha^*\alpha=0\\
                        \beta\beta^*+\beta^*\beta=0\\
                        %\gamma\gamma^*+\gamma^*\gamma=0\\
                        \alpha\beta=0\\
                        %\gamma^*\beta^*+\beta^*\alpha^*+\alpha^*\gamma^*=0\\
                        %\gamma\alpha=0\\
                        %\beta^*\alpha^*=0\\
                        \end{array}
                        $}\\
\begin{xy}
(0,0)*[o]+{1}="A",(13,0)*[o]+{2}="B",(26,0)*[o]+{3}="C",
(-7,0)*[0]={\Lambda_{16}:}="E"
\ar @<2pt> "A";"B"^{\alpha}
\ar @<2pt> "B";"A"^{\alpha^*}
\ar @<2pt> "B";"C"^{\beta}
\ar @<2pt> "C";"B"^{\beta^*}
%\ar @(ld,rd) @<1pt>"C";"A"^{\gamma}
%\ar @(ru,lu) @<1pt> "A";"C"^{\gamma^*}
\end{xy}
&
\scalebox{0.6}{$\begin{array}{l}
                        \alpha\alpha^*+\alpha^*\alpha=0\\
                        \beta\beta^*+\beta^*\beta=0\\
                        %\gamma\gamma^*+\gamma^*\gamma=0\\
                        \alpha\beta=0\\
                        \beta^*\alpha^*=0\\
                        %\alpha\beta=0\\
                        %\gamma^*\beta^*=0\\
                        %\alpha^*\gamma^*=0\\
                        \end{array}
                        $}
&\ \ 
\begin{xy}
(0,0)*[o]+{1}="A",(13,0)*[o]+{2}="B",(26,0)*[o]+{3}="C",
(-7,0)*[0]={\Lambda_{17}:}="E"
\ar @<2pt> "A";"B"^{\alpha}
\ar @<2pt> "B";"A"^{\alpha^*}
\ar @<2pt> "B";"C"^{\beta}
%\ar @<2pt> "C";"B"^{\beta^*}
%\ar @(ld,rd) @<1pt>"C";"A"^{\gamma}
%\ar @(ru,lu) @<1pt> "A";"C"^{\gamma^*}
\end{xy}
&
\scalebox{0.6}{$\begin{array}{l}
                        \alpha\alpha^*+\alpha^*\alpha=0\\
                        %\beta\beta^*+\beta^*\beta=0\\
                        %\gamma\gamma^*+\gamma^*\gamma=0\\
                        %\alpha\beta=0\\
                        %\gamma^*\beta^*+\beta^*\alpha^*+\alpha^*\gamma^*=0\\
                        %\gamma\alpha=0\\
                        %\beta^*\alpha^*=0\\
                        \end{array}
                        $}
&\ \ 
\begin{xy}
(0,0)*[o]+{1}="A",(13,0)*[o]+{2}="B",(26,0)*[o]+{3}="C",
(-7,0)*[0]={\Lambda_{18}:}="E"
\ar @<2pt> "A";"B"^{\alpha}
\ar @<2pt> "B";"A"^{\alpha^*}
\ar @<2pt> "B";"C"^{\beta}
%\ar @<2pt> "C";"B"^{\beta^*}
%\ar @(ld,rd) @<1pt>"C";"A"^{\gamma}
%\ar @(ru,lu) @<1pt> "A";"C"^{\gamma^*}
\end{xy}
&
\scalebox{0.6}{$\begin{array}{l}
                        \alpha\alpha^*+\alpha^*\alpha=0\\
                        %\beta\beta^*+\beta^*\beta=0\\
                        %\gamma\gamma^*+\gamma^*\gamma=0\\
                        \alpha\beta=0\\
                        %\beta^*\alpha^*=0\\
                        %\alpha\beta=0\\
                        %\gamma^*\beta^*=0\\
                        %\alpha^*\gamma^*=0\\
                        \end{array}
                        $}\\
\begin{xy}
(0,0)*[o]+{1}="A",(13,0)*[o]+{2}="B",(26,0)*[o]+{3}="C",
(-7,0)*[0]={\Lambda_{19}:}="E"
\ar @<2pt> "A";"B"^{\alpha}
\ar @<2pt> "B";"A"^{\alpha^*}
%\ar @<2pt> "B";"C"^{\beta}
\ar @<2pt> "C";"B"^{\beta^*}
%\ar @(ld,rd) @<1pt>"C";"A"^{\gamma}
%\ar @(ru,lu) @<1pt> "A";"C"^{\gamma^*}
\end{xy}
&
\scalebox{0.6}{$\begin{array}{l}
                        \alpha\alpha^*+\alpha^*\alpha=0\\
                        %\beta\beta^*+\beta^*\beta=0\\
                        %\gamma\gamma^*+\gamma^*\gamma=0\\
                        %\alpha\beta=0\\
                        %\gamma^*\beta^*+\beta^*\alpha^*+\alpha^*\gamma^*=0\\
                        %\gamma\alpha=0\\
                        %\beta^*\alpha^*=0\\
                        \end{array}
                        $}
&\ \ 
\begin{xy}
(0,0)*[o]+{1}="A",(13,0)*[o]+{2}="B",(26,0)*[o]+{3}="C",
(-7,0)*[0]={\Lambda_{20}:}="E"
\ar @<2pt> "A";"B"^{\alpha}
\ar @<2pt> "B";"A"^{\alpha^*}
%\ar @<2pt> "B";"C"^{\beta}
\ar @<2pt> "C";"B"^{\beta^*}
%\ar @(ld,rd) @<1pt>"C";"A"^{\gamma}
%\ar @(ru,lu) @<1pt> "A";"C"^{\gamma^*}
\end{xy}
&
\scalebox{0.6}{$\begin{array}{l}
                        \alpha\alpha^*+\alpha^*\alpha=0\\
                        %\beta\beta^*+\beta^*\beta=0\\
                        %\gamma\gamma^*+\gamma^*\gamma=0\\
                        %\alpha\beta=0\\
                        \beta^*\alpha^*=0\\
                        %\alpha\beta=0\\
                        %\gamma^*\beta^*=0\\
                        %\alpha^*\gamma^*=0\\
                        \end{array}
                        $}
&\ \ 
\begin{xy}
(0,0)*[o]+{1}="A",(13,0)*[o]+{2}="B",(26,0)*[o]+{3}="C",
(-7,0)*[0]={\Lambda_{21}:}="E"
\ar @<2pt> "A";"B"^{\alpha}
%\ar @<2pt> "B";"A"^{\alpha^*}
\ar @<2pt> "B";"C"^{\beta}
%\ar @<2pt> "C";"B"^{\beta^*}
%\ar @(ld,rd) @<1pt>"C";"A"^{\gamma}
\ar @(ru,lu) @<1pt> "A";"C"^{\gamma^*}
\end{xy}
&
\scalebox{0.6}{$\begin{array}{l}
                        %\alpha\alpha^*+\alpha^*\alpha=0\\
                        %\beta\beta^*+\beta^*\beta=0\\
                        %\gamma\gamma^*+\gamma^*\gamma=0\\
                        \alpha\beta=0\\
                        %\gamma^*\beta^*+\beta^*\alpha^*+\alpha^*\gamma^*=0\\
                        %\gamma\alpha=0\\
                        %\beta^*\alpha^*=0\\
                        \end{array}
                        $}\\ 
\begin{xy}
(0,0)*[o]+{1}="A",(13,0)*[o]+{2}="B",(26,0)*[o]+{3}="C",
(-7,0)*[0]={\Lambda_{22}:}="E"
\ar @<2pt> "A";"B"^{\alpha}
%\ar @<2pt> "B";"A"^{\alpha^*}
\ar @<2pt> "B";"C"^{\beta}
%\ar @<2pt> "C";"B"^{\beta^*}
%\ar @(ld,rd) @<1pt>"C";"A"^{\gamma}
%\ar @(ru,lu) @<1pt> "A";"C"^{\gamma^*}
\end{xy}
&
%\scalebox{0.6}{$\begin{array}{l}
                        %\alpha\alpha^*+\alpha^*\alpha=0\\
                        %\beta\beta^*+\beta^*\beta=0\\
                        %\gamma\gamma^*+\gamma^*\gamma=0\\
                        %\alpha\beta=0\\
                        %\beta^*\alpha^*=0\\
                        %\alpha\beta=0\\
                        %\gamma^*\beta^*=0\\
                        %\alpha^*\gamma^*=0\\
                        %\end{array}
                        % $}
&\ \ 
\begin{xy}
(0,0)*[o]+{1}="A",(13,0)*[o]+{2}="B",(26,0)*[o]+{3}="C",
(-7,0)*[0]={\Lambda_{23}:}="E"
\ar @<2pt> "A";"B"^{\alpha}
%\ar @<2pt> "B";"A"^{\alpha^*}
\ar @<2pt> "B";"C"^{\beta}
%\ar @<2pt> "C";"B"^{\beta^*}
%\ar @(ld,rd) @<1pt>"C";"A"^{\gamma}
%\ar @(ru,lu) @<1pt> "A";"C"^{\gamma^*}
\end{xy}
&
\scalebox{0.6}{$\begin{array}{l}
                        %\alpha\alpha^*+\alpha^*\alpha=0\\
                        %\beta\beta^*+\beta^*\beta=0\\
                        %\gamma\gamma^*+\gamma^*\gamma=0\\
                        \alpha\beta=0\\
                        %\beta^*\alpha^*=0\\
                        %\alpha\beta=0\\
                        %\gamma^*\beta^*=0\\
                        %\alpha^*\gamma^*=0\\
                        \end{array}
                        $}
&\ \ 
\begin{xy}
(0,0)*[o]+{1}="A",(13,0)*[o]+{2}="B",(26,0)*[o]+{3}="C",
(-7,0)*[0]={\Lambda_{24}:}="E"
\ar @<2pt> "A";"B"^{\alpha}
%\ar @<2pt> "B";"A"^{\alpha^*}
%\ar @<2pt> "B";"C"^{\beta}
\ar @<2pt> "C";"B"^{\beta^*}
%\ar @(ld,rd) @<1pt>"C";"A"^{\gamma}
%\ar @(ru,lu) @<1pt> "A";"C"^{\gamma^*}
\end{xy}
& \\%\scalebox{0.6}{$\begin{array}{l}
                        %\alpha\alpha^*+\alpha^*\alpha=0\\
                        %\beta\beta^*+\beta^*\beta=0\\
                        %\gamma\gamma^*+\gamma^*\gamma=0\\
                        %\alpha\beta=0\\
                        %\beta^*\alpha^*=0\\
                        %\alpha\beta=0\\
                        %\gamma^*\beta^*=0\\
                        %\alpha^*\gamma^*=0\\
                        %\end{array}
                        % $}
\begin{xy}
(0,0)*[o]+{1}="A",(13,0)*[o]+{2}="B",(26,0)*[o]+{3}="C",
(-7,0)*[0]={\Lambda_{25}:}="E"
%\ar @<2pt> "A";"B"^{\alpha}
\ar @<2pt> "B";"A"^{\alpha^*}
\ar @<2pt> "B";"C"^{\beta}
%\ar @<2pt> "C";"B"^{\beta^*}
%\ar @(ld,rd) @<1pt>"C";"A"^{\gamma}
%\ar @(ru,lu) @<1pt> "A";"C"^{\gamma^*}
\end{xy}
\end{array}$}
\]
\caption{ The list of $\overline{\Theta}'_3$. }
\label{fig:one}
\end{figure}
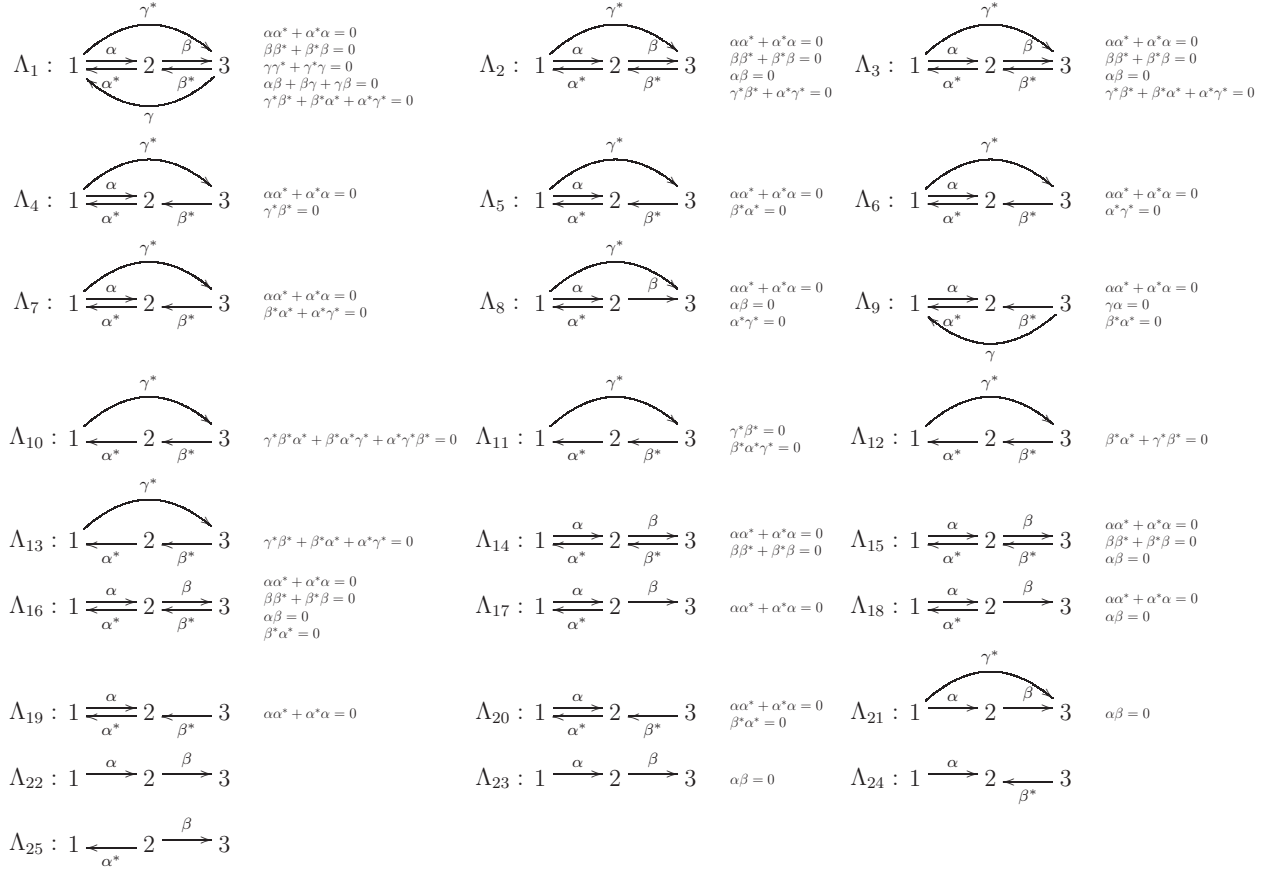

%%%%%%%%%%%%%%%%%%%%%%%%%%%%%%%%%%%%%%%%%%%%%%%%%%%%%%%%%%%%%%%%%%%%%%%%%%%%%%%%%%%%%%%%%%%%%%%%%%%%%%%%%%%%%%
%%%%%%%%%%%%%%%%%%%%%%%%%%%%%%%%%%%%%%%%%%%%%%%%%%%%%%%%%%%%%%%%%%%%%%%%%%%%%%%%%%%%%%%%%%%%%%%%%%%%%%%%%%%%%%


\begin{thebibliography}{9}
\bibitem{A}
{\sc T. Adachi}, 
The classification of $\tau$-tilting modules over  Nakayama algebras,
J. Algebra. \textbf{452} (2016), 227--262. 
\bibitem{AAC}
{\sc T. Adachi, T. Aihara and A. Chan},
Classification of two-term tilting complexes over Brauer graph algebras,
arXiv e-prints (2015), http://arxiv.org/abs/1504.04827.
\bibitem{AIR}
{\sc T. Adachi, O. Iyama and I. Reiten},
$\tau$-tilting theory,
 Compos. Math. \textbf{150}, no. 3 (2014), 415--452.

\bibitem{AI}
{\sc T. Aihara and O. Iyama},
Silting mutation in triangulated categories,
 J. Lond. Math. Soc. (2) \textbf{85} (2012), no. 3, 633--668. 
\bibitem{AK}
{\sc T. Aihara and R. Kase},
Algebras sharing the same support $\tau$-tilting poset with tree quiver algebras,
arXiv e-prints (2016), https://arxiv.org/abs/1609.01880.

\bibitem{Asa}
{\sc S. Asai},
Semibricks, arXiv e-prints (2016), https://arxiv.org/abs/1610.05860.

\bibitem{ASS}
{\sc I. Assem, D. Simson and A. Skowro\'{n}ski},
 Elements of the representation theory of associative algebras. Vol.~\textbf{1},
  London Mathematical Society Student Texts \textbf{65}, Cambridge University Press (2006).
  
\bibitem{ARS}
{\sc M. Auslander, I. Reiten and S. Smal\o},
 Representation theory of artin algebras,
Cambridge studies in advanced mathematics \textbf{36}, Cambridge University Press (1995).

\bibitem{AS}
{\sc M. Auslander and S. Smal\o}, Preprojective modules over Artin algebras,
 J. Algebra \textbf{66} (1980), no. 1, 61--122.

%\bibitem{BjB}
%\textit{A. Bj\"{o}rner} and \textit{F. Brenti},
%Combinatorics of Coxeter groups,
%Graduate Texts in Mathematics, vol. \textbf{231}, Springer, New York, 2005.

%\bibitem{BrB}
%{\sc S. Brenner and M.C.R. Butler},
% Generalizations of the Bernstein-Gelfand-Ponomarev reflection functors,
%in: Representation theory, II (Proc. Second Internat. Conf., Carleton Univ., Ottawa, Ont., 1979), 
%103--169,
%  Lecture Notes in Math., \textbf{832}, Springer, Berlin-New York (1980).

%\bibitem{BIRS}
%{\sc A. B. Buan, O. Iyama, I. Reiten and J. Scott},
%Cluster structures for 2-Calabi-Yau categories
%and unipotent groups.
% Compos. Math. \textbf{ 145}, no. 4 (2009), 1035--1079.

\bibitem{DIJ}
{\sc L. Demonet, O. Iyama, G, Jasso},
$\tau$-tilting finite algebras, bricks and $g$-vectors,
arXiv e-prints (2015), https://arxiv.org/abs/1503.00285v6.
%Lattice structure of Weyl groups via representation theory of preprojective algebras,
%arXiv preprint 2016, https://arxiv.org/abs/1604.08401.

\bibitem{EJR}
{\sc F. Eisele, G. Janssens and T. Raedschelders},
A reduction theorem for $\tau$-rigid modules,
arXiv e-prints (2016), http://arxiv.org/abs/1603.04293.

%\bibitem[GLS]{GLS}
%{\sc C. Geiss, B. Leclerc, J. Schr$\ddot{\mathrm{o}}$er},
%Quivers with relations for symmetrizable Cartan matrices I: Foundations.
%arXiv:1410.1403v2.

%\bibitem{H} 
%{\sc D. Happel}, 
%Triangulated categories in the representation theory of finite-dimensional algebras,
%London Mathematical Society Lecture Note Series, \textbf{119},
%Cambridge University Press, Cambridge (1988).

\bibitem{HU1}
{\sc D. Happel and L. Unger},
On the quiver of tilting modules,
J. Algebra. \textbf{ 284}, no. 2 (2005), 857--868.

\bibitem{HU2}
{\sc D. Happel and L. Unger},
Reconstruction of path algebras from their posets of tilting modules.
 Trans. Amer. Math. Soc {\bf 361}, no.7, 3633-3660 (2009).

\bibitem{IRRT}
{\sc O. Iyama, N. Reading, I. Reiten and H. Thomas},
Lattice structure of Weyl groups via representation theory of preprojective algebras,
arXiv preprint 2016, https://arxiv.org/abs/1604.08401.

\bibitem{IRTT}
{\sc O. Iyama, I. Reiten, H. Thomas, G. Todorov},
Lattice structure of torsion classes for path algebras,
Bull. Lond. Math. Soc. \textbf{47} (2015), no. 4, 639--650. 
%\bibitem{IZ}
%{\sc O. Iyama and X. Zhang},
 %Classifying $\tau$-tilting modules over the Auslander algebras of $K[X]/(X^n)$,
%arXiv preprint 2016, http://arxiv.org/abs/1602.05037.

\bibitem{J}
{\sc G. Jasso},
Reduction of $\tau$-tilting modules and torsion pairs,
 Int. Math. Res. Not. IMRN 2015, no. 16, 7190--7273. 
 
 \bibitem{K1}
  {\sc R. Kase},
  Taking tilting modules from the poset of support tilting modules,
  Math. Z. \textbf{280}, no. 3-4 (2015), 893--904.
\bibitem{K2}
{\sc R. Kase},
Weak orders on symmetric groups and the poset of support $\tau$-tilting modules,
 Int. J. Algebra Comput. \textbf{27} (2017), no. 5.

\bibitem{KY}
{\sc S. Koenig, D. Yang}, 
Silting objects, simple-minded collections, t-structures and co-t-structures for finite dimensional
algebras, Doc. Math. \textbf{19} (2014), 403--438.

\bibitem{MRZ}
{\sc R. Marsh, M. Reineke and A. Zelevinsky},
Generalized associahedra via quiver representations,
Trans. Amer. Math. Soc. \textbf{355} (2003), no. 10, 4171--4186. 
\bibitem{M}
{\sc Y. Mizuno},
Classifying $\tau$-tilting modules over preprojective algebras of Dynkin type,
 Math. Z. \textbf{ 277}, no. 3-4 (2014), 665--690.


%\bibitem{RS}
%{\sc C. Riedtmann and A. Schofield},
%On a simplicial complex associated with tilting modules,
%Comment. Math. Helv. \textbf{ 66}, no. 1 (1991), 70--78.

\bibitem{WW}
{\sc J. Wald and Waschb\"usch}, 
Tame biserial algebras, 
J. Algebra \textbf{95} , no. 2(1985), 480--500.
\end{thebibliography}
\end{document}